\documentclass{article}
\pdfoutput=1

\usepackage{amsmath}
\usepackage{amsfonts, mathtools}
\usepackage{amssymb}
\usepackage{amsthm}
\usepackage{tikz-cd}
\usepackage{tikz}
\usepackage{amsxtra}
\usepackage{wrapfig}
\usepackage{color}
\usepackage{geometry}
\usepackage{url}
\usepackage[hidelinks]{hyperref}
\usepackage[percent]{overpic}
\usepackage{mathrsfs}
\usepackage{stmaryrd}

\usepackage{setspace}

\usepackage{textcomp}

\DeclareSymbolFont{ugrf@m}{U}{eur}{m}{n}
\DeclareMathSymbol{\upmu}{\mathord}{ugrf@m}{"16}

\newtheorem{thm}{Theorem}[subsection]
\newtheorem{prop}[thm]{Proposition}

\newtheorem{cor}[thm]{Corollary}

\newtheorem{lem}[thm]{Lemma}
\theoremstyle{definition}
\newtheorem{defn}[thm]{Definition}
\newtheorem{setup}[thm]{Setup}
\newtheorem*{thma}{Theorem A}
\newtheorem*{thmb}{Theorem B}
\newtheorem*{thmc}{Theorem C}
\newtheorem*{conj}{Conjecture}
\theoremstyle{remark}
\newtheorem{ex}[thm]{Example}
\newtheorem{rmk}[thm]{Remark}

\newcommand{\cat}[1]{{\mathbf{#1}}}
\newcommand{\p}{\paragraph{}}
\newcommand{\spec}{\operatorname{Spec}}

\newcommand{\into}{\hookrightarrow}

\newcommand{\lot}{\otimes^{\mathbb{L}}}
\newcommand{\per}{{\ensuremath{\cat{per}}}\kern 1pt}
\newcommand{\stab}{\underline{\mathrm{CM}}}

\DeclareMathOperator{\id}{id}

\let\ker\relax\DeclareMathOperator{\ker}{ker}

\let\hom\relax\newcommand{\hom}{\mathrm{Hom}}
\newcommand{\enn}{\mathrm{End}}
\DeclareMathOperator{\tor}{Tor}
\DeclareMathOperator{\ext}{Ext}

\newcommand{\C}{\mathbb{C}}
\newcommand{\Q}{\mathbb{Q}}
\newcommand{\Z}{\mathbb{Z}}
\newcommand{\N}{\mathbb{N}}

\renewcommand{\P}{\mathbb{P}}
\newcommand{\R}{{\mathrm{\normalfont\mathbb{R}}}}
\newcommand{\mm}{{\upmu\upmu}}

\newcommand{\tbtm}[4]{\ensuremath{\begin{pmatrix}#1&#2\\#3&#4\end{pmatrix}}}
\newcommand{\stbtm}[4]{\ensuremath{\left(\begin{smallmatrix}#1&#2\\#3&#4\end{smallmatrix}\right)}}
\newcommand{\sthbthm}[9]{\ensuremath{\left(\begin{smallmatrix}#1&#2&#3\\#4&#5&#6\\#7&#8&#9\end{smallmatrix}\right)}}

\usepackage[utf8]{inputenc}
\usepackage[T1]{fontenc}
\usepackage{lmodern}

\title{The derived contraction algebra}
\author{Matt Booth}
\numberwithin{equation}{section}
\usepackage[nottoc]{tocbibind}

\newcommand{\con}{\mathrm{con}}
\newcommand{\dca}{{\ensuremath{A^\mathrm{der}_\con}}}
\newcommand{\dq}{\ensuremath{A/^{\mathbb{L}}\kern -2pt AeA} }
\newcommand{\dqb}{\ensuremath{B/^{\mathbb{L}}\kern -2pt BeB} }
\newcommand{\thick}{\ensuremath{\cat{thick} \kern 0.5pt}}

\newcommand{\dgart}{\cat{dgArt}_k^{\leq 0}}

\newcommand{\proart}{{\cat{pro}(\cat{dgArt}_k^{\leq 0})}}

\newcommand{\recol}{\mathrel{\substack{\textstyle\leftarrow\\[-0.6ex]
			\textstyle\rightarrow \\[-0.6ex]
			\textstyle\leftarrow}}}

\begin{document}
\date{}
\maketitle{}

\begin{center}
	\textsc{Abstract}
	\begin{figure}[ht]\hspace{0.1\linewidth}
		\begin{minipage}[c]{0.8\linewidth}
			Using Braun--Chuang--Lazarev's derived quotient, we enhance the contraction algebra of Donovan--Wemyss to an invariant valued in differential graded algebras. Given an isolated contraction $X \to X_\mathrm{con}$ of an irreducible rational curve $C$ to a point $p$, we show that its derived contraction algebra controls the derived noncommutative deformations of $C$. We use dg singularity categories to prove that, when $X$ is smooth, the derived contraction algebra recovers the geometry of $X_\mathrm{con}$ complete locally around $p$, establishing a positive answer to a derived version of a conjecture of Donovan and Wemyss. When $X \to X_\mathrm{con}$ is a simple threefold flopping contraction, it is known that the Bridgeland--Chen flop-flop autoequivalence of $D^b(X)$ is a `noncommutative twist' around the contraction algebra. We show that the derived contraction algebra controls an analogous autoequivalence in more general settings, and in particular for partial resolutions of Kleinian singularities.
		\end{minipage}
	\end{figure}
\end{center}

\section{Introduction}
To understand the minimal model program, one must be able to control \textbf{flops}, which are a special class of birational maps that are isomorphisms in codimension one. Indeed, any two minimal models of a given variety are linked by flops, which was first proved in all dimensions by Kawamata \cite{kawconnect}. Given a flopping contraction of an irreducible curve inside a threefold, Donovan and Wemyss \cite{DWncdf} define an invariant, the \textbf{contraction algebra}, which is a noncommutative Artinian local algebra. In this paper, we enhance their construction to a differential graded algebra (\textbf{dga} for short) whose zeroth cohomology is the classical contraction algebra. The point of our construction is that it behaves well in a general setup, meaning that our results hold just as well for isolated contractions of a curve in a surface or in a variety of high dimension. This allows us to generalise the key results of \cite{DWncdf} to non-threefold settings.

\p The contraction algebra recovers many interesting invariants of flops: the normal bundle of the flopping curve (for which there are three choices \cite{pinkham}), the width \cite{reidpagoda} and the length \cite{ckmcplx}. In fact, Donovan and Wemyss conjecture that it is a complete invariant: let $X \to \spec R$ and $X' \to \spec R'$ be formal flopping contractions of an irreducible rational curve in a smooth projective threefold. If the associated contraction algebras are isomorphic, then they conjecture that $R \cong  R'$. We prove (Theorem A) that the derived version of this conjecture holds: if the associated derived contraction algebras are quasi-isomorphic, then $R\cong R'$. We show that this holds in all dimensions, not just for threefolds. The proof makes use of a recent theorem of Hua and Keller \cite{huakeller} to recover $R$ from its dg singularity category. In fact, in the smooth setting the derived contraction algebra determines and is determined by the dg singularity category of $R$.

\p Our results are not specialised to the case of threefold flops: indeed, the search for a derived analogue of the contraction algebra was motivated by a desire to generalise the Donovan--Wemyss theory to surfaces, in particular partial resolutions of Kleinian singularities. In this setting, many statements about the contraction algebra are no longer true, and we must move to the derived world to obtain the results we want.

\p The contraction algebra is known to control the noncommutative deformation theory of the flopping curves. Our second main result (Theorem B) shows that, when one restricts to contractions of a single irreducible curve, the derived contraction algebra controls the derived noncommutative deformations of this curve. We note that this deformation-theoretic interpretation allows us to carry out explicit calculations, and we provide several examples. We include an appendix on $A_\infty$-algebras to collect together the results we use when doing computations; in particular, Kadeishvili's theorem (and Merkulov's proof) and some facts on formality, Massey products, and $A_\infty$ Koszul duality.

\p In the threefold setting, Bridgeland \cite{bridgeland} and Chen \cite{chen} prove that a flop $X \dashrightarrow X^+$ induces a derived equivalence $D^b(X)\xrightarrow{\simeq} D^b(X^+)$. Algebraically, flop autoequivalences can be interpreted as \textbf{mutation autoequivalences}, which are analogous to Fomin--Zelevinsky mutation \cite{iyamareiten}. Donovan and Wemyss prove that the contraction algebra controls the mutation-mutation autoequivalence for threefolds. In the surface setting, although one no longer has flops -- since a flop is an isomorphism in codimension one -- one still has mutation autoequivalences, so in some sense one can still flop curves on the derived level. Our third main result (Theorem C) shows that the derived contraction algebra controls this mutation autoequivalence. More precisely, the mutation equivalence is a sort of `noncommutative twist' around a quotient of the derived contraction algebra. Restricted to the threefold setting, this gives a new proof of Donovan and Wemyss' result. Along the way we prove generalisations of several results from \cite{DWncdf}.

\subsection{Motivation: the contraction algebra} Consider a threefold flopping contraction of a chain of rational curves $\pi: X \to \spec R$ with $R$ complete local. In order to give a noncommutative proof of Bridgeland's theorem on derived equivalences, Van den Bergh \cite{vdb} constructs a ring $A=\enn_R(R\oplus M)$ equipped with a derived equivalence $D(A)\xrightarrow{\simeq} D(X)$. When $X$ is smooth, $A$ is an example of a \textbf{noncommutative crepant resolution} (\textbf{NCCR}) of $R$. Donovan and Wemyss \cite{DWncdf} used this noncommutative model $A$ to define a new invariant of $\pi$, a noncommutative finite-dimensional algebra called the \textbf{contraction algebra}, defined as $A_\con\coloneqq A/AeA$, the quotient of $A$ by $e=\id_R$. Equivalently, $A_\con$ is the stable endomorphism algebra of $M$.

\p Suppose that the chain of flopping curves is composed of $n$ irreducible rational curves linked together in some (Dynkin) configuration. Then $A_\con$ is an $n$-\textbf{pointed} algebra, meaning that it has an augmentation $A_\con \to k^n$. In particular, if the flop is \textbf{simple} (meaning that the exceptional locus is irreducible) then the contraction algebra is an Artinian local algebra. More generally, if the exceptional locus has irreducible components $C_1,\ldots, C_n$ then the sheaves $\mathcal{O}_{C_i}$ on $X$ correspond across the derived equivalence $D(A) \xrightarrow{\simeq} D(X) $ to the one-dimensional simple $A$ -modules appearing as the irreducible summands of the $n$-dimensional $A$-module $A_\con / \mathrm{rad}(A_\con)$ (see \cite[\S2]{DWncdf}).

\p However, for partial resolutions $Y \to \spec S$ of Kleinian singularities, although the contraction algebra can be defined, it is no longer a satisfactory invariant. By producing an infinite family of one-curve partial resolutions of type $A_n$ surface singularities, we show explicitly that the contraction algebra does not control the mutation-mutation autoequivalence (see \S\ref{intromut}). Because the contraction algebra is $k$ for each member of this family, the exceptional locus (a copy of $\P^1$) is rigid and does not deform, even noncommutatively. However, one can show that this curve admits nontrivial derived deformations, indicating that one should study a derived version of the contraction algebra.

\subsection{The derived quotient}
Noncommutative partial resolutions, such as those constructed in \cite{vdb}, often yield rings with idempotents, which motivates the serious homological study of such rings. In particular, if $A$ is a ring with idempotent $e$ then putting $R\coloneqq eAe$, the standard functors $D(A)\recol D(R)$ fit into one half of a \textbf{recollement}, a strong type of short exact sequence of triangulated categories. Kalck and Yang \cite{kalckyang, kalckyang2} show that there exists a nonpositive cohomologically graded dga $B$ with $H^0(B)\cong A/AeA$ fitting into a recollement $D(B)\recol D(A) \recol D(R)$.

\p Coming from the background of homotopical algebra, Braun, Chuang, and Lazarev \cite{bcl} define the \textbf{derived quotient} $\dq$ of $A$ by $e$ to be the universal dga under $A$ that homotopy annihilates $e$. If $A$ is an algebra in degree zero, then $\dq$ is a nonpositive cohomologically graded dga with $H^0(\dq)\cong A/AeA$. Braun--Chuang--Lazarev prove that the derived quotient fits into a standard recollement $D(\dq) \recol D(A)\recol D(eAe)$, which gives an abstract construction of Kalck and Yang's dga $B$. 

\p Given a suitably general isolated contraction $X \to X_\con$ of an irreducible rational curve to a point, we will define the derived contraction algebra $\dca$ to be the derived quotient $\dq$, where $A$ is Van den Bergh's noncommutative model for $X$ (see \S\ref{dercondefns} for the rigorous construction). We construct $\dca$ complete locally; in particular it will only depend on the formal fibre $U \to \spec R$, where we let $R$ denote the completion of the local ring of $X_\con$ at $p$. It follows that $eAe\cong R$ here, so that one can think of $D(\dq)$ as a sort of `derived exceptional locus'. Immediately we see that $H^0(\dca)\cong A_\con$, the Donovan--Wemyss contraction algebra. When $\spec R$ is a hypersurface, 2-periodicity in the singularity category of $R$ will give us a \textbf{periodicity element} $\eta\in H^{-2}(\dca)$ which is central in the graded algebra $H^*(\dca)$. In the setting of threefold flopping contractions, we will show that when $X \to X_\con$ is a minimal model, then $H^*(\dca)\cong A_\con[\eta]$, the classical contraction algebra with the periodicity element freely adjoined in degree $-2$. In general, $\dca$ is not formal: in \S\ref{threecomps} we will provide an explicit computation (as a minimal $A_\infty$-algebra) of the derived contraction algebra associated to the Pagoda flop, and we will see that $\dca$ is not formal in this setting.

\p Working independently, Hua and Keller \cite{huakeller} show that, in the smooth threefold setting, the derived contraction algebra may be computed as a Ginzburg dga, and we verify this for our threefold examples. It is unclear to us whether the same should be true in the singular setting.
\subsection{The derived Donovan--Wemyss conjecture}
We recall the Donovan--Wemyss conjecture:
\begin{conj}[{\cite[1.4]{DWncdf}}]
	Let $X \to \spec R$ and $X' \to \spec R'$ be flopping contractions of an irreducible rational curve in a smooth projective threefold, with $R$ and $R'$ complete local rings. If the associated contraction algebras are isomorphic, then $R \cong  R'$.
\end{conj}Note that over $\C$, this is an analytic classification. We prove a derived version of the above, which uses singularity categories as the key step in the proof. We are also able to remove the assumption that the dimension is three.

\p If $R$ is a noetherian ring, its singularity category is the triangulated (or dg) category given by the quotient $D_{\mathrm{sg}}(R)\coloneqq D^b(R)/\per(R)$, which can be seen as quantifying the type of singularities of $R$. Singularity categories were introduced by Buchweitz \cite{buchweitz} who proved that when $R$ is Gorenstein, $D_{\mathrm{sg}}(R)$ is equivalent to the stable category $\stab R$ of \textbf{maximal Cohen--Macaulay} (\textbf{MCM}) $R$-modules. Let $R$ be a complete local hypersurface singularity and $M$ a MCM $R$-module, and put $A\coloneqq \enn_R(R\oplus M)$. Then $A$ comes with an idempotent $e=\id_R$ and we can view $A$ as a kind of \textbf{noncommutative partial resolution} of $R$. In \cite{dqdefm}, it is shown that (under some finiteness and smoothness conditions) the derived quotient $\dq$ recovers the dg singularity category $D_{\mathrm{sg}}(R)$. Hua and Keller \cite{huakeller} show that this dg singularity category actually recovers $R$, and it then follows that $\dq$ determines $R$, which we recap as \ref{recovthm}. After checking that the relevant finiteness conditions hold, we can immediately obtain our first main theorem:
\begin{thma}[derived Donovan--Wemyss conjecture (\ref{ddwconj})]\label{thrma}
	Let $X \to \spec R$ and $X' \to \spec R'$ be isolated contractions of an irreducible rational curve in a smooth variety, with $R$ and $R'$ complete local rings. If the associated derived contraction algebras are quasi-isomorphic, then $R \cong  R'$.
\end{thma}
We remark that in the multiple-curve case, the correct generalisation of the of the Donovan--Wemyss conjecture for threefolds \cite[1.3]{jennyf} stipulates that the contraction algebras should be derived equivalent, not necessarily isomorphic. So in the multiple-curve case, one would expect that the quasi-isomorphism type of the derived contraction algebra is too fine an invariant, and that one should use something like derived Morita equivalence instead.
\subsection{Deformation theory}
Let $X \to X_\con$ be a threefold simple flopping contraction, with exceptional locus the single irreducible rational curve $C$, and let $A$ be van den Bergh's noncommutative model for $X$. Recall that across the derived equivalence $D(X)\xrightarrow{\simeq} D(A)$, the simple sheaf $\mathcal{O}_{C}$ corresponds to the simple one-dimensional $A$-module $S\coloneqq A_\con / \mathrm{rad}(A_\con)$. Donovan and Wemyss \cite[\S 3]{DWncdf} prove that $A_\con$ represents the functor of noncommutative deformations of $S$. In fact, in the non-simple case where $C$ may be a chain of rational curves, $A_\con$ represents the functor of noncommutative pointed deformations of $A_\con / \mathrm{rad}(A_\con)$ \cite{contsdefs}.

\p If $X\to X_\con$ is an isolated contraction of an irreducible rational curve in a smooth variety, then results of Efimov, Lunts and Orlov \cite{ELO, ELO2} allow us to conclude that $\dca$ (pro)represents the functor of derived noncommutative deformations of $S$. In \cite{dqdefm}, we generalised their result to the singular setting to show that in general, the derived quotient admits a deformation-theoretic interpretation. This immediately allows us to obtain our second main theorem, a derived analogue of Donovan and Wemyss's result which we state informally:
\begin{thmb}[\ref{dcaprorep}]\label{thrmb}
Let $X \to X_\con$ be an isolated contraction of an irreducible rational curve $C$. Then $\dca$ (pro)represents the functor of derived noncommutative deformations of $C$.
\end{thmb}
We note that this provides a new proof, via the inclusion-truncation adjunction, that $A_\con$ represents the underived noncommutative deformations of $C$. We also remark that a similar theorem ought to hold in the non-simple (equivalently, pointed) case, when $C$ is a not necessarily irreducible chain of curves.

\p The deformation-theoretic interpretation of $\dca$ allows us to establish local-to-global arguments on computing the derived contraction algebra of a contraction with non-affine base, and also allows us to compute $\dca$ as a (minimal) $A_\infty$-algebra via Koszul duality. We provide some explicit computations of derived contraction algebras for Pagoda flops, as well as one-curve partial resolutions of $A_n$ singularities.

\subsection{The mutation-mutation autoequivalence}\label{intromut}
Let $X \to X_\con$ be a threefold simple flopping contraction. If $X^+$ is the flop of $X$, then $X$ is also the flop of $X^+$. Composing Bridgeland and Chen's derived equivalences yields a nontrivial autoequivalence of $D(X)$, the \textbf{flop-flop autoequivalence}. On the algebraic side, this corresponds to a \textbf{mutation-mutation autoequivalence} $ \mm: D(A) \to D(A)$. Donovan and Wemyss \linebreak \cite[5.10]{DWncdf} prove that $A_\con$ controls the mutation-mutation autoequivalence, in the sense that $\mm$ is isomorphic to a `noncommutative twist' around the $A$-bimodule $A_\con$. More precisely, they show that $\mm$ is represented by the $A$-bimodule ${AeA \simeq \mathrm{cocone}(A \to A_\con)}$.

\p When $X \to \spec R$ is a partial resolution of a Kleinian singularity, one can still define the mutation-mutation autoequivalence, although, on the level of the geometry, it no longer comes from a genuine birational map. It is possible, using the machinery developed in \S\ref{mutnauto}, to show that in the case of one-curve partial resolutions of $A_n$ singularities considered in \S\ref{surfcomps}, the mutation-mutation autoequivalence is not represented by $AeA$. Hence, it follows that $A_\con$ does not always control $\mm$ for singular surfaces. Our third main theorem shows that $\dca$ does control $\mm$, in the following analogous sense:
\begin{thmc}[\ref{mutncontrol} and \ref{twistrmk}]\label{thrmc}
Let $X \to \spec R$ be either a threefold simple flopping contraction to a complete local base, or a cut of such a contraction to a one-curve partial resolution of a Kleinian singularity. Let $A_\mm\coloneqq  \tau_{\geq -1}(\dca)$ be the two-term truncation of the associated derived contraction algebra. Then $\mm$ is a `noncommutative twist' around $A_\mm$, in the sense that $\mm$ is represented by the $A$-bimodule $\mathrm{cocone}(A \to A_\mm)$.
\end{thmc}
The twist interpretation comes from the fact that one has an exact triangle $$\R\hom_A(A_\mm,-) \to \id \to \mm \to$$ of endofunctors of $D(A)$. The crux of our proof of Theorem C is the statement that $\mm$ restricts to the shift functor $[-2]$ on $D(\dca)$, and the proof of this second fact makes crucial use of the recollement $D(\dca)\recol D(A) \recol D(R)$ to reduce to a calculation in the singularity category of $R$. In the threefold setting, $\dca$ has no cohomology in degree $-1$, and hence $A_\mm\simeq A_\con$, which recovers Donovan and Wemyss's result. In the surface setting, Auslander-Reiten duality \cite{auslander} allows one to conclude that $A_\mm$ always has cohomology in degree -1, and hence is never $A_\con$. Note that, by the inclusion-truncation adjunction, $A_\mm$ represents the functor of $[-1,0]$-truncated derived noncommutative deformations of $S$. One can informally think of $A_\mm$ as $\dca/\eta$, the quotient of the derived contraction algebra by the periodicity element $\eta \in H^{-2}(\dca)$.

\subsection{Notation and conventions}
Throughout this paper, $k$ will denote an algebraically closed field of characteristic zero. We will need this assumption, although many assertions we make will work in much greater generality. Modules are right modules, unless stated otherwise. Consequently, noetherian means right noetherian, global dimension means right global dimension, et cetera. Unadorned tensor products are by default over $k$. All complexes, unless stated otherwise, are $\Z$-graded cochain complexes, i.e. the differential has degree 1. If $X$ is a complex, let $X[i]$ denote `$X$ shifted left $i$ times': the complex with $X[i]^j=X^{i+j}$ and differential twisted by a sign of $(-1)^i$.  Recall that the \textbf{mapping cone} $\mathrm{cone}(f)$ of a degree zero map $f:X \to Y$ of complexes is (a representative of) the homotopy cokernel of $f$; concretely it is given by $X[1]\oplus Y$ with differential that combines $f$ with the differentials on $X$ and $Y$. The \textbf{mapping cocone} of $f$ is $\mathrm{cone}(f)[-1]$; it is a representative of the homotopy kernel. If $X$ is a complex of modules we will denote its cohomology complex by $H(X)$ or just $HX$. If $x$ is a homogeneous element of a complex of modules, we denote its degree by $|x|$.

\p A $k$-algebra is a $k$-vector space with an associative unital $k$-bilinear multiplication. A \textbf{differential graded algebra} (\textbf{dga} for short) over $k$ is a complex of $k$-vector spaces $A$ with an associative unital chain map $\mu:A\otimes A \to A$, which we refer to as the multiplication. Note that the condition that $\mu$ be a chain map forces the differential to be a derivation for $\mu$.  A $k$-algebra is equivalently a dga concentrated in degree zero, and a graded $k$-algebra is equivalently a dga with zero differential. We will sometimes refer to $k$-algebras as \textbf{ungraded algebras} to emphasise that they should be considered as dgas concentrated in degree zero. A dga is \textbf{graded-commutative} or just \textbf{commutative} if all graded commutator brackets $[x,y]=xy-(-1)^{|x||y|}yx$ vanish. Commutative polynomial algebras are denoted with square brackets $k[x_1,\ldots, x_n]$ whereas noncommutative polynomial algebras are denoted with angle brackets $k\langle x_1,\ldots, x_n\rangle$. 

\p If $A$ is an algebra, write $\cat{Mod}\text{-}A$ for its category of right modules, $\cat{mod}\text{-}A\subseteq\cat{Mod}\text{-}A$ for its category of finitely generated modules, $D(A)\coloneqq D(\cat{Mod}\text{-}A)$ for its unbounded derived category, $D^b(A)\coloneqq D^b(\cat{mod}\text{-}A)$ for its bounded derived category, and $\cat{per}(A) \subseteq D^b(A)$ for the subcategory on perfect complexes (i.e. those complexes quasi-isomorphic to bounded complexes of finitely generated projective modules). Recall that an object $X$ of a triangulated category $\mathcal T$ is \textbf{compact} if $\hom_\mathcal{T}(X,-)$ commutes with all direct sums. We then have $\cat{per}(A)=\{ \text{compact objects in } D(A)\}$. A \textbf{dg module} (or just a \textbf{module}) over a dga $A$ is a complex of vector spaces $M$ together with an action map $M \otimes A \to M$ satisfying the obvious identities (equivalently, a dga map $A \to \enn_k(M)$). Note that a dg module over an ungraded ring is exactly a complex of modules. If $A$ is a dga, write $D(A)$ for its unbounded derived category: this is the category of all dg modules over $A$ localised along the quasi-isomorphisms. If $A$ is a dga we define $\cat{per}(A)\coloneqq \{ \text{compact objects in } D(A)\}$.

\p Let $V$ be a complex. The \textbf{total dimension} or just \textbf{dimension} of $V$ is $\sum_{n\in \Z}\mathrm{dim}_k V^n$. Say that $V$ is \textbf{finite-dimensional} or just \textbf{finite} if its total dimension is finite. Say that $V$ is \textbf{locally finite} if each $\mathrm{dim}_kV^n$ is finite. Say that $V$ is \textbf{cohomologically locally finite} if the cohomology dg vector space $HV$ is locally finite. There are obvious implications $$\text{finite }\implies\text{ locally finite }\implies\text{cohomologically locally finite}.$$We use the same terminology in the case that $V$ admits extra structure (e.g. that of a dga). We denote isomorphisms (of modules, functors, \ldots) with $\cong$ and quasi-isomorphisms with $\simeq$.

\subsection{Structure of the paper}
In Section 2, we recall some key facts about derived quotients from \cite{dqdefm}. In Section 3, we define the derived contraction algebra, and prove Theorems A and B. In Section 4, we compute the derived contraction algebras of the family of Pagoda flops, and sketch a computation for the Laufer flop. In Section 5, we compute the derived contraction algebras of a family of one-curve partial resolutions of $A_n$ singularities. In Section 6 we consider the mutation-mutation autoequivalence and prove Theorem C. In the Appendix we gather some results we use about $A_\infty$-algebras to do our computations.

\subsection{Acknowledgements} The author is a PhD student at the University of Edinburgh, supported by the Engineering and Physical Sciences Research Council, and this work forms part of his thesis. He would like to thank his supervisor Jon Pridham for his consistent support and guidance. He would like to thank Michael Wemyss for introducing him to contraction algebras, and his help at every stage of the project, especially with the arguments of \S\ref{slicingsctn}. He would like to thank Jenny August, Ben Davison, Will Donovan, and Bernhard Keller for helpful discussions and comments.

\section{Preliminaries}
In this section we will recall some constructions and theorems from \cite{dqdefm}. We review some facts about the derived quotient, in particular some structure theorems for derived quotients of noncommutative partial resolutions of complete local hypersurface singularities. We note that the derived quotient admits a deformation-theoretic interpretation. We recall an important recovery theorem (\ref{recovthm}) which will be the main technical underpinning of our solution to the derived Donovan--Wemyss conjecture (\ref{ddwconj}). We focus solely on the base $\spec R$; in the next section we will look at contractions $X \to \spec R$ and define derived contraction algebras as certain derived quotients of noncommutative partial resolutions of $R$ derived equivalent to $X$.
\subsection{Rings with idempotents: derived quotients and partial resolutions}
We briefly review some properties of the derived quotient. This section will also help us establish some notation. For more details, see \cite[\S4]{dqdefm} (which is tailored to what we want to do) or the original reference \cite{bcl}. Given an algebra $A$ and an idempotent $e \in A$, Braun--Chuang--Lazarev construct the \textbf{derived quotient}, a nonpositive dga $\dq$ with $H^0(\dq)\cong A/AeA$, which is the universal dga under $A$ that homotopy annihilates $e$ \cite{bcl}. When the quotient $A/AeA$ is a local ring, and under some finiteness conditions, $\dq$ has a deformation-theoretic interpretation:
\begin{prop}[{\cite[Theorem A]{dqdefm}}]\label{prorepthm}
	Let $A$ be a $k$-algebra and $e$ an idempotent. Suppose that $A/AeA$ is a finite-dimensional local algebra and moreover that $H^j(\dq)$ is finite-dimensional for all $j<0$. Let $S$ be $A/AeA$ modulo its radical, regarded as a right $A$-module. Then $\dq$ is naturally a pro-Artinian dga, and moreover prorepresents the functor of framed noncommutative derived deformations of $S$.
\end{prop}
The derived quotient fits into the following recollement, a strong type of short exact sequence of triangulated categories (see \cite{bbd} for the definition):
\begin{prop}[e.g. {\cite[4.4.1]{dqdefm}}]\label{recoll}
	Let $A$ be an algebra over $k$, and let $e\in A$ be an idempotent. For brevity, write $Q\coloneqq \dq$ and $R\coloneqq eAe$, and put \begin{align*}
	i^*\coloneqq -\lot_A Q, &\quad j_!\coloneqq  -\lot_{R} eA
	\\ i_*=\R\hom_{Q}(Q,-), &\quad j^!\coloneqq \R\hom_A(eA,-)
	\\ i_!\coloneqq \lot_{Q}Q, & \quad j^*\coloneqq -\lot_A Ae
	\\ i^! \coloneqq \R\hom_{A}(Q,-), & \quad j_*\coloneqq \R\hom_{R}(Ae,-)
	\end{align*}
	Then the diagram of unbounded derived categories
	$$\begin{tikzcd}[column sep=huge]
	D(Q) \ar[r,"i_*=i_!"]& D(A)\ar[l,bend left=25,"i^!"']\ar[l,bend right=25,"i^*"']\ar[r,"j^!=j^*"] & D(R)\ar[l,bend left=25,"j_*"']\ar[l,bend right=25,"j_!"']
	\end{tikzcd}$$
	is a recollement diagram.
\end{prop}
\begin{rmk}
	In particular, $D(A)$ admits a semi-orthogonal decomposition $\langle D(\dq), D(R)\rangle $.
\end{rmk}
\begin{defn}Let $\mathcal X$ be a subclass of objects of a triangulated category $\mathcal{T}$. Then $\thick_\mathcal{T} \mathcal X $ denotes the smallest triangulated subcategory of $\mathcal{T}$ containing $\mathcal{X}$ and closed under taking direct summands. Similarly, $\langle \mathcal X \rangle_\mathcal{T}$ denotes the smallest triangulated subcategory of $\mathcal{T}$ containing $\mathcal{X}$, and closed under taking direct summands and all existing set-indexed coproducts. We will often drop the subscripts if $\mathcal T$ is clear. If $\mathcal X$ consists of a single object $X$, we will write $\thick X$ and $\langle X \rangle$.
\end{defn}
\begin{defn}
	Let $\mathcal{T}$ be a triangulated category and let $X$ be an object of $\mathcal{T}$. Say that $X$ is \textbf{relatively compact} (or \textbf{self compact}) in $\mathcal{T}$ if it is compact as an object of $\langle X \rangle_{\mathcal{T}}$.
\end{defn}
\begin{prop}[{\cite[4.4.7]{dqdefm} or \cite[1.7]{jorgensen}}]
	The right $A$-module $\dq$ is relatively compact in $D(A)$.
\end{prop}
Our preferred method of generating rings with idempotents will be as \textbf{noncommutative partial resolutions}. This method allows us to specify the corner ring $eAe$, which will always be Gorenstein.
\begin{defn}
	Let $R$ be a commutative ring. Say that $R$ is \textbf{Gorenstein} if $R$ is noetherian and has finite injective dimension as an $R$-module. 
	\end{defn}
\begin{rmk}
	When $R$ has finite Krull dimension, being Gorenstein is a local property.
	\end{rmk}
\begin{defn}
Let $R$ be a Gorenstein ring. If $M$ is an $R$-module, write $M^\vee$ for the $R$-linear dual $\hom_R(M,R)$. A finitely generated $R$-module $M$ is \textbf{maximal Cohen--Macaulay} or just \textbf{MCM} if the natural map $\R\hom_R(M,R)\to M^\vee$ is a quasi-isomorphism (equivalently, if $\ext_R^j(M,R)$ vanishes whenever $j> 0$).
\end{defn}
\begin{defn}\label{partrsln}
	Let $R$ be a Gorenstein $k$-algebra. A $k$-algebra $A$ is a \textbf{(noncommutative) partial resolution} of $R$ if it is of the form $A\cong\enn_R(R\oplus M)$ for some MCM $R$-module $M$. Note that $A$ is a finitely generated module over $R$, and hence itself a noetherian $k$-algebra. Say that a partial resolution is a \textbf{resolution} if it has finite global dimension.
\end{defn}
\begin{rmk}
	Nothing really stops us omitting the Gorenstein condition on $R$, but the above definition is general enough for us. 
\end{rmk}
If $A=\enn_R(R\oplus M)$ is a noncommutative partial resolution of $R$, observe that $e\coloneqq \id_R$ is an idempotent in $A$. One has $eAe\cong R$, $Ae \cong R \oplus M$, and $eA \cong R \oplus M^\vee$. In particular, $A/AeA$ is the stable endomorphism ring $\underline{\enn}_R(M)$ (two maps $X \to X$ are stably equivalent if and only if their difference factors through a projective summand of $X$, as in e.g. \cite{buchweitz}).

\subsection{Hypersurfaces}\label{clr}
We recap some of the ideas of \cite[\S5]{dqdefm}. In this section, fix $R\coloneqq k\llbracket x_1,\ldots x_n \rrbracket / \sigma$ a complete local isolated hypersurface singularity (with $\sigma\neq 0$ in the maximal ideal of $k\llbracket x_1,\ldots x_n \rrbracket$). Note that $R$ is Gorenstein by e.g. \cite[21.19]{eisenbud}. Fix $A=\enn_R(R\oplus M)$ a partial resolution of $R$.
\begin{defn}
	The \textbf{singularity category} of $R$ is the triangulated category $$D_\mathrm{sg}(R)\coloneqq D^b(R)/\per R.$$
\end{defn}
\begin{prop}[\cite{buchweitz}]
Let $\stab R$ be the category whose objects are the MCM $R$-modules and whose morphisms are the $R$-linear maps up to stable equivalence. It's triangulated, with shift given by the (inverse of the) syzygy functor $\Omega$. Moreover, $D_\mathrm{sg}(R)$ and $\stab R$ are triangle equivalent.
\end{prop}
\begin{rmk}
	This can be enhanced to a quasi-equivalence of dg categories, but we will not need to use this fact.
\end{rmk}
\begin{prop}[\cite{eisenbudper}]
	One has $\Omega ^2 \cong \id$ on $\stab R$.
\end{prop}
Hence, one can just as well regard $\Omega$ as the shift functor of $\stab R$.
\begin{prop}[{\cite[4.6.5]{dqdefm} or \cite[6.6]{kalckyang2}}] \label{kymap}
There is a map of triangulated categories $\Sigma:\per(\dq) \to D_\text{sg}(R)$, sending $\dq$ to $Ae$. Moreover $\Sigma$ has image $\thick_{D_\text{sg}(R)}(Ae)$, and kernel $\per_\mathrm{fg}(\dq)$, those perfect modules with total cohomology finitely generated over $A/AeA$. We call $\Sigma$ the \textbf{singularity functor}.
\end{prop}
The singularity functor induces periodicity in the cohomology of $\dq$:
\begin{prop}[{\cite[5.2.1]{dqdefm}}]\label{dqcohom}
One has $$H^j(\dq)\cong \begin{cases}
0 & j>0 \\
\underline{\enn}_R(M) & j=0\\
\ext_R^{-j}{(M,M)} & j\leq 0.
\end{cases}$$Moreover, if $i\geq 0$ is even then $\ext_R^i(M,M)\cong \underline{\enn}_R(M)$ as $R$-modules. If $i\geq 0$ is odd then $\ext_R^i(M,M)\cong \ext_R^1(M,M)$.
	\end{prop}
In fact, this 2-periodicity is witnessed by an element in the cohomology of $\dq$:
\begin{prop}[{\cite[\S5.2]{dqdefm}}]\label{etaex}Suppose that $\underline{\enn}_R(M)$ is an Artinian local algebra.
	\begin{enumerate}
		\item[\emph{i)}] There is a central element $\eta \in H^{-2}(\dq)$, unique up to multiplication by units, that we call the \textbf{periodicity element}.
		\item[\emph{ii)}] Multiplication by $\eta$ induces isomorphisms $H^j(\dq) \to H^{j-2}(\dq)$ for all $j\leq 0$.
		\item [\emph{ii)}]The element $\eta$ lifts to an element of $ HH^{-2}(\dq)$, the $-2^\text{nd}$ Hochschild cohomology of $\dq$ with coefficients in itself.
		\item[\emph{iv)}] There is an exact triangle of $\dq$-bimodules $\dq \xrightarrow{\eta} \dq \to \tau_{>-2}(\dq)$.
	\end{enumerate}
\end{prop}
\begin{rmk}
	Note that $\eta$ is a central element of the cohomology algebra $H(\dq)$, and need not lift to a genuinely central cocycle in $\dq$.
\end{rmk}
\begin{prop}[{\cite[Theorem B]{dqdefm}, see also \cite[5.9]{huakeller}}]\label{recovthm}Suppose that $\underline{\enn}_R(M)$ is an Artinian local algebra and that $H^j(\dq)$ is finite-dimensional for all $j<0$. Let $R'\coloneqq k\llbracket x_1,\ldots x_n \rrbracket / \sigma'$ be a complete local isolated hypersurface singularity with $\sigma'\neq 0$ in the maximal ideal of $k\llbracket x_1,\ldots x_n \rrbracket$. Let $M'$ be a MCM $R'$-module and $A'=\enn_{R'}(R'\oplus M')$ the associated partial resolution, with idempotent $e'=\id_{R'}$. Suppose that both $A$ and $A'$ are smooth (i.e. they have finite global dimension). Then if $\dq$ and $A'/^{\mathbb{L}}A'e'A'$ are quasi-isomorphic, we have an isomorphism $R\cong R'$.
\end{prop}
\begin{rmk}
In \cite{dqdefm}, instead of $A$ (resp. $A'$) being smooth it is required that $M$ (resp. $M'$) generates $D_\mathrm{sg}(R)$ (resp. $D_\mathrm{sg}(R')$). But by \cite[4.6.10]{dqdefm}, there is an exact sequence of triangulated categories $$\frac{D_{\text {fd}}(\dq) }{\per_\mathrm{fd}(\dq)} \longrightarrow D_\text{sg}(A) \longrightarrow \frac{D_{\text{sg}}(R)}{\thick_{D_\text{sg}(R)}(M)}$$ where $ D_\mathrm{fd}(\dq)$ denotes the subcategory of $D(\dq)$ on those modules whose total cohomology is finite-dimensional, and $\per_\mathrm{fd}(\dq)$ denotes the subcategory of $\per(\dq)$ on those modules whose total cohomology is finite-dimensional. So smoothness of $A$ implies that $M$ generates. Conversely, if $M$ generates then for $A$ to be smooth we require that every $\dq$-module of finite total cohomological dimension is perfect, which is equivalent to $S\coloneqq (A/AeA)/\mathrm{rad}(A/AeA)$ being perfect as a $\dq$-module. In this setting, since $\dq$ is quasi-isomorphic to its own Koszul double dual $(\dq)^{!!}$ by \ref{dqkd}, this latter condition is equivalent to $(\dq)^! \simeq \R\enn_A(S)$ having finite total cohomological dimension, i.e. the graded algebra $\ext^*_A(S,S)$ is finite-dimensional.
\end{rmk}

\begin{rmk}\label{dsgrmk}
Viewing $M$ as an object of the dg singularity category $D_\mathrm{sg}(R)$, let $\R\underline\enn_R(M)$ denote its endomorphism dga. It is shown in \cite[5.1.11]{dqdefm} that there is a quasi-isomorphism $$\dq\xrightarrow{\simeq}\tau_{\leq 0}\R\underline\enn_R(M)$$ between the derived quotient and the truncation of $\R\underline\enn_R(M)$ to nonpositive degrees. In particular, one can compute $\dq$ directly from knowledge of the dg singularity category. This also provides a way to produce an explicit model of $\dq$ where $\eta$ is represented by a genuinely central cocycle: first, stitch together the syzygy exact sequences for $M$ into a 2-periodic resolution $\tilde M \to M$. Let $\theta: \tilde M \to \tilde M$ be the degree 2 map whose components are the identity that witnesses this periodicity. Let $E=\enn_R(\tilde M)$, which is a dga. It is easy to see that $\theta$ is a central cocycle in $E$. Then, \cite[\S5.2]{dqdefm}
 shows that $\R\underline\enn_R(M)$ is quasi-isomorphic to the dga $E[\theta^{-1}]$, and $\eta$ is identified with $\theta^{-1}$ across this quasi-isomorphism. So it follows that $\dq$ is quasi-isomorphic to the dga $\tau_{\leq 0}\left(\enn_R(\tilde M)[\theta^{-1}]\right)$, which is naturally a dga over $k[\eta]=k[\theta^{-1}]$.

\end{rmk}

AR duality will assist us in some computations later:
\begin{prop}[Auslander-Reiten duality \cite{auslander}]
Let $T$ be a commutative complete local	Gorenstein isolated singularity of (Krull) dimension $d$. Let $X,Y$ be MCM $T$-modules. Then we have $$\underline{\hom}_T(X,Y) \cong {\ext}_T^1(Y,\Omega^{2-d} X)^*$$
\end{prop}
\begin{cor}
If the dimension of $R$ is even, then $\underline{\hom}_R(M,M) \cong {\ext}_R^1(M,M)^*$.
\end{cor}
\begin{defn}
	Call $N \in \stab R$ \textbf{rigid} if $\ext_R^1(N,N)\cong 0$.
\end{defn}

Note that AR duality implies that there are no nontrivial rigid modules if $\dim(R)$ is even. If $M$ is rigid then we have $H(\dq)\cong A/AeA[\eta]$, but in general $\dq$ need not be formal.

\section{The definition of the derived contraction algebra}\label{dercondefns}
In this section we define the derived contraction algebra associated to the contraction of a rational curve to a point. Our motivation is to mimic the constructions of Donovan and Wemyss from \cite{DWncdf, contsdefs, enhancements}. We give a deformation-theoretic description of the derived contraction algebra (\ref{dcaprorep}), which is a derived analogue of \cite[3.9]{DWncdf}. We show that the derived analogue of the Donovan--Wemyss conjecture is true (\ref{ddwconj}). We use the deformation-theoretic interpretation of the derived contraction algebra to globalise some of our results.

\subsection{The construction}
The global setup will be as follows:
\begin{setup}[Global]\label{globalsetup}
	Let $\pi:X \to X_\con$ be a projective birational morphism between two noetherian normal integral schemes over an algebraically closed field $k$ of characteristic zero. Assume that $\pi$ is crepant, that $\R\kern 2pt \pi_*\mathcal O_X = \mathcal O _{X_\con}$, and that $\pi$ is an isomorphism away from a single closed point $p$ in the base, where $C\coloneqq \pi^{-1}(p)$ is an irreducible rational (possibly non-reduced) curve. Assume in addition that $X_\con$ is a Gorenstein scheme that, complete locally around $p$, is an isolated hypersurface singularity.
\end{setup}
\begin{rmk}
Note that, if $X$ has dimension 3 or higher, we disallow divisorial contractions. If $X$ is smooth, then $\pi$ is a crepant resolution of an isolated singularity. One does not need the assumption that $X_\con$ is Gorenstein -- or that $p$ is an isolated hypersurface singularity -- to define the derived contraction algebra, but it is hard to prove much about it without them. One ought to be able make the more general assumption that $C$ is a tree of rational curves, but to do much with this definition, one needs first to prove pointed versions of results from \cite{dqdefm}. More generally, one should be able to drop the condition that $\pi^{-1}(p)$ is 1-dimensional, at the cost of some more assumptions about tilting bundles, as in \cite[2.5]{enhancements}.
\end{rmk}

Take an affine neighbourhood $\spec R \to X_\con$ of $p$, and let $U$ be the preimage of $\spec R$ under $\pi$. This gives the Zariski local setup:
\begin{setup}[Zariski Local]\label{localsetup}
	Let $\pi:U \to \spec R$ be a projective birational morphism between two noetherian normal integral schemes over an algebraically closed field $k$ of characteristic zero. Moreover, $\pi$ is crepant, $\R\kern 2pt \pi_*\mathcal O_U = R$, and $\pi$ is an isomorphism away from a single closed point $p$ in the base, where $C\coloneqq \pi^{-1}(p)$ is an irreducible rational (possibly non-reduced) curve. Furthermore, $R$ is Gorenstein and $\hat R _p$ is an isolated hypersurface singularity.
\end{setup}
\begin{rmk}\label{hmmpsetup}
	Note that this generalises the Crepant Setup 2.9 of \cite{hmmp}.
\end{rmk}
Now, in the local setup, by \cite[3.2.8]{vdb} there exists a (finite rank) tilting bundle $\mathcal V = \mathcal O_U \oplus \mathcal N$ on $U$. Put $\Lambda\coloneqq \enn_U(\mathcal V)$. By our assumptions, $\Lambda$ can be computed on the base: more precisely, \cite[2.5(2)]{enhancements} tells us that $\pi_*:\Lambda \to \enn_R(\pi_*\mathcal V)$ is an isomorphism. From now on, we'll identify $\Lambda$ with $\enn_R(\pi_*\mathcal V)\cong \enn_R(R \oplus N)$, where we write $N\coloneqq \pi_*\mathcal N$. I claim that $N$ is a maximal Cohen--Macaulay (MCM) $R$-module: to see this, first note that the $R$-module $\Lambda$ is MCM by \cite[\S4.2]{iyamawemyssfactorial}. Then, since $N$ is a summand of $\Lambda$, it must be MCM too. 

\p We would like to define the contraction algebra to be the derived quotient of $\Lambda$ by the idempotent $e=\id_R$. In order for this to behave well, we would like the finite-dimensional algebra $\Lambda_\con\coloneqq \Lambda/\Lambda e \Lambda$ to be local -- unfortunately, this need not happen, for the same reasons as \cite[\S2.4]{DWncdf}. In order to ensure locality, we need to pass to a complete local base, and then through a Morita equivalence. Letting $\hat R$ be the completion of $R_p$ along its maximal ideal, and letting $\hat U$ be the formal fibre, we obtain the following setup:
\begin{setup}[Complete local]
	$\pi:\hat U \to \spec \hat R$ is a projective birational morphism between two noetherian normal integral schemes over an algebraically closed field $k$ of characteristic zero, and $\hat R$ is a complete local hypersurface singularity with maximal ideal $p$. Moreover, $\pi$ is crepant, an isomorphism away from $p$, $\R\kern 2pt \pi_*\mathcal O_{\hat U} = \hat R$, and $C\coloneqq \pi^{-1}(p)$ is an irreducible rational (possibly non-reduced) curve.
\end{setup}
The arguments of \cite[\S2.4]{DWncdf} adapt to ensure that $\hat{ \mathcal V} \cong \mathcal O_{\hat U} \oplus \hat{\mathcal N}$ is a tilting bundle on $\hat U$, and that $\hat \Lambda \cong \enn_{\hat R}(\hat R \oplus \hat N)\cong \enn_{\hat U}(\hat{\mathcal V})$. Again, we may apply \cite{iyamawemyssfactorial} to see that $\hat N$ is still MCM over $\hat R$. Now, by \cite[3.2.7 and 3.5.5]{vdb} we may put $\hat{R}\oplus\hat N = \hat{R}^{\oplus a}\oplus M^{\oplus b}$, for some (necessarily MCM) indecomposable $\hat R$-module $M$ and some positive integers $a,b$. It follows that $A\coloneqq \enn_{\hat R}(\hat R \oplus M)$ is the basic algebra Morita equivalent to $\hat \Lambda$.
\begin{defn}Put $e\coloneqq \id_R$. The \textbf{contraction algebra} $A_\con$ associated to $\pi$ is the stable endomorphism algebra $A/AeA \cong \underline{\enn}_{\hat R}(\hat{R}\oplus M)\cong\underline{\enn}_{\hat R}(M)$. The \textbf{derived contraction algebra} $\dca$ is the derived quotient $\dq$.
\end{defn}
From the definition, it is obvious that $H^0(\dca)\cong A_\con$.

\subsection{First properties}
Keep notation as in the Global Setup \ref{globalsetup}. Let $\hat R$ be the completion of the local ring of $X_\con$ at $p$. We begin by proving some easy finiteness properties about $\dca$.
\begin{lem}\label{fdlem}
	Let $W$ be a finitely generated $\hat R$-module which is supported at $p$. Then $W$ is finite-dimensional over $k$.
\end{lem}
\begin{proof}
	This is standard: some power $n$ of $p$ annihilates $W$, and hence $W$ is a finitely generated module over ${\hat R}/p^n$, which is finite-dimensional over ${\hat R}/p\cong k$.
\end{proof}
\begin{lem}
Let $M$ be the $\hat R$-module defining $A_\con$ and $\dca$. If $q\neq p$ is a prime ideal of $\hat R$, then $M_q$ is projective.
\end{lem}
\begin{proof}
$M$ is the pushforward of a vector bundle along a map that is an isomorphism away from $p$.
\end{proof}
\begin{prop}[cf. {\cite[2.13.(1)]{DWncdf}}]\label{aconlocal}
	The algebra $A_\con$ is an Artinian local algebra.
\end{prop}
\begin{proof}If $q\neq p$ is a prime ideal of $\hat R$ then $(A_\con)_q\cong \underline{\enn}_{\hat R _q}(M_q)$, which vanishes because $M_q$ is projective. Hence $A_\con$ is supported at $p$ and hence Artinian. It's local because $M$ was indecomposable.
\end{proof}
\begin{prop}\label{fdcohom}
	The dga $\dca$ has finite-dimensional cohomology in each degree.
\end{prop}
\begin{proof}
We already know that $H^0(\dca)$ is finite-dimensional. Let $j<0$. Then we have an isomorphism $H^j(\dca)\cong{\ext}_{\hat R}^{-j}(M,M)$ by \ref{dqcohom}. But ${\ext}_{\hat R}^{-j}(M,M)_q\cong {\ext}_{{\hat R}_q}^{-j}(M_q,M_q)$ which vanishes if $q\neq p$. So ${\ext}_{\hat R}^{-j}(M,M)$ is supported at $p$, and so \ref{fdlem} applies.
\end{proof}

Since $\hat R$ is a hypersurface, the cohomology of $\dca$ is 2-periodic, by \ref{dqcohom} again. If $\dim R$ is even then every $H^j(\dca)$ hence has the same dimension for $j\leq 0$, by AR duality. Moreover, by \ref{etaex}, $H(\dca)$ is a finitely generated algebra, finite-dimensional in each degree, generated in degrees $0$, $-1$, and $-2$. The only degree $-2$ generator is the periodicity element $\eta$, which is central and torsionfree. Our first main theorem is a positive answer to a derived version of the Donovan--Wemyss conjecture \cite[1.4]{DWncdf}:
\begin{thm}[derived Donovan--Wemyss conjecture]\label{ddwconj}
Let $\pi:X \to X_\con$ and $\pi': X' \to X'_\con$ be two contractions satisfying the conditions of the Global Setup \ref{globalsetup}, contracting curves to points $p$ and $p'$ respectively. Assume in addition that $X$ and $X'$ are smooth and of the same dimension. Let $\dca$ and $\dca'$ be the derived contraction algebras of $\pi$ and $\pi'$ respectively. If $\dca$ and $\dca'$ are quasi-isomorphic, then the completions $\widehat{(X_\con)}_p$ and $\widehat{(X_\con')}_{p'}$ are isomorphic.
\end{thm}
\begin{proof}
A simple application of \ref{recovthm}.
\end{proof}

\subsection{Local to global computations via deformation theory}
We discuss further the deformation-theoretic description of the derived contraction algebra, and how we can use this description to compute the derived contraction algebra in differing neighbourhoods of $p$. For the relevant deformation theory, see \cite{ELO, ELO2} or \cite[\S3]{dqdefm}. Keep notation as in the Global Setup \ref{globalsetup}. Let $\hat R$ be the completion of the local ring of $X_\con$ at $p$. Let $C$ be the exceptional locus.
\begin{lem}
	Across the derived equivalence $D^b(\hat U) \to D^b(\hat A)$, the sheaf $\mathcal{O}_C(-1)$ corresponds to the simple $S\coloneqq A_\con / \mathrm{rad}(A_\con)$.
\end{lem}
\begin{proof}
	Denote the image of $\mathcal{O}_C(-1)$ by $S'$. The sheaf $\mathcal{O}_C(-1)$ is simple by \cite[3.5.7]{vdb}. The proof of \cite[2.13(3)]{DWncdf} adapts to show that $S'$ is naturally a module over $A_\con$. Since it's simple, it must be the unique simple module $S$.
\end{proof}
With this in mind, we define:
\begin{defn}
	A \textbf{noncommutative deformation of $C$} is a noncommutative deformation of the $A$-module $S$. A \textbf{derived noncommutative deformation of $C$} is a derived noncommutative deformation of the $A$-module $S$.
\end{defn}
\begin{thm}\label{dcaprorep}
	The derived contraction algebra prorepresents the functor of derived noncommutative deformations of the curve $C$. The contraction algebra prorepresents the functor of underived noncommutative deformations.
\end{thm}
\begin{proof}
	The first follows immediately from \ref{prorepthm}. The second is \cite[3.9]{contsdefs}.
\end{proof}
\begin{rmk}\label{prorepremk}
This gives a new proof that $A_\con$ represents the functor of underived noncommutative deformations of $C$. Let $\cat{Art}_k$ be the category of noncommutative Artinian local algebras with residue field $k$, and let $\dgart$ be the category of noncommutative Artinian local dgas with residue field $k$. Let $\mathrm{Def}_C:\cat{Art}_k\to\cat{Set}$ denote the functor of underived noncommutative deformations of $C$, in the sense of \cite[2.4]{enhancements} and let $\mathrm{dDef}_C:\dgart \to\cat{Set}$ denote the functor of derived noncommutative deformations of $C$, in the sense of \cite[10.1]{ELO} (Efimov--Lunts--Orlov use groupoid-valued functors, but we will only be concerned with $\pi_0$). Let $\Gamma \in \cat{Art}_k$. We know that $\mathrm{dDef}_C(\Gamma)\cong \hom_{\proart}(\dca,\Gamma)$, which by the inclusion-truncation adjunction is isomorphic to $\hom_{\cat{pro}(\cat{Art}_k)}(A_\con,\Gamma)\cong\hom_{\cat{Art}_k}(A_\con, \Gamma)$. So we need to prove that $\mathrm{dDef}_C(\Gamma)\cong \mathrm{Def}_C(\Gamma)$. 

\p Let $\tilde S$ be an underived deformation of $S$ over $\Gamma$, i.e. an $A\otimes\Gamma$-module, flat over $\Gamma$, which reduces to $S$ modulo $\mathfrak{m}_\Gamma$ (equivalently, such that $\tilde{S}\otimes_\Gamma k \cong S$). It's easy to see that $\tilde S \lot_\Gamma k \cong S$ inside the derived category $D(A\otimes \Gamma)$. Hence $\tilde S$ is a derived deformation of $S$. Moreover, if two underived deformations are isomorphic, they are clearly isomorphic as derived deformations, and hence we obtain a map of sets $\Phi:\mathrm{Def}_C(\Gamma)\to \mathrm{dDef}_C(\Gamma)$. It is injective, because $A\otimes\Gamma \text{-}\cat{mod}$ embeds in $D(A\otimes \Gamma)$. Observe that if $\tilde S\in D(A\otimes \Gamma)$ is a derived deformation of $S$ over $\Gamma$, then it must actually be an $A\otimes\Gamma$-module, concentrated in degree zero. Because we have $\tilde S \lot_\Gamma k \simeq S$, we have $\tor^\Gamma_i(\tilde S,k)\cong 0$ for $i>0$. Because $\Gamma$ is local Artinian, this implies Tor-vanishing for all $\Gamma$-modules, and hence $\tilde S$ is flat as a $\Gamma$-module. Hence $\tilde S$ is in the image of $\Phi$, and so $\Phi$ is a surjection and thus an isomorphism of sets. For a similar proof that the groupoid-valued deformation functors respect truncation, see \cite[2.5]{huakeller}.

\end{rmk}
During the proof of \cite[Theorem A]{dqdefm}, we make use of the following fact. Recall from \ref{ainfkd} the definition of the \textbf{Koszul dual} $E^!$ of an augmented dga $E \to k$: one takes the tensor coalgebra on the shifted augmentation ideal of $E$ (along with an extra differential, the bar differential) and dualises to get a dga $E^!$.
\begin{prop}[{\cite[4.5.1]{dqdefm}}]\label{dqiskd}
	The dga $\dca$ is the Koszul dual of the derived endomorphism algebra $\R\enn_A(S)$.
\end{prop}

The derived contraction algebra is defined complete locally. Can we compute it by using larger neighbourhoods of $p$? The first step is to examine how the dga $\R\enn(S)$ changes under completion.
\begin{prop}\label{globalrend}
	Suppose that we are in the situation of the Zariski local setup \ref{localsetup}. Let $\Lambda=\enn_R(R\oplus N)$ be the associated noncommutative model for $U$. Let $\hat \Lambda$ be the completion at $p$, and let $A$ be the basic algebra Morita equivalent to $\hat \Lambda$. Let $S_A$ (resp. $S_{\Lambda}$) be the simple $A$-module (resp. $\Lambda$-module) corresponding across the derived equivalence to $\mathcal{O}_C(-1)$. Then there is a quasi-isomorphism
	$$\R\enn_A(S_A) \simeq \R\enn_\Lambda(S_\Lambda).$$
\end{prop}
\begin{proof}
	This is the proof of \cite[3.9(3)]{contsdefs}: the idea is that derived endomorphisms of finite length modules supported at $p$ behave well under Morita equivalence and completion.
\end{proof}
When one can compute the contraction algebra in a Zariski neighbourhood, one can also compute the derived contraction algebra there:
\begin{prop}
	Suppose that we're in the local setup \ref{localsetup}. Suppose that $\Lambda_\con\coloneqq \Lambda/\Lambda e \Lambda$ is Artinian local. Then there is a quasi-isomorphism $\dca \simeq\Lambda/^\mathbb{L}\Lambda e \Lambda$.
\end{prop}
\begin{proof}We have $H^j(\Lambda/^\mathbb{L}\Lambda e \Lambda)\cong \ext^j_R(N,N)$. Exactly as in \ref{fdcohom}, because this Ext group is supported at $p$ it must be finite-dimensional. By hypothesis, $H^0(\Lambda/^\mathbb{L}\Lambda e \Lambda)\cong \Lambda_\con$ is Artinian local, and for $j>0$ the module $H^j(\Lambda/^\mathbb{L}\Lambda e \Lambda)\cong 0$ is certainly finite. So we may apply \cite[4.5.1]{dqdefm} to conclude that $\R\enn_\Lambda(S_\Lambda)^!\simeq \Lambda/^\mathbb{L}\Lambda e \Lambda$. But by \ref{globalrend}, we know that $\R\enn_\Lambda(S_\Lambda)$ is quasi-isomorphic to $\R\enn_A(S_A)$. Since the Koszul dual preserves quasi-isomorphisms, $\Lambda/^\mathbb{L}\Lambda e \Lambda$ is quasi-isomorphic to $\R\enn_A(S_A)^! \simeq \dca$.
\end{proof}
Now we may globalise:
\begin{prop}
	Suppose that we are in the situation of Setup \ref{globalsetup}. Let $\Lambda$ and $S_\Lambda$ be as above. Then there is a quasi-isomorphism $$\R\enn_\Lambda(S_\Lambda) \simeq \R\enn_{X}(\mathcal{O}_C(-1))$$
\end{prop}
\begin{proof}This is the proof of \cite[3.9(1)]{contsdefs}. Let $i:U \to X$ be the base change of the inclusion $\spec R \into X_\con$; it's an affine map because it's a base change of an affine map. It's also an open embedding, and hence induces an embedding $D(U)\into D(X)$. Hence we get quasi-isomorphisms $$\R\enn_\Lambda(S_\Lambda) \xrightarrow{\simeq} \R\enn_{U}(\mathcal{O}_C(-1)) \xrightarrow{\simeq} \R\enn_{X}(\mathcal{O}_C(-1))$$
	as required. \end{proof}
We obtain the immediate corollary:
\begin{thm}In the situation of Setup \ref{globalsetup}, the derived contraction algebra can be computed as the Koszul dual of the dga $\R\enn_{X}(\mathcal{O}_C(-1))$. In particular, it is intrinsic to the contraction $\pi$, and does not depend on any choice of affine neighbourhood of $p$ or tilting bundle.
\end{thm}

\begin{rmk}
	We note for computational purposes that $\dca$ is invariant under derived equivalences. Specifically, if $D(A)$ is equivalent to $D(A')$ in such a way that $S$ gets sent to $S'$, then one gets a quasi-isomorphism $\R\enn_A(S)\simeq\R\enn_{A'}(S')$, and hence quasi-isomorphisms $$\dca\simeq \R\enn_A(S)^!\simeq \R\enn_{A'}(S')^!.$$
\end{rmk}

	\subsection{Ginzburg dgas}
Hua and Keller \cite{huakeller} show that, for a flopping contraction in a smooth threefold $X$, one may compute its derived contraction algebra as a certain Ginzburg dga associated to an NCCR of $X$. We briefly recall the definition of Ginzburg dgas and review Hua and Keller's result.

\p Ginzburg dgas, as well as Calabi--Yau algebras, were first defined in the seminal paper \cite{ginzburgdga}. We are following the sign conventions of \cite{vdbginzburg}.

\begin{defn}
	Let $Q$ be a finite quiver. A \textbf{superpotential} on $Q$ is an element of the cocentre $HH_0(Q)\cong\frac{kQ}{[kQ,kQ]}$.
\end{defn}	

\begin{defn}
	Let $Q$ be a quiver and $W$ a superpotential on $Q$. Let $a$ be an arrow in $Q$. Define the \textbf{cyclic derivative} $\partial_aW$ by the formula
	$$\partial_aW\coloneqq \sum_{W=uav}vu$$where we sum over all arrows $u,v$ such that $W=uav$.
\end{defn}

\begin{defn}
	Let $Q$ be a finite quiver and $W$ a superpotential on $Q$. Let $\bar Q$ be the graded quiver with the same vertex set as $Q$, and with three types of arrows: the arrows $a$ from $Q$, all of degree zero, a reversed degree $-1$ arrow $a^*$ for every arrow $a$ in $Q$, and a loop $z_i$ of degree $-2$ at every vertex $i$. The \textbf{Ginzburg dga} associated to the pair $(Q,W)$ is the dga $\Pi(Q,W)$ with underlying graded algebra the path algebra of $k\bar Q$, and with differential given by \begin{align*}
	da&=0
	\\da^*&=-\partial_aW
	\\dz_i&=\sum_ae_i[a,a^*]e_i
	\end{align*}
	Note that we may compute $dz_i$ by summing over only the arrows incident to $i$.
\end{defn}
\begin{defn}
	Let $Q$ be a finite quiver and $W$ a superpotential on $Q$. The \textbf{Jacobi algebra} is the algebra $H^0(\Pi(Q,W))$. It can be computed as the path algebra of $kQ$ modulo the relations given by the $\partial_a W$.
\end{defn}
If one thinks of a Jacobi algebra as simply the path algebra of a quiver with relations, one can think of the superpotential as `integrating' the relations.

\p Recall that to compute the contraction algebra, Donovan and Wemyss take a noncommutative algebra $A=\enn_R(R\oplus M)$ and quotient it by the idempotent $e=\id_R$. Presenting $A$ as the path algebra of a quiver $Q$ with relations $r_i$, one can hence compute $A_\con$ by simply throwing out the vertex corresponding to $R$ and modifying the relations accordingly. If the quiver admits a superpotential $W$ making $A$ into the associated Jacobi algebra, then one can compute $A_\con$ by taking the Jacobi algebra of the one-vertex quiver $Q'$ obtained by deleting the vertex at $R$ from $Q$, equipped with modified superpotential $W'$ obtained by removing all the arrows that are not loops at $M$. See \ref{lauferrmk} for an example of such a computation. Since one hence has $H^0(\dca)\cong H^0(\Pi(Q',W'))$, one might wonder if one can compute $\dca$ as a Ginzburg dga, and under some smoothness hypotheses Hua and Keller proved that this is indeed the case for threefolds:
\begin{thm}
	Let $X \to \spec R$ be a threefold flopping contraction with $X$ smooth. Let $(Q',W')$ be the quiver with superpotential defined above that computes $A_\con$. Then $\dca$ is quasi-isomorphic to the Ginzburg dga $\Pi(Q',W')$.
\end{thm}
\begin{proof}
	By \cite[4.17]{huakeller}, there is a quasi-isomorphism $$\Pi(Q',W')\simeq \tau_{\leq 0}\R\underline\enn_R(M)$$between the Ginzburg dga and the truncation to nonpositive degrees of the dg endomorphism algebra of the module $M$ considered as an object of the dg singularity category of $R$. But this truncation is quasi-isomorphic to $\dca$ by \cite[5.1.11]{dqdefm}
\end{proof}

\begin{rmk}
	Let $A$ be the Jacobi algebra of a quiver with superpotential $W$ and let $S$ be the direct sum of the simple vertex modules. Let $\hat A$ be the completion of $A $ along its augmentation ideal. Using deformation theory, Segal \cite[\S2]{segaldefpt} proves that under some finiteness conditions one can recover $\hat A$ from the Ext-algebra $\ext_A(S,S)$ along with the higher multiplications $m_r$ needed to make it quasi-isomorphic to $\R\enn_A(S)$. More precisely, he identifies $\hat A$ as $H^0(\R\enn_A(S)^!)$, which is analogous to our isomorphism $A/AeA\cong H^0(\dq)$. Under the presence of an additional Calabi--Yau condition, Segal also proves \cite[3.3]{segaldefpt} that one can recover $W$ from the $m_r$: loosely, one uses the CY pairing to obtain a superpotential on the completion of the tensor algebra $T(\ext^1_A(S,S)^*)$, and this presents $\hat A$.
\end{rmk}

\section{Threefold examples}\label{threecomps}
In this section we do some computations. We'll compute the derived contraction algebra associated to the $n$-Pagoda flop, with base $xy-(u+v^n)(u-v^n)$. As a warmup, we'll do the case $n=1$, which is the classical Atiyah flop. We'll also sketch the computation of the derived contraction algebra associated to the Laufer flop. We'll use the interpretation of the derived contraction algebra as a Koszul dual to actually do the computations. We'll also see that, in the situations we study, one can compute the derived contraction algebra as a certain Ginzburg dga (see \cite{vdbginzburg} for a definition) of a quiver with potential that one obtains by deleting vertices of a quiver representation for a noncommutative model of the contraction. For references on the $A_\infty$-algebra methods we use, see the Appendix, specifically \ref{minmods}.
\subsection{General setup}
The setup will be the following variation on that of \cite[2.9]{hmmp}:
\begin{setup}\label{flopsetup}
	$\pi:U \to \spec R$ is a crepant projective birational morphism between two noetherian normal integral threefolds over an algebraically closed field $k$ of characteristic zero. Moreover, $\pi$ is an isomorphism away from a single closed point $p$ in the base, where $C\coloneqq \pi^{-1}(p)$ is an irreducible rational (possibly non-reduced) floppable curve. We also assume that $R$ is Gorenstein and that $U$ has at worst Gorenstein terminal singularities.
\end{setup}
\begin{prop}
	Suppose that we are in the situation of Setup \ref{flopsetup}. Then we are in the situation of the Zariski local setup \ref{localsetup}. In particular we may define the derived contraction algebra.
\end{prop}
\begin{proof}
	We need to check that the completion $\hat{R}_p$ is an isolated hypersurface singularity and that $\R\pi_*\mathcal{O}_{X}\simeq R$. The first claim follows from the classification of terminal singularities; namely by \cite[5.38]{kollarmori}, $p$ is an isolated cDV singularity and in particular a hypersurface. The second claim follows from Kawamata vanishing \cite{kawamatavanish}.
\end{proof}
\begin{prop}
Suppose that we are in the Threefold Setup \ref{flopsetup}. If $U$ is $\Q$-factorial (i.e. is a minimal model of $R$) then the cohomology of $\dca$ is simply $A_\con[\eta]$.
\end{prop}
\begin{proof}
	By \cite[4.10]{hmmp}, the associated MCM $\hat R$-module $M$ used to define $\dca$ is rigid, and the claim now follows from \ref{dqcohom} and \ref{etaex}. In fact, the same holds more generally if $X$ is a `partial minimal model' \cite[4.13]{hmmp}.
\end{proof}
In general, our strategy will be to compute $\dca$ as a Koszul double dual. We'll start with Setup \ref{flopsetup} and write down a noncommutative model $\Lambda=\enn_R(R\oplus N)$ for $U$, presenting $\Lambda$ as the path algebra of a two-vertex quiver with relations. Letting $S$ be the simple of $\Lambda$ at the vertex corresponding to $N$, we'll compute the derived endomorphism algebra $\R\enn_\Lambda(S)$. It's easiest to do this as an $A_\infty$-algebra, by placing higher multiplications on $\ext_\Lambda(S,S)$ and appealing to Kadeishvili's theorem (cf. \ref{kadeish}). Next, we'll compute $\dca$ as the $A_\infty$ Koszul dual of $\R\enn_\Lambda(S)$. Again, we'll do this via Kadeishvili's theorem, along with the fact that the cohomology of {\dca} can be calculated explicitly in advance. In order to actually do the $A_\infty$ computations, we'll either use Merkulov's construction, or grade the resolution of $S$ and note that many higher multiplications become ruled out, which itself appeals to a graded version of Merkulov's construction (cf. \ref{minmods}). We'll often use Massey products to detect non-formality (\ref{masseys}, \ref{masseylemma}). 

\p Suppose that $\Gamma$ is a quiver (possibly with relations) and $A=P(\Gamma)$ its path algebra. We denote multiplication in $A$ left-to-right: that is, $ab$ means `follow edge $a$, then edge $b$'. If $i$ is a vertex of $\Gamma$, with associated idempotent $e_i \in A$, then the \textbf{projective at $i$} is the right $A$-module $P_i\coloneqq  e_i A$ consisting of all those paths starting at $i$.

\p The noncommutative model $\Lambda$ we write down for $U$ will be a \textbf{noncommutative crepant resolution} (\textbf{NCCR}) of $R$ \cite{vdbnccr}. Because all crepant resolutions of threefolds with terminal singularities are derived equivalent \cite[6.6.3]{vdbnccr}, exhibiting a single NCCR will suffice for our calculations.

\subsection{The Atiyah flop} \label{atiyahflop}
The Atiyah flop is the simplest example of a flopping contraction, and was historically the first discovered \cite{atiyahflop}. The base is the cone ${R=\frac{k[u,v,x,y]}{(uv-xy)}}$. One MCM module is $(u,x)$, and it is well known that this gives an NCCR $\Lambda$ with a presentation as the path algebra of the following quiver with relations:
$$   
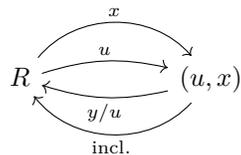
\begin{tikzcd}
{{R}} \arrow[rr,bend left=15,"u"] \arrow[rr,bend left=50,"x"]  && (u,x) \arrow[ll,bend left=15,"y/u"] \arrow[ll,bend left=50,"\text{incl.}"]
\end{tikzcd}$$
\vfill\pagebreak 
Relabelling, we can write this quiver abstractly as \newline
\begin{wrapfigure}[4]{L}{0.5\textwidth}
	\centering
	\vspace{-15pt}$
	\begin{tikzcd}
	1 \arrow[rr,bend left=15,"b"] \arrow[rr,bend left=50,"a"]  && 2 \arrow[ll,bend left=15,"s"] \arrow[ll,bend left=50,"t"]
	\end{tikzcd}$
\end{wrapfigure} 

$asb=bsa$

$sbt=tbs$

$atb=bta$

$sat=tas$\newline \newline
One can check that a resolution for $S$, the simple at 2, is given by 
$$
P_2 \xrightarrow{\left(\begin{smallmatrix} b \\ -a \end{smallmatrix}\right)} P_1^2 \xrightarrow{\left(\begin{smallmatrix} bt & at \\ -bs & -as \end{smallmatrix}\right)} P_1^2 \xrightarrow{(s \; t)} P_2$$and it easily follows that the Ext-algebra of $S$ is $k[x]/x^2$, with $x$ placed in degree 3. It's also easy to see that $\R\enn_\Lambda(S)$ must be formal. So the derived contraction algebra $\dca$ is the Koszul dual of $k[x]/x^2$, which is the same as the tensor algebra on a single element $\eta=x^*$ of degree -2. Hence, $\dca \simeq k[\eta]$.

\subsection{Pagoda flops}
The Pagoda flops, first defined by Reid in \cite{reidpagoda}, are natural generalisations of the Atiyah flop. They're all $cA_1$ singularities, and fit into a family parametrised by a natural number $n\geq 1$. The $n=1$ case is the Atiyah flop. The main theorem of this section is the computation \ref{dcapagoda}. The base of the Pagoda flop is $R=\frac{k[u,v,x,y]}{(uv-(x+y^n)(x-y^n))}$. One MCM module for $R$ is $N\coloneqq (u,x+y^n)$. One can write down a quiver presentation for the resulting NCCR $\Lambda=\enn_R(R\oplus N)$: it looks like $$\begin{tikzcd}
R \arrow[rr,bend left=15,"u"]\arrow[loop left, "y"] \arrow[rr,bend left=50,"x+y^n"]  && N\arrow[loop right, "y"] \arrow[ll,bend left=15,"\text{inc.}"] \arrow[ll,bend left=50,"\frac{x-y^n}{u}"]\end{tikzcd}$$ 
Rewriting this abstractly, we obtain the quiver with relations
\begin{wrapfigure}[6]{L}{0.5\textwidth}
	\centering $
	\begin{tikzcd}
	\cdot_1 \arrow[rr,bend left=15,"b"]\arrow[loop left, "l"] \arrow[rr,bend left=50,"a"]  && \cdot_2 \arrow[ll,bend left=15,"s"] \arrow[ll,bend left=50,"t"]\arrow[loop right, "m"]
	\end{tikzcd}$
\end{wrapfigure} 

$la=am$

$lb=bm$

$sl=ms$

$tl=mt$

$as=bt+2l^n$

$sa=tb+2m^n$
\newline  \newline  where the simple $S$ at 2 corresponds to the exceptional locus in the resolution. So to compute \dca, we first need to resolve $S$. Assuming $n>1$, then using a tedious basis argument one can write down a four-term resolution $$ P_2 \xrightarrow{\left(\begin{smallmatrix}-a\\b\\m\end{smallmatrix}\right)} P_1^2P_2 \xrightarrow{\left(\begin{smallmatrix}s & t & 2m^{n-1}\\ -l & 0 & -a \\ 0 & -l & b\end{smallmatrix}\right)} P_2P_1^2 \xrightarrow{(m,s,t)} P_2$$of $S$. It's now easy to see that the $\ext$-algebra of $S$ is four-dimensional over $k$, with Hilbert series $1+t+t^2+t^3$. In fact, one can check that the $\ext$-algebra is generated by two elements $f_1$, $f_2$, with $f_i$ placed in degree $i$. Concretely, $f_1$ is represented by $$\begin{tikzcd}[row sep= large, column sep=15ex] & P_2\ar[r]\ar[d,"{\left(\begin{smallmatrix}0\\0\\1\end{smallmatrix}\right)}"] & P_1^2P_2\ar[r]\ar[d,"{-\sthbthm{0}{0}{2m^{n-2}}{1}{0}{0}{0}{1}{0}}"] & P_2P_1^2\ar[r]\ar[d,"{\left(1,0,0\right)}"] & P_2 \\ P_2\ar[r] & P_1^2P_2\ar[r] & P_2P_1^2\ar[r] & P_2 &\end{tikzcd}$$
while $f_2$ is represented by $$\begin{tikzcd}[row sep= large, column sep=large] & & P_2\ar[r]\ar[d,"{\left(\begin{smallmatrix}1\\0\\0\end{smallmatrix}\right)}"] & P_1^2P_2\ar[r]\ar[d,"{\left(0,0,1\right)}"] & P_2P_1^2\ar[r] & P_2 \\ P_2\ar[r] & P_1^2P_2\ar[r] & P_2P_1^2\ar[r] & P_2 & &\end{tikzcd}$$and $f_3=f_1f_2$ is represented by the identity map between the two copies of $P_2$ at the ends. One can check that $f_1$ and $f_2$ strictly commute, and that $f_2$ genuinely squares to zero (purely for degree reasons). However, one can check that $f_1^2=-2m^{n-2}f_2$, which is merely homotopic to zero (if $n>2$), not identically zero.

\p So if $n=2$ then the $\ext$-algebra is $k[f_1]/f_1^4$, whereas if $n>2$ then it's isomorphic to the algebra $\frac{k[f_1,f_2]}{(f_1^2,\;f_2^2)}$. As for the Atiyah flop, when $n=2$, $\R\enn_\Lambda(S)$ is formal. In general, the derived endomorphism algebra will not be formal; we'd like to use Merkulov's construction (or a variant) to work out the higher $A_\infty$ operations on the Ext-algebra. We'll grade the resolution of $S$ in order to eliminate many of these: the point is that one can apply the Adams graded version of Merkulov's construction to the graded algebra $\R\enn_\Lambda(S)$, and since the higher multiplication $m_r$ must have bidegree $(r-2,0)$, this allows us to conclude that many of the $m_r$ must be zero. 

\p One can put a grading on $\Lambda$, with $l,m$ in degree 2 and $a,b,s,t$ in degree $n$. We'll refer to the degree of a homogeneous element of $\Lambda$ in this secondary grading as its \textbf{weight}, since we're already using `degree' to refer to maps. Recall that $S$ has projective resolution $$P_2 \to P_1^2P_2 \to P_2P_1^2 \to P_2.$$The following lemma is proved simply by working out the computations:
\begin{lem}
	Suppose that $n \geq 2$. Give $S$ the trivial grading as a $\Lambda$-module. Then the resolution of $S$ admits a secondary grading, compatible with the gradings on $\Lambda$ and $S$, with weights $$(2n+2) \to (n+2,n+2,2n)\to(2,n,n) \to (0)$$Moreover, the differential has weight zero, $f_1$ has weight $-2$, and $f_2$ has weight $-2n$. 
\end{lem}

For purely degree reasons, the nontrivial higher multiplications $m_r$ on $\frac{k[f_1,f_2]}{(f_1^2,\;f_2^2)}$ must (up to permutation) all be of the form $m_r(f_1,\ldots,f_1)=\lambda_r f_2$ or $m_r(f_1,\ldots,f_1,f_2)=\mu_{r}f_1f_2$. We see that for the weights to match up, the only nonzero coefficients are possibly $\lambda_n$ and $\mu_2$. We already know that $\mu_2=1$, since $m_2$ is the multiplication. Now we know what to look for, we can prove:
\begin{lem}Let $n \geq 2$. Then $-2f_2$ is an element of the Massey product $\langle f_1,\ldots,f_1\rangle_n$.
\end{lem}
\begin{proof}
	For $0\leq r \leq n$, let $h_r$ be the degree $1$, weight $2-2r$ map $$\begin{tikzcd}[row sep= large, column sep=15ex] & P_2\ar[r]\ar[d,"0"] & P_1^2P_2\ar[r]\ar[d,"{\sthbthm{0}{0}{-2m^{n-r}}{0}{0}{0}{0}{0}{0}}"] & P_2P_1^2\ar[r]\ar[d,"0"] & P_2 \\ P_2\ar[r] & P_1^2P_2\ar[r] & P_2P_1^2\ar[r] & P_2 & \end{tikzcd}$$One can check that $h_3$ is a homotopy from $f_1^2$ to 0, motivating the definition of $h_r$. One can in fact compute $dh_{r+1}=-2m^{n-r}f_2$. Note that the $h_r$'s are all orthogonal, in the sense that $h_rh_{r'}=0$ for all $r,r'$. Furthermore, one can check that $f_1h_r+h_rf_1=-2m^{n-r}f_2=dh_{r+1}$; this is the crucial relation.
	
	\p Now we need to use \ref{masseylemma}. Set $b_1=f_1$ and $b_r=h_{r+1}$ for $2\leq r<n$. Then we have \linebreak $b_1b_{i-1}+\cdots+b_{i-1}b_1=b_1b_{i-1}+b_{i-1}b_1$ since the middle terms are of the form $h_rh_{r'}$. Hence, we have $b_1b_{i-1}+\cdots+b_{i-1}b_1=db_i$. So, finally we can compute $b_1b_{n-1}+\cdots+b_{n-1}b_1=-2f_2$ as claimed.
\end{proof}
\begin{cor}
	If $n>2$, then $\R\enn_\Lambda(S)$ is not formal.
\end{cor}
We sum up our study of the dga $\R\enn_\Lambda(S)$ as $n$ varies:
\begin{prop}
If $n=2$, then $\R\enn_\Lambda(S)$ is a formal dga, quasi-isomorphic to the graded algebra $k[f_1]/f_1^4$ with $f_1$ in degree 1. If $n>2$, then $\R\enn_\Lambda(S)$ is quasi-isomorphic to the strictly unital minimal $A_\infty$-algebra with underlying graded algebra $\frac{k[f_1,f_2]}{(f_1^2,\;f_2^2)}$ with $f_i$ placed in degree $i$, and a single nonzero higher multiplication given by $m_n(f_1,\ldots,f_1)=f_2$.
\end{prop}
\begin{proof}
We know the statement for $n=2$ already. If $n>2$, then since $\R\enn_\Lambda(S)$ is not formal, there must exist at least one nonzero higher multiplication. But by the above, the only candidate is $m_n(f_1,\ldots,f_1)=\lambda f_2$ for some $\lambda\neq 0$. Rescaling $f_2$ as necessary, we may choose $\lambda=1$.
\end{proof}
\begin{rmk}
	A way of stating this that depends less on $n$ is to say that $\R\enn_\Lambda(S)$ is the strictly unital minimal $A_\infty-\frac{k[f_2]}{f_2^2}$-algebra generated by a single element $f_1$ subject to the relations $m_r(f_1,\ldots,f_1)=\delta_{r,n}f_2$.
	\end{rmk}

\begin{prop}\label{pagodadga}
	Suppose that $n \geq 2$. As a noncommutative dga, $\dca$ is freely generated by generators $\xi$, $\zeta$, $\theta$, with degrees $0,-1,-2$ and weights $2,2n,2n+2$ respectively. The differential is given by $d\theta=[\xi,\zeta]$, $d\zeta=\xi^n$, and $d\xi=0$.
\end{prop}
\begin{proof}
	This is the definition of the Koszul dual dga: take the three basis elements $f_1$, $f_2$, $f_3$ of the augmentation ideal of $\R\enn_\Lambda(S)$, dualise (we put $\xi=f_1^*,\zeta=f_2^*,\theta=f_3^*$), and shift. The differential $d(x^*)$ is the signed sum of the products $x_1^*\cdots x_r^*$ such that $d(x_1 | \cdots | x_r)=x$, where $d$ denotes the $A_\infty$ bar differential.
\end{proof}
\begin{rmk}
	As in \cite[3.3]{amquivers}, one can check that the relations in $\Lambda$ come from equipping the quiver with the superpotential $\overline{W}\coloneqq las-lbt-ams+bmt-\frac{2}{n+1}l^{n+1}+\frac{2}{n+1}m^{n+1}$ and taking the associated Jacobi algebra. In order to compute $A_\con$, one simply takes the subquiver $\begin{tikzcd} \cdot_2 \ar[loop right, "m"]\end{tikzcd}$ and equips it with the modified superpotential $W\coloneqq \frac{2}{n+1}m^{n+1}$. One can easily check that the Ginzburg dga associated to $W$ is precisely the dga appearing in \ref{pagodadga}.
\end{rmk}

Observe that $H^0(\dca)\cong k[\xi]/\xi^n \cong A_\con$, as expected. It's also not too difficult to compute $H^{-1}(\dca)$: the only elements in degree -1 are noncommutative polynomials in $\xi, \zeta$ with exactly one occurrence of $\zeta$. Noting that $\xi$ is a cocycle and $d\theta=[\xi,\zeta]$, we see that such a polynomial is homotopic to one of the form $p=\zeta\sum_i a_i \xi^i$. But $dp=\sum_i a_i \xi^{i+n}$, and this is zero if and only if $p=0$. So $H^{-1}(\dca)\cong 0$. Hence we find that $H(\dca)$ is zero in odd degrees, and $A_\con$ in nonpositive even degrees. We can now prove:
\begin{lem}
	The algebra map $\dca \to \dca/(\theta, d\theta )$ is a quasi-isomorphism.
\end{lem}
\begin{proof}
	It's easy to check that the dga $\dca/(\theta, d\theta )$ is isomorphic to the dga $\frac{k[\xi]\langle\zeta\rangle}{\xi\zeta-\zeta\xi}$ with $d\zeta=\xi^n$. The cohomology algebra of this dga is $\frac{k[\xi]}{\xi^n}[\eta]$, where $\eta=\zeta^2$ is a degree -2 element. In particular, the cohomology algebra is levelwise isomorphic to that of $\dca$, so in order to check that the quotient map $\dca \to \dca/(\theta, d\theta )$ is a quasi-isomorphism we only need to check that it is a quasi-surjection. But one can check that $\zeta^2+\Sigma_{i=1}^n\xi^{i-1}\theta\xi^{n-i} \in \dca$ is a cocycle, and maps to $\zeta^2$ in the quotient.
\end{proof}
\begin{rmk}
	The expression $\zeta^2+\Sigma_{i=1}^n\xi^{i-1}\theta\xi^{n-i}$ comes from using $d\theta=\xi\zeta-\zeta\xi$ to repeatedly commute $\xi$ with $\zeta$ in $d(\zeta^2)=\xi^n\zeta-\zeta\xi^n$. One can also check that this cocycle is homogenous of weight $4n$.
\end{rmk}
\begin{thm}\label{dcapagoda}
	The derived contraction algebra associated to the $n$-Pagoda flop is quasi-isomorphic to the strictly unital minimal $A_\infty$-algebra $\frac{k[\xi]}{\xi^n}[\eta]$ with $\xi$ in degree $0$, $\eta$ in degree $-2$, and higher multiplications given by $$m_r(\eta^{i_1}\xi^{j_1} \otimes \cdots\otimes \eta^{i_r}\xi^{j_r}) = \begin{cases} -(-1)^{\frac{r}{2}}C_{\frac{r}{2}}\eta^{i + \frac{r}{2}-1}\xi^{j -n(r-2)} & r\text{ is even and } n(r-2)\leq j < n(r-1)\\ 0 & \text{otherwise}
	\end{cases}$$where we put $i=i_1 + \cdots + i_r$ and $j=j_1+\cdots+j_r$ and the $C_p$ are the (shifted) Catalan numbers with $C_1=1$, $C_2=1$, $C_3=2$, $C_4=5$, et cetera.
\end{thm}
\begin{proof}
	By the above, $\dca$ is quasi-isomorphic to the dga $C\coloneqq \frac{k[\xi]\langle\zeta\rangle}{\xi\zeta-\zeta\xi}$ with $d\zeta=\xi^n$. We know that the cohomology of $C$ is $\frac{k[\xi]}{\xi^n}[\eta]$, where $\eta=\zeta^2$. We use Merkulov's construction to augment $C$ with higher multiplications $m_r$ inducing an $A_\infty$ quasi-isomorphism with $\dca$. One linear section of the projection map $\pi: C \to HC$ is the map $\sigma$ that in odd degrees is zero, and in even degrees sends $\eta^i\xi^j$ to $\eta^i\xi^j$ if $j<n$, and zero otherwise. Composing the projection with the section yields the linear endomorphism of $C$ given by $$\sigma \pi\quad = \quad\begin{tikzcd} \cdots \ar[r,"0"] & k[\xi] \ar[r,"\xi^n"] \ar[d,"0"] & k[\xi] \ar[r,"0"] \ar[d,"\epsilon"] & k[\xi] \ar[r,"\xi^n"] \ar[d,"0"]& k[\xi] \ar[d,"\epsilon"] \\ \cdots \ar[r,"0"] & k[\xi] \ar[r,"\xi^n"] & k[\xi] \ar[r,"0"] & k[\xi] \ar[r,"\xi^n"] & k[\xi]  \end{tikzcd}$$ where $\epsilon(\xi^i)$ is $\xi^i$ if $i<n$, and 0 otherwise. Firstly, we need to construct a homotopy $h: \id_C \to \sigma\pi$. Interpreting a negative power of $\xi$ as 0, one can check that the (periodically extended) homotopy given by $h_1(\xi^i)= \xi^{i-n}$ and $h_2=0$ works: $$\begin{tikzcd} \cdots \ar[r,"0"] & k[\xi] \ar[r,"\xi^n"] \ar[d,"0"] \ar[dl,swap,"h_2"] & k[\xi] \ar[r,"0"] \ar[d,"\epsilon"] \ar[dl,swap,"h_1"] & k[\xi] \ar[r,"\xi^n"] \ar[d,"0"]\ar[dl,swap,"h_2"]& k[\xi] \ar[d,"\epsilon"]\ar[dl,swap,"h_1"] \\ \cdots \ar[r,"0"] & k[\xi] \ar[r,"\xi^n"] & k[\xi] \ar[r,"0"] & k[\xi] \ar[r,"\xi^n"] & k[\xi]  \end{tikzcd}$$In other words, we have $h(\eta^i\xi^j)=\zeta\eta^i\xi^{j-n}$. Now we put inductively $$m_r\coloneqq \sum_{s+t=r}(-1)^{s+1}(s,t)$$where for brevity I write $(s,t)\coloneqq m_2(hm_s\otimes hm_t)$ and $hm_1\coloneqq -\id_A$. Then Merkulov's theorem tells us that $HC$, augmented with the $m_r$, is $A_\infty$ quasi-isomorphic to $C$. First I claim that when $r>1$ is odd, then $m_r$ vanishes: this is clear for degree reasons. From now on we may assume that $r$ is even, and we can conclude that the sum defining $m_r$ reduces to $m_r=-\sum_{s+t=r}(s,t)$. I next claim that the maps $m_r$ are $\eta$-linear, in the sense that $$m_r(\eta^{i_1}\xi^{j_1} \otimes \cdots\otimes \eta^{i_r}\xi^{j_r}) = \eta^{i_1 + \cdots + i_r}m_r(\xi^{j_1}\otimes \cdots\otimes \xi^{j_r})$$holds. This is not hard to see inductively: it's clearly true for $m_2$. Suppose that it's true for all $r'<r$, and suppose $s+t=r$ with $s,t>0$. Then \begin{align*}
	&(s,t)(\eta^{i_1}\xi^{j_1} \otimes \cdots\otimes \eta^{i_r}\xi^{j_r})
	\\&=m_2(hm_{s}(\eta^{i_1}\xi^{j_1} \otimes \cdots\otimes \eta^{i_{s}}\xi^{j_{s}})\otimes hm_{t}(\eta^{i_{r-t+1}}\xi^{j_{r-t+1}} \otimes \cdots\otimes \eta^{i_r}\xi^{j_r}))\\&=m_2(h(\eta^{i_1+\cdots+ i_s}m_{s}(\xi^{j_1} \otimes \cdots\otimes \xi^{j_{s}}))\otimes h(\eta^{i_{r-t+1}+\cdots+ i_r} m_{t}(\xi^{j_{r-t+1}} \otimes \cdots\otimes \xi^{j_r})))\\&=m_2(\eta^{i_1+\cdots+ i_s}hm_{s}(\xi^{j_1} \otimes \cdots\otimes \xi^{j_{s}})\otimes \eta^{i_{r-t+1}+\cdots+ i_r}h m_{t}(\xi^{j_{r-t+1}} \otimes \cdots\otimes \xi^{j_r}))\\&=\eta^{i_1+\cdots+ i_r}m_2(hm_{s}(\xi^{j_1} \otimes \cdots\otimes \xi^{j_{s}})\otimes h m_{t}(\xi^{j_{r-t+1}} \otimes \cdots\otimes \xi^{j_r}))\\&=\eta^{i_1+\cdots+ i_r}(s,t)(\xi^{j_1} \otimes \cdots\otimes \xi^{j_r})
	\end{align*}
	where the second and sixth line follow because all degrees of elements are even, the third line follows from the induction hypothesis, the fourth uses that $h(\eta x)=\eta h(x)$, and the fifth uses centrality of $\eta$ as well as collating powers of $\eta$. The claim now follows by adding all of these up using $m_r=-\sum_{s+t=r}(s,t)$. Now all that needs to be done is determine $m_r(\xi^{j_1} \otimes \cdots\otimes \xi^{j_r})$ when $r>2$ is even. First observe that $h(\xi^i)=\zeta\xi^{i-n}$, where we again interpret a negative power of $\xi$ as zero. Let $C_p$ be the $p^\text{th}$ Catalan number, with indexing starting from $p=1$; so $C_1=1$, $C_2=1$, $C_3=2$, $C_4=5$, et cetera. The important point for us will be that $C_p=\sum_{s+t=p}C_sC_t$ and $C_1=1$, where in the sum we require that $s$ and $t$ are positive integers. For even $r>0$, put $C'_r\coloneqq -(-1)^{\frac{r}{2}}C_{\frac{r}{2}}$. I claim that $m_r(\xi^{j_1} \otimes \cdots\otimes \xi^{j_r})=C'_r\zeta^{r-2}\xi^{j_1+\cdots+j_r -n(r-2)}$. Certainly this holds for $r=2$. Inductively as before, one sees that for $2s+2t=r$, the expression $(2s,2t)(\xi^{j_1} \otimes \cdots\otimes \xi^{j_r})$ is exactly $C'_sC'_t\zeta^{r-2}\xi^{j_1+\cdots+j_r -n(r-2)}$. Hence, since we may compute the sum defining $m_r$ by summing only over even terms, we get \begin{align*}
	m_r(\xi^{j_1} \otimes \cdots\otimes \xi^{j_r})&=-\sum_{2s+2t=r}(2s,2t)
	(\xi^{j_1} \otimes \cdots\otimes \xi^{j_r})\\&=\left(-\sum_{2s+2t=r}C'_{2s}C'_{2t}\right)\zeta^{r-2}\xi^{j_1+\cdots+j_r -n(r-2)}\\&=\left(-\sum_{s+t=\frac{r}{2}}-(-1)^sC_{s}\cdot -(-1^t)C_{t}\right)\zeta^{r-2}\xi^{j_1+\cdots+j_r -n(r-2)}
	\\&=C'_r\zeta^{r-2}\xi^{j_1+\cdots+j_r -n(r-2)}
	\end{align*}where the final line follows by the identity defining the Catalan numbers. Putting everything together, using $\eta=\zeta^2$, and recalling that we interpret a negative power of $\xi$ as zero, we obtain the required identities. Note that $n(r-2)\leq j$ is required to make the exponent of $\xi$ positive, and $j<n(r-1)$ is required to make it less than $n$ (so one could drop this condition if necessary).
\end{proof}

\begin{rmk}
Note that the higher multiplications are all $\eta$-linear, so another way to state the above is to say that the derived contraction algebra $\dca$ is quasi-isomorphic to the strictly unital minimal $A_\infty-k[\eta]$-algebra generated by $\xi$ subject to the relations $\xi^n=0$ and $$m_r(\xi^{j_1} \otimes \cdots\otimes \xi^{j_r}) = \begin{cases} -(-1)^{\frac{r}{2}}C_{\frac{r}{2}}\eta^{ \frac{r}{2}-1}\xi^{j -n(r-2)} & r\text{ is even and } n(r-2)\leq j < n(r-1)\\ 0 & \text{otherwise}\end{cases}$$
\end{rmk}

\subsection{The Laufer Flop}\label{lauferex}
We sketch a computation of the derived contraction algebra associated to the Laufer flop, which is a $D_4$ flop with (completed) base $\frac{k\llbracket x,y,u,v \rrbracket}{u^2+v^2y - x(x^2+y^3)}$ first appearing in \cite{laufer}. Following \cite{amquivers}, one noncommutative model is the algebra $A$ given by the (completion of) the following quiver with relations:\\

\begin{wrapfigure}[4]{L}{0.5\textwidth}
	\centering
	\vspace{-15pt}$
	\begin{tikzcd}
	1  \arrow[rr,bend left=30,"a"]  && 2  \arrow[ll,bend left=30,"b"] \arrow[in=20, out=70, distance=7mm,"x"] \arrow[out=-20, in=-70, distance=7mm,"y"]
	\end{tikzcd}$
\end{wrapfigure} 

$ay^2=-aba$

$y^2b=-bab$

$xy=-yx$

$x^2+yba+bay=y^3$
\\\hfill\\
One can check that $A$ admits a nontrivial grading with $a,b,y$ of weight 2 and $x$ of weight 3. The simple $S$ at the vertex 2 has a resolution (compatible with the weights) given by $$P_2 \xrightarrow{\left( \begin{smallmatrix} x \\ a \\ y \end{smallmatrix} \right)\cdot} P_2P_1P_2 \xrightarrow{\sthbthm{0}{ab}{ay}{y}{0}{x}{x}{yb}{ba-y^2}\cdot} P_1P_2P_2 \xrightarrow{( b ,  x , y )\cdot} P_2$$and using this one can check that the Ext-algebra of $S$ is of the form $$\ext^*_A(S,S)\cong\frac{k[g]\langle f\rangle}{f^3, fg-gf}$$ where the generators $g$ and $f$ both have degree 1 and weights $-2$ and $-3$ respectively (and in particular $g$ is square-zero). This is a six-dimensional algebra, with basis $\{1,g,f, gf, f^2, gf^2\}$. A computation with weights tells us that the only possible nontrivial higher Massey product for $\ext^*_A(S,S)$ is of the form $\langle g,g,g \rangle=\lambda f^2$. Picking an explicit lift $G$ of $g$ and a homotopy $H:G^2 \xrightarrow{\simeq} 0$ allows one to show that $f^2=[GH+HG]\in \langle g,g,g \rangle$. Hence $\R\enn_A(S)$ is not formal, and so must be quasi-isomorphic to the strictly unital minimal $A_\infty$-algebra with underlying graded algebra $\ext^*_A(S,S)$ and a single nonzero higher multiplication given by $m_3(g,g,g)=f^2$.
\p Now we wish to compute the Koszul dual of $\R\enn_A(S)$. Put $x=f^*$, $y=g^*$, $\zeta=(fg)^*$, $\xi=(f^2)^*$ and $\theta=(gf^2)^*$. Then $\R\enn_A(S)^!$ is freely generated by $\{x,y,\zeta,\xi,\theta\}$, and after working out the $A_\infty$ bar differential one arrives at the following theorem:

\begin{thm}\label{lauferflop}
	The derived contraction algebra of the Laufer flop is freely generated as a noncommutative dga by elements $x,y,\zeta,\xi,\theta$ in degrees $0,0,-1,-1,-2$ and of weights $3,2,5,6,8$ respectively. The differential is defined on generators by $dx=dy=0$, $d\zeta=-(xy+yx)$, $d\xi=y^3-x^2$, and $d\theta = [\xi,y] + [\zeta,x]$.
	\end{thm}
\begin{rmk}
In particular, one can use the above description to see that $H^0(\dca)\cong \frac{k\langle x,y\rangle }{xy+yx, x^2-y^3}$, the \textbf{quantum cusp}, which recovers the computation of \cite[Example 1.3]{DWncdf}.
\end{rmk}
It is unclear to the author how to produce an explicit cocycle representing the periodicity element $\eta \in H^{-2}(\dca)$ in terms of the generators given above. It may be feasible to use \ref{dsgrmk} to produce a model of $\dca$ where $\eta$ is represented by a genuinely central cocycle.

\begin{rmk}\label{lauferrmk}
The relations on the path algebra $A$ come from a superpotential $$\overline{ W}\coloneqq bay^2+\frac{1}{2}abab+x^2y-\frac{1}{4}y^4$$(cf. \cite[4.4]{amquivers}). Following \cite{DWncdf}, to compute the contraction algebra $A_\con$ one considers the subquiver 
	$$
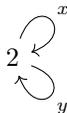
\begin{tikzcd}
	 2  \arrow[in=20, out=70, distance=7mm,"x"] \arrow[out=-20, in=-70, distance=7mm,"y"]
	\end{tikzcd}$$equipped with the modified superpotential $W\coloneqq x^2y-\frac{1}{4}y^4$. One can easily see that the path algebra of this quiver with superpotential is precisely $A_\con$. The Ginzburg dga associated to $W$ has generators $\{x,y,x^*,y^*,z\}$ in degrees $0,0,-1,-1,-2$ respectively, with differential $dx=dy=0$, $dx^*=-(xy+yx)$, $dy^*=y^3-x^2$, and $dz=[x,x^*] + [y,y^*]$. One can easily see that this Ginzburg dga is isomorphic (not just quasi-isomorphic!) to the dga we obtain in \ref{lauferflop} above.

\end{rmk}

\section{Surface examples}\label{surfcomps}
In this section, we compute some examples of derived contraction algebras for surfaces. All of the examples we will consider will be partial crepant resolutions of Kleinian singularities obtained by taking a slice of a threefold flopping contraction. Note that in this situation, we are automatically in the Local Setup \ref{localsetup} so may define the derived contraction algebra. Firstly, we'll prove some useful facts about slicing a threefold flopping contraction by a generic hyperplane, especially with regards to the tilting theory. Then we'll compute the derived contraction algebras associated to one-curve partial resolutions of $A_n$ singularities.

\subsection{Slicing}\label{slicingsctn}
In this section, we'll think about slicing flopping contractions by generic hyperplanes to get partial crepant resolutions of surface singularities. We'll pay special attention to how tilting bundles and their endomorphism rings behave under slicing; our aim is to prove Theorem \ref{cutsthm}. All of the arguments we use in this part were communicated to us by Michael Wemyss; the general idea is to adapt a proof of Ishii and Ueda \cite[8.1]{ishiiueda}.

\p The setup for this part will be a threefold flopping contraction $\pi: X \to \spec R$. Slice it by a generic hyperplane section to get a pullback diagram of the form $$\begin{tikzcd} Y\ar[d,"\psi"] \ar[r,"j"]& X \ar[d,"\pi"]\\ \spec(R/g) \ar[r,"i"]& \spec R
\end{tikzcd}$$First note that, by Reid's general elephant principle \cite[1.1, 1.14]{reidpagoda}, $\psi$ is a partial crepant resolution of an ADE singularity, and in particular projective and birational.

\begin{lem}\label{cutslem}
With the setup as above, the following hold:\hfill
\begin{enumerate}
	\item[\emph{i)}] $g: \mathcal{O}_X \to \mathcal{O}_X $ is an injection.
	\item[\emph{ii)}] Let $\mathcal W$ be a vector bundle on $X$. Then there is a short exact sequence $$0\to \mathcal W \xrightarrow{g} \mathcal W \to j_*j^*\mathcal W \to 0$$
	\item[\emph{iii)}] Let $\mathcal W$ be a vector bundle on $X$ such that $\R^p\pi_*\mathcal W \simeq 0$ for $p>0$. Putting $W\coloneqq \pi_*\mathcal W$, we have a quasi-isomorphism $\R\pi_*(j_*j^* \mathcal W) \simeq W / g W$.
\end{enumerate}
\end{lem}
\begin{proof}
For ${i)}$, note first that $g$ is a global section of $\mathcal{O}_X$, or equivalently an endomorphism of $\mathcal{O}_X$. Let $\mathcal K$ be its kernel. Because $X$ is normal, $\mathcal K$ is a reflexive sheaf: this is because on a normal integral noetherian scheme reflexive sheaves are characterised by the $S_2$ property \cite[2.10]{schwedediv}, which is closed under taking kernels. Because $R$ is an integral domain and $\pi_*$ preserves kernels, one has $\pi_*\mathcal K \cong 0$. But $\pi_*$ is a reflexive equivalence \cite[4.2.1]{vdb}, and hence $\mathcal K \cong 0$ as required. For {ii)}, the only thing to check is exactness on the left. But this follows by tensoring the injection $g: \mathcal{O}_X \to \mathcal{O}_X$ with $\mathcal W$. For {iii)}, note that by assumption $\mathcal W$ is $\pi_*$-acyclic, so we may compute the derived pushforward of $j_*j^* \mathcal W$ using its $\pi_*$-acyclic resolution $\mathcal W \xrightarrow{g} \mathcal W$. We hence get $\R\pi_*(j_*j^* \mathcal W)\simeq W \xrightarrow{g} W$ where the righthand $W$ is placed in degree zero. By reflexive equivalence again, $W$ is reflexive, and since it is a submodule of a free module we see that $g: W \to W$ is also injective. It follows that $\R\pi_*(j_*j^* \mathcal W)$ has cohomology only in degree zero, where it is $W/gW$.
\end{proof}

\begin{prop}\label{slicetilt}
With the setup as above, let $\mathcal W$ be a tilting bundle on $X$. Then $j^*\mathcal W$ is a tilting bundle on $Y$, with endomorphism ring $\enn_Y(j^*\mathcal W)\cong \enn_{R/g}(\psi_*j^*\mathcal W)\cong \enn_{R/g}(i^*\pi_*\mathcal W)$.
\end{prop}
\begin{proof}
First observe that because $\mathcal W$ is a vector bundle, so is $j^*\mathcal W$, and we also have $j^*\mathcal W \simeq \mathbb{L} j^*\mathcal W$. Because $j$ is a closed immersion (it's the pullback of the closed immersion $i$), we have \linebreak $\R j_* \simeq j_*$. Now it follows by adjunction that $\R\hom_Y(j^*\mathcal W,-) \simeq  \R\hom_X(\mathcal W, j_* -)$. For generation, let $\mathcal F \in D(\mathrm{QCoh}(Y))$. Then $\R\hom_Y(j^*\mathcal W,\mathcal F)\simeq 0$ if and only if $\R\hom_X(\mathcal W, j_* \mathcal F)\simeq 0$. But $\mathcal W$ generates by assumption, so this is the case if and only if $j_* \mathcal F\simeq 0$, which is the case if and only if $\mathcal F \simeq 0$. Hence, $j^*\mathcal W$ generates. 
\p To show Ext vanishing, we first compute $\R\enn_Y(j^*\mathcal W)\simeq \R\hom_X(\mathcal W, j_*j^* \mathcal W)$ as before. Because $\mathcal W$ is a vector bundle, we have $$\R\hom_X(\mathcal W, j_*j^* \mathcal W)\simeq \R\hom_X(\mathcal{O}_X,\mathcal W^*\otimes j_*j^* \mathcal W)\simeq \R\pi_*(\mathcal W^*\otimes j_*j^* \mathcal W).$$Again, because $\mathcal W$ is a vector bundle, $\mathcal W^*\otimes j_*j^* \mathcal W$ is quasi-isomorphic to $\mathcal W ^* \otimes (\mathcal W \xrightarrow{g} \mathcal{W})$ using \ref{cutslem}, {ii)}. But $\mathcal W ^* \otimes (\mathcal W \xrightarrow{g} \mathcal{W})\cong (\mathcal W ^* \otimes \mathcal W )\xrightarrow{g} (\mathcal{W}^* \otimes \mathcal{W})$, which, using \ref{cutslem}, {ii)} again, is quasi-isomorphic to $j_*j^*(\mathcal{W}^* \otimes \mathcal{W})$. So we have $\R\enn_Y(j^*\mathcal W)\simeq \R \pi_*j_*j^*(\mathcal{W}^* \otimes \mathcal{W})$, which by \ref{cutslem}, {iii)} (using that higher Exts between $\mathcal W$ vanish) is concentrated in degree zero. So $j^*\mathcal W$ is tilting. Note that this doesn't tell us about the ring structure on $\R\enn_Y(j^*\mathcal W)$, since we had to pass through adjunctions.

\p For the statements about endomorphism rings, observe first that we have a ring map\\ $\psi_*:\enn_Y(j^*\mathcal W)\to \enn_{R/g}(\psi_*j^*\mathcal W)$ which is also a map of reflexive $R/g$-modules. Since it is an isomorphism at height one primes, and $R/g$ is normal, it hence must be an isomorphism. It remains to check that $\enn_{R/g}(\psi_*j^*\mathcal W)\cong \enn_{R/g}(i^*\pi_*\mathcal W)$. To prove this we will show that $\psi_*j^*\mathcal W \cong i^*\pi_*\mathcal W$. Proceeding as before, we have \begin{align*}
\psi_*j^*\mathcal W &\simeq \R\hom_Y(\mathcal{O}_Y, j^*\mathcal W)
\\ &\simeq \R\hom_Y(j^*\mathcal{O}_X, j^*\mathcal W)
\\ &\simeq \R\hom_X(\mathcal{O}_X, j_*j^*\mathcal W)
\\ & \simeq \R\pi_*j_*j^*\mathcal W
\\&\simeq i^*\pi_*\mathcal W
\end{align*}where the last isomorphism is \ref{cutslem}, {iii)}.
\end{proof}
We'd like to say a little more: not just that one can compute the endomorphism ring of a tilting bundle on the base, but also that one can compute the endomorphism ring of $i^*W$ by applying the functor $i^*$ to the endomorphism ring of $W$. This is a little delicate and will require some more hypotheses; we show this in the case that $\mathcal W$ is Van den Bergh's tilting bundle.
\begin{thm}\label{cutsthm}
	With the setup as above, let $\mathcal V$ be Van den Bergh's tilting bundle on $X$ constructed in \text{\normalfont\cite[3.2.8]{vdb}}. Then  $j^*\mathcal{V}$ is a tilting bundle on $Y$, and one has a ring isomorphism $\enn_Y(j^*\mathcal{V})\cong \enn_R(\pi_*\mathcal{V})/g\enn_R(\pi_*\mathcal{V})$.
\end{thm}
\begin{proof}
Immediately from \ref{slicetilt}, we see that $j^*\mathcal V$ is tilting and has endomorphism ring given by $\enn_Y(j^*\mathcal{V})\cong  \enn_{R/g}(i^*\pi_*\mathcal V)$. Putting $V\coloneqq \pi_*\mathcal V$, it remains only to prove that we have an isomorphism $\enn_{R/g}(i^*V)\cong i^*\enn_R(V)$. By \cite[3.2.10]{vdb}, both $V$ and $\enn_R(V)$ are Cohen--Macaulay $R$-modules, and one moreover has an isomorphism $\enn_X(\mathcal V)\cong\enn_R(V)$. Because $R$ is an isolated singularity, \cite[2.7]{iwmaxmod} now gives $\ext^1_R(V,V)\cong 0$. We now follow the proof of \cite[A.1]{vdb}. Note that because $V$ is Cohen--Macaulay it follows that $i^*V$ is Cohen--Macaulay over $R/g$. Applying $\hom_R(V,-)$ to the exact sequence $$0 \to V \xrightarrow{g} V \to i^* V \to 0$$ gives an exact sequence $$0 \to \enn_R(V) \xrightarrow{g} \enn_R(V) \to \hom_R(V,i^*V)\to 0$$or in other words an isomorphism $\hom_R(V,i^*V)\cong i^*\enn_R(V)$. But $\hom_R(V,i^*V) \cong \enn_{R/g}(i^*V)$, and it is not hard to check that the induced linear isomorphism $\enn_{R/g}(i^*V)\cong i^*\enn_R(V)$ is a ring map.
\end{proof}

\subsection{Partial resolutions of $A_n$ singularities}
Let $\tilde{X} \to \spec\tilde{R}$ be the Atiyah flop. Observe that, for any choice of $n$, one can slice $\tilde{R}$ along the hypersurface $x=y^n$ to obtain a partial crepant resolution $X \to \spec(R)$ of an $A_n$ singularity. Let $\tilde \Lambda$ be the NCCR of $\tilde R$ from section \ref{atiyahflop}, with quiver presentation\\
\begin{wrapfigure}[4]{L}{0.5\textwidth}
	\centering
	\vspace{-15pt}$
	\begin{tikzcd}
	1 \arrow[rr,bend left=15,"b"] \arrow[rr,bend left=50,"a"]  && 2 \arrow[ll,bend left=15,"s"] \arrow[ll,bend left=50,"t"]
	\end{tikzcd}$
\end{wrapfigure} 

$asb=bsa$

$sbt=tbs$

$atb=bta$

$sat=tas$\newline \newline
By \ref{cutsthm}, it follows that $X$ is derived equivalent to the `sliced NCCR' $\Lambda\coloneqq \tilde \Lambda / (x-y^n)\tilde \Lambda$. Recalling the construction of $\tilde \Lambda$, we had $x=at+ta$ and $y=sb+bs$. Moreover, since $sbbs=bssb=0$, we have $y^n=(sb)^n + (bs)^n$. So we need to add the relation $at+ta=(sb)^n + (bs)^n$, which is equivalent to adding the two relations $at=(bs)^n$ and $ta=(sb)^n$. The algebra $\Lambda$ we get is\newline
\begin{wrapfigure}[4]{L}{0.5\textwidth}
	\centering
	\vspace{-15pt}$
	\begin{tikzcd}
	1 \arrow[rr,bend left=15,"b"] \arrow[rr,bend left=50,"a"]  && 2 \arrow[ll,bend left=15,"s"] \arrow[ll,bend left=50,"t"]
	\end{tikzcd}$
\end{wrapfigure} 

$asb=bsa$

$sbt=tbs$

$at=(bs)^n$

$ta=(sb)^n$\newline \newline  noting that $atb=bta$ and $sat=tas$ follow from the new relations. Observe that $\Lambda$ admits a grading by putting the generators $e_1,e_2$ in degree 0, the generators $b$ and $s$ in degree 1, and the generators $a$ and $t$ in degree $n$. We'll refer to the degree of a homogeneous element of $\Lambda$ as its \textbf{weight}, since we're already using `degree' to refer to maps. Let $S$ be the simple module at 2. For brevity, put $\beta\coloneqq bs$ and $\sigma\coloneqq sb$. Our main theorem is the following:
\begin{thm}\label{mainsurfacethm}
If $n=1$, then $\dca$ is quasi-isomorphic to the free noncommutative graded algebra $k\langle\zeta\rangle$ on a variable $\zeta$ in degree $-1$. If $n\geq 2$, then $\dca$ is quasi-isomorphic to the strictly unital minimal $A_\infty$-algebra with underlying graded algebra $k[\eta,\zeta]$ where $\eta$ has degree $-2$, $\zeta$ has degree -1, and the only nontrivial higher multiplications are $$m_{n+1}(\eta^{b_1}\zeta,\ldots, \eta^{b_{n+1}}\zeta)=\eta^{b_1+\cdots+b_{n+1} + n}.$$Note in particular that since $\zeta$ has odd degree and is a commutative element, it must square to zero.
\end{thm}
\begin{proof}
We split the proof into two separate cases: firstly when $n=1$, which is straightforward and appears as \ref{mainstone}, and secondly when $n\geq2$, which is the real meat of the proof and appears as \ref{surfdcang2}.
\end{proof}
\begin{rmk}
Note that in the $n=1$ case, the periodicity element $\eta$ is given by $\eta=\zeta^2$. When $n=1$, note that $m_{n+1}(\eta^{b_1}\zeta,\ldots, \eta^{b_{n+1}}\zeta)=\eta^{b_1+\cdots+b_{n+1} + n}$ still holds; hence a more uniform way to state the above is to say that $\dca$ as a strictly unital minimal $A_\infty-k[\eta]$-algebra is generated by the single element $\zeta$ subject to the relations $m_{r}(\zeta,\ldots,\zeta)=\delta_{r,n+1}\eta^n$.
\end{rmk}
\begin{rmk}
Unlike in the threefold setting, there does not seem to be a simple way of obtaining $\dca$ as a Ginzburg dga. Indeed, the subquiver at the vertex 2 is a point, and so must have associated Ginzburg dga $k[\eta]$, with $\eta$ in degree -2, independent of what the superpotential is.
\end{rmk}

We begin by noting a few preliminaries. One can check that $S$ has projective resolution $S^\bullet$ given by
$$\cdots \xrightarrow{\stbtm{-\beta^{n-1}}{a}{t}{-\sigma} \cdot} P_1 P_2 \xrightarrow{\stbtm{\beta}{a}{t}{\sigma^{n-1}}\cdot} P_1 P_2 \xrightarrow{\stbtm{-\beta^{n-1}}{a}{t}{-\sigma} \cdot} P_1 P_2 \xrightarrow{\stbtm{bt}{\beta^{n-1}b}{-\beta}{-a}\cdot} P_1^2 \xrightarrow{\left( s \;\; t\right)\cdot}  P_2$$which eventually becomes periodic with period two. Moreover one can check that with the grading conventions on $\Lambda$ from above this admits a secondary grading by weight$$\cdots \xrightarrow{} (3n+2,4n) \xrightarrow{} (3n,2n+2) \xrightarrow{} (n+2,2n) \xrightarrow{} (1,n) \xrightarrow{}  (0) $$ where the differential has weight zero.

\begin{prop}\label{mainstone}
If $n=1$, then $\dca$ is the dga $k\langle\zeta\rangle$ on a noncommutative variable $\zeta$ in degree $-1$.
\end{prop}
\begin{proof}
In this case, one can check that $\hom_\Lambda(S^\bullet,S)$ is the complex $$k\to 0 \to k \xrightarrow{0} k \xrightarrow{-1} k \xrightarrow{0} k \xrightarrow{-1}\cdots $$ and hence its cohomology is just a copy of $k$ in degrees 0 and 2. We see that the cohomology algebra $\ext_\Lambda^*(S,S)$ is just $k[x]/(x^2)$, where $x$ has degree 2. It's easy to see that $\R\enn_\Lambda(S)$ must be formal, and hence the derived contraction algebra is the noncommutative dga $$\dca=k\langle\zeta\rangle$$ where $\zeta$ has degree $-1$. Note that the periodicity element is $\eta=\zeta^2$.
\end{proof}

\p From now on, we assume that $n\geq2$. We see that $$\ext^i_\Lambda(S,S)\cong \begin{cases} 0 & i<0\normalfont\text{  or  } i=1\\
k & i=0 \normalfont\text{  or  } i>1
\end{cases}$$spanned by the classes of the projection maps $P_2 \to S$. Define maps $g_0=\id$, and for $k\geq 1$ 
$$g_{2k}\coloneqq \begin{tikzcd}[sep=large]	S^{2k}\ar[d,swap,"\left( 0 \;\; -1\right)\cdot"]&S^{2k+1}\ar[l,swap,"d_{2k}"]\ar[d,swap,"\stbtm{0}{b}{-1}{0}\cdot"]&S^{2k+2}\ar[l,swap,"d_{2k+1}"]\ar[d,swap,"\id"]&S^{2k+3}\ar[l,swap,"d_{2k+2}"]\ar[d,swap,"\id"]&S^{2k+4}\ar[l,swap,"d_{2k+3}"]\ar[d,swap,"\id"]&\cdots\ar[l,swap,"d_{2k+4}"]\\S^0&S^1\ar[l,"d_0"]&S^2\ar[l,"d_1"]&S^3\ar[l,"d_2"]&S^4\ar[l,"d_3"]&\cdots\ar[l,"d_4"]\end{tikzcd}$$
$$g_{2k+1}\coloneqq \begin{tikzcd}[sep=large]S^{2k+1}\ar[d,swap,"\left( 0 \;\; -1\right)\cdot"]&S^{2k+2}\ar[l,swap,"d_{2k+1}"]\ar[d,swap,"\stbtm{0}{\beta^{n-2}b}{1}{0}\cdot"]&S^{2k+2}\ar[l,swap,"d_{2k+3}"]\ar[d,swap,"\stbtm{-\beta^{n-2}}{0}{0}{1}\cdot"]&S^{2k+4}\ar[l,swap,"d_{2k+3}"]\ar[d,swap,"\stbtm{-1}{0}{0}{\sigma^{n-2}}\cdot"]&S^{2k+5}\ar[l,swap,"d_{2k+4}"]\ar[d,swap,"\stbtm{-\beta^{n-2}}{0}{0}{1}\cdot"]&\cdots\ar[l,swap,"d_{2k+5}"]\\S^0&S^1\ar[l,"d_0"]&S^2\ar[l,"d_1"]&S^3\ar[l,"d_2"]&S^4\ar[l,"d_3"]&\cdots\ar[l,"d_4"]\end{tikzcd}$$
Then the $g_k$ span the cohomology algebra $\ext_\Lambda^*(S,S)$ since each (up to sign) lifts the projection maps $P_2 \to S$. Moreover, letting $\phi$ be the degree zero map with $\phi_0=\sigma^{n-2}\cdot$, $\phi_1=\beta^{n-2}\cdot$, and $\phi_j=\tbtm{\beta^{n-2}}{0}{0}{\sigma^{n-2}}\cdot$ for all $j>1$, one can check that the $g_k$ satisfy $$g_ig_j=\begin{cases} g_{i+j} & \text{if }ij \text{ is even} \\ g_{i+j}\phi & \text{ else}\end{cases}$$Put $x\coloneqq [g_2]$ and $y\coloneqq [g_3]$.

\begin{prop}
Suppose that $n=2$. Then the derived endomorphism algebra $\R\enn_\Lambda(S)$ is formal, with cohomology algebra $\frac{k[x]\langle y \rangle}{(xy-yx, x^3-y^2)}$.  Note that this is a noncommutative dga, because $y$ does not commute with itself.
\end{prop}
\begin{proof}
	We see that $\phi=\id$ and hence $\ext_\Lambda^*(S,S)$ is isomorphic to $\frac{k[x]\langle y \rangle}{(xy-yx, x^3-y^2)}$ where $x$ has degree 2, $y$ has degree 3, and the differential is zero. It's easy to see that $\R\enn_\Lambda(S)$ must be formal, since it's quasi-isomorphic to the subalgebra generated by $\id, g_2$ and $g_3$.
\end{proof}

\begin{prop}
	Suppose that $n>2$. Then the Ext-algebra $\ext_\Lambda^*(S,S)\cong k[x,y]$ is freely generated as a cdga by $x$ and $y$. Note that $y^2=0$.
\end{prop}
\begin{proof}
One can check that $[\phi]=0$, and the result follows.
\end{proof}

Now we need to split our argument into cases. We can handle the $n=2$ case already, but part of the argument will be identical for $n>2$, so we defer this for the present moment. We aim first to identify, for $n>2$, the higher $A_\infty$ multiplications on $\ext_\Lambda^*(S,S)$ making it quasi-isomorphic to $\R\enn_\Lambda(S)$, via a Massey product computation. In order to do this, note that the resolution $S^\bullet$ of $S$ becomes eventually periodic, with period 2. It will be convenient for us to work in the 2-periodic part of the dga $\R\enn_\Lambda(S)$.

\begin{defn}
	Let $E^\mathrm{ep}$ be the subspace of the dga $\enn_\Lambda(S^\bullet)$ consisting of those maps of degree at least 2 which commute with $g_2$. We call such a map an \textbf{eventually periodic} map.
\end{defn}

\begin{lem}
	An eventually periodic map $f\in E^\mathrm{ep}$ is given by the formula
	$$f= \begin{tikzcd}[sep=large]	S^{j}\ar[d,swap,"f_0"]&S^{j+1}\ar[l,swap,"d_{j}"]\ar[d,swap,"f_1"]&S^{j+2}\ar[l,swap,"d_{j+1}"]\ar[d,swap,"f_2"]&S^{j+3}\ar[l,swap,"d_{j+2}"]\ar[d,swap,"f_3"]&S^{j+4}\ar[l,swap,"d_{j+3}"]\ar[d,swap,"f_2"]&\cdots\ar[l,swap,"d_{j+4}"]\\S^0&S^1\ar[l,"d_0"]&S^2\ar[l,"d_1"]&S^3\ar[l,"d_2"]&S^4\ar[l,"d_3"]&\cdots\ar[l,"d_4"]\end{tikzcd}$$where $f_i=f_{i+2}$ for $i\geq 2$, and $f_0=\left( 0 \;\; -1\right)f_2$ and $f_1=\stbtm{0}{b}{-1}{0}f_3$.
\end{lem}
\begin{proof}
	Compute $[f,g_2]=f\circ g_2 - g_2 \circ f$ and set it to zero.
\end{proof}
In particular, $f$ is determined by the pair $(f_2, f_3)$, and any such pair of maps defines an eventually periodic map.
\begin{defn}
	Let $f$ be an eventually periodic map of a given degree. Since $f$ is determined by its components $f_2$ and $f_3$, we use the notation $f_2 \vert f_3$ to specify $f$ uniquely.
\end{defn}
\begin{defn}
	Let $f\in \enn_\Lambda(S^\bullet)$ be a map of degree $\geq 2$ satisfying $f_i=f_{i+2}$ for $i\geq N$ for some natural $N\geq 2$. The \textbf{periodicisation} of $f$ is the map $f^{\mathrm{ep}} \in E^{\mathrm{ep}}$ of the same degree as $f$ defined by the formula $$f^{\mathrm{ep}}\coloneqq\begin{cases}f_N\vert f_{N+1}&\text{if }N \text{ is even} \\ f_{N+1}\vert f_{N}& \text{if }N \text{ is odd}\end{cases}.$$
\end{defn}
In particular if $f\in E^{\mathrm{ep}}$, then $f^{\mathrm{ep}}=f$. Note that $f^{\mathrm{ep}}$ agrees with $f$ in all degrees $\geq N$.
\begin{lem}
	The complex $E^\mathrm{ep}$ is a nonunital dga, and the inclusion $\iota:E^\mathrm{ep}\into \enn_\Lambda(S^\bullet)$ is a dga map that induces isomorphisms on cohomology in degrees $> 2$ and a surjection on cohomology in degrees $\geq 2$.
\end{lem}
\begin{proof}
	The fact that $\iota$ is an inclusion of nonunital dgas is not hard to see. One can verify that the $g_j$ are eventually periodic, and since they generate the cohomology of $\ext^*_\Lambda(S,S)$, the map $\iota$ must be a quasi-surjection in degrees $\geq 2$. To see that $\iota$ is a quasi-injection in degrees $>2$, take an $h\in \enn_\Lambda(S^\bullet)$ of degree $\geq 2$, and assume that $dh\in E^\mathrm{ep}$. We need to find an $l\in E^\mathrm{ep}$ with $dh=dl$. Because $dh\in E^\mathrm{ep}$, $h$ must be $2$-periodic in high degrees. Since $h^\mathrm{ep}\in E^\mathrm{ep}$, we have $dh^\mathrm{ep}\in E^\mathrm{ep}$, and so $dh-dh^\mathrm{ep}\in E^\mathrm{ep}$. But $h$ agrees with $h^\mathrm{ep}$ in high degrees, and so $dh-dh^\mathrm{ep}$ is zero in high degrees. So $dh-dh^\mathrm{ep}=0$.
\end{proof}
In particular, any map of degree at least $3$ in $\enn_\Lambda(S^\bullet)$  is homotopic to an eventually periodic map. We use this to assist us in our Massey product computation.
\begin{prop}Suppose that $n>2$. Then the Massey product $\langle y,\ldots, y\rangle_n$ is nontrivial.
\end{prop}
\begin{proof}
	This is rather involved but ultimately straightforward. We in fact show that $(-1)^nx^{n+1}$ is an element of $\langle y,\ldots, y\rangle_n$. We're going to proceed by setting $e_1\coloneqq g_3$ and inductively finding $e_i$ such that $de_i=e_1 e_{i-1}+\cdots+e_{i-1}e_1$. Note that we will require $de_2 =g_3^2$. For $2\leq i\leq n$, define a degree $2i-1$ eventually periodic map $\nu_i$ by the formula$$\nu_i\coloneqq \tbtm{\beta^{n-i}}{0}{t}{-\sigma}\mid\tbtm{\beta}{0}{t}{-\sigma^{n-i}}$$
	The $\nu_i$ will satisfy some simple relations, but we will need to keep track of the degrees of our maps. Unfortunately this makes things notationally messy. If $w=w_2\vert w_3$ is an eventually periodic map of degree $j$, then we denote by $w\{l\}$ the eventually periodic map of degree $j+l$ given by the formula$$w\{l\}\coloneqq\begin{cases}w_2\vert w_3&\text{if }l \text{ is even} \\ w_3\vert w_2& \text{if }l \text{ is odd}\end{cases}.$$In other words, $w\{l\}$ is $w$ but viewed as a map of a different degree. One can check that the following hold:
	
	\begin{enumerate}
		\item $d \nu_i = \tbtm{\beta^{n-i+1}}{0}{0}{\sigma^{n-i+1}} \mid \tbtm{\beta^{n-i+1}}{0}{0}{\sigma^{n-i+1}}$
		\item $\nu_i\nu_j = d(\nu_i\{2j-2\} + \nu_j\{2i-2\})$
		\item $g_3\nu_i+\nu_ig_3 =d( -\nu_{i+1}-\nu_2\{2i-1\})$ if $i<n$
		\item If $i<n$ then $\nu_i \simeq 0$
	\end{enumerate}
	Observe that $d \nu_3=g_3^2$. So we set $e_1\coloneqq g_3$, and we want to inductively find $e_i$ such that \linebreak $d e_i=e_1 e_{i-1}+\cdots+e_{i-1}e_1$, starting with $e_2=\nu_3$. We need to check that we can do this. We prove by induction that for $2\leq i<n$ there exist maps $e_i$ of degree $2i+1$ such that: \begin{enumerate}
		\item$e_i$ is a $\Z$-linear combination of the maps $\nu_{i+1},\ldots,\nu_2\{2i-1\}$, and the coefficient of $\nu_{i+1}$ is $(-1)^i$.
		\item $d e_i= e_1 e_{i-1}+\cdots+e_{i-1}e_1$.
	\end{enumerate}

	The idea of the induction is simple; we just expand out each expression $e_1 e_{i-1}+\cdots+e_{i-1}e_1$ and `integrate term-by-term'. The hard part is keeping track of all the indices. The base case is clear; we may take $e_2\coloneqq \nu_3$ as above. For the induction step, suppose that for all $j<i$, all $e_j$ are defined and have the two properties above. We wish to construct $e_i$. For $j\geq 2$ write $$e_j=\sum_{r=2}^{j+1}\lambda_r^j\nu_r\{a^j_r\}$$with $\lambda_{j+1}^j=(-1)^j$ and $a^j_r=2(j-r)+2$. Then it is clear that for $1<j,k<i$, we have the identity $$e_je_k=\sum_{r=2}^{j+1}\sum_{q=2}^{k+1}\lambda_r^j\lambda_q^k\nu_r\nu_q\{a^j_r+a^k_q\}$$Hence, if we set $$m_{jk}\coloneqq \sum_{r=2}^{j+1}\sum_{q=2}^{k+1}\lambda_r^j\lambda_q^k(\nu_r\{a^j_r+2k\}+\nu_q\{a^k_q+2j\})$$we see that $d m_{jk}= e_je_k$. Observe that $m_{jk}$ is a map of degree $2(j+k)+1$. Moreover we have $$g_3e_{i-1}+e_{i-1}g_3=\sum_{r=2}^{i}\lambda_r^{i-1}(g_3\nu_r+\nu_rg_3)\{a^{i-1}_r\}$$So if we set $$m\coloneqq -\sum_{r=2}^{i}\lambda_r^{i-1}(\nu_{r+1}+\nu_2\{2r-1\})\{a^{i-1}_r\}$$we see that $m$ is a map of degree $2i+1$ with $d m=e_1e_{i-1}+e_{i-1}e_1$. Thus if we set $$e_i\coloneqq m + m_{2(i-2)}+\cdots+m_{(i-2)2}$$ we see that by construction, $e_i$ satisfies condition $2.$. Clearly $e_i$ is a is a $\Z$-linear combination of $\nu_{i+1},\ldots,\nu_2\{2i-1\}$. So we just need to check what the coefficient of $\nu_{i+1}$ in $e_i$ is. It is easy to see that this coefficient is $-\lambda^{i-1}_i$ which by the induction hypothesis is $-(-1)^{i-1}=(-1)^i$. Hence $e_i$ satisfies condition 1.
	\p
	We're almost done. We observe that one element of the $n$-fold Massey product $\langle g_3,\ldots,g_3\rangle$ is given by $[e_1 e_{n-1}+\cdots+e_{n-1}e_1]$. So it suffices to prove that $e_1 e_{n-1}+\cdots+e_{n-1}e_1 \not\simeq 0$. We see that $e_je_k\simeq 0$ holds as long as $1<j,k<n$, so that we have a homotopy \linebreak$e_1 e_{n-1}+\cdots+e_{n-1}e_1\simeq e_1 e_{n-1}+e_{n-1}e_1$. Observe also that $e_1\nu_j+\nu_je_1\simeq0$ holds if $j<n$. Hence we see that we have a homotopy $e_1e_{n-1}+e_{n-1}e_1\simeq (-1)^{n-1}(e_1\nu_n+\nu_ne_1)$. It is easy to check that $e_1\nu_n+\nu_ne_1 \simeq -g_{2n+2}$. Hence $e_1 e_{n-1}+\cdots+e_{n-1}e_1 \simeq (-1)^ng_{2n+2}\not\simeq 0$.
\end{proof}

\begin{cor}
	When $n>2$, $\R\enn_\Lambda(S)$ is not a formal dga.
\end{cor}

	\begin{prop}\label{surfeng2}
		Let $n>2$. Then $\R\enn_\Lambda(S)$ is quasi-isomorphic to the strictly unital minimal $A_\infty$-algebra with underlying graded algebra $k[x,y]$, with $x$ in degree 2 and $y$ in degree 3, with the only nontrivial higher multiplications being $m_n(x^{b_1}y,\ldots,x^{b_l}y)=x^{n+1+b_1+\cdots +b_n}$.
	\end{prop}
\begin{proof}
	We employ the usual trick of using the extra grading on the resolution $S^\bullet$ to rule out most higher multiplications. Observe that in the secondary grading on the resolution, $x$ has weight $-2n$, and $y$ has weight $-(2n+2)$. Appealing to the graded version of Merkulov's construction, one can consider the higher multiplication $m_{r+l}(x^{a_1},\ldots,x^{a_r},x^{b_1}y,\ldots,x^{b_l}y)$, which must be of degree $2-r+2a+2b+2l$ and weight $-2na-2nb-2(n+1)l$, where we write $a=a_1+\cdots+a_r$ and $b=b_1+\cdots+b_r$. Via casework on the parity of $r$, one can see that if $r+l>2$, the only way for this to be nonzero is when we are looking at a product of the form $m_{n}(x^{b_1}y,\ldots,x^{b_n}y)=\lambda x^{1+b+n}$, where $\lambda$ depends on the $b_i$. Consideration of the Stasheff identity $\mathrm{St}_{n+1}$ with inputs of the form $x^{b_1}y\otimes\cdots\otimes x^{b_i}y\otimes x^m \otimes x^{b_{i+1}}y\otimes\cdots\otimes x^{b_n}y$ shows that the higher multiplications $m_n$ are $x$-linear, in the sense that $m_n(x^{b_1}y,\ldots, x^{b_n}y)=x^bm_n(y,\ldots,y)$. So the only higher multiplication of interest is $m_n(y,\ldots,y)=\lambda_0 x^{1+n}$. Because $\R\enn_\Lambda(S)$ is not formal, we must have $\lambda_0\neq 0$, and rescaling if necessary one can fix $\lambda_0=1$.
\end{proof}

\begin{rmk}\label{surfextrmk}
	Alternately, one can say that $\R\enn_\Lambda(S)$ is quasi-isomorphic to the strictly unital minimal $A_\infty-k[x]$-algebra generated by $y$ subject to the relations $m_r(y,\ldots,y)=\delta_{r,n}x^{n+1}$. Note that this also holds for $n=2$.
\end{rmk}

\begin{prop}\label{surfdcacohom}Let $n\geq 2$.
	We have $H^*(\dca)\cong k[\eta,\zeta]$ where $\eta=x^*$ has degree -2 and $\zeta=y^*$ has degree -1.
\end{prop}
\begin{proof}
	For brevity put $E\coloneqq \R\enn_\Lambda(S)$ and recall that $\dca=E^!$ the Koszul dual. We filter $E$ by powers of $y$, use this to get a filtration on $\dca$, and consider the resulting spectral sequence. Let $W^0E=k[x]$ and let $W^1E=E$. One can check easily that this is a multiplicative filtration. We obtain $\mathrm{gr}_1^WE\cong k[y]$ and $\mathrm{gr}_0^WE\cong k[x]$. The filtration $W$ gives us a filtration on $E^!$, which we also call $W$, with associated graded $\mathrm{gr}^W(E^!)\cong(\mathrm{gr}^WE)^!$. Now, $\mathrm{gr}^WE\cong k[x,y]$ and so $\mathrm{gr}^W(E^!)\cong k[\eta,\zeta]$.
	
	\p Now we consider the spectral sequence $F$ associated to the filtration $W$ on $E^!$. It has $F_0$ page $F_0^{pq}=(\mathrm{gr}_p^W (E^!))^{p+q}\;\implies\;H^{p+q}(E^!)$. Writing out this page, we see that all differentials are trivial and so $F_0=F_\infty$. Hence we have $(\mathrm{gr}_p^W (E^!))^{p+q}=\mathrm{gr}_p^W H^{p+q}(E^!)$, and so forgetting the extra grading we get $H(E^!)\cong\mathrm{gr}^W(E^!)\cong k[\eta,\zeta]$ as required.
\end{proof}

\begin{rmk}
	Note that this holds for both $n=2$ and $n>2$. To see this in a more unified way, one can use the description of \ref{surfextrmk}.
\end{rmk}	

\begin{prop}\label{surfdcang2}
	Let $n\geq2$. Then the derived contraction algebra $\dca$ is quasi-isomorphic to the strictly unital minimal $A_\infty$-algebra with underlying graded algebra $ k[\eta,\zeta]$, where $\eta$ has degree $-2$, $\zeta$ has degree -1, and the only nontrivial higher multiplication is $$m_{n+1}(\eta^{b_1}\zeta,\ldots, \eta^{b_{n+1}}\zeta)=\eta^{b_1+\cdots+b_{n+1} + n}.$$

\end{prop}
\begin{proof}
This is extremely similar to the proof of \ref{surfeng2}. A calculation with degree and weight yields that the only possible nontrivial higher multiplications are of the form	$$m_{n+1}(\eta^{b_1}\zeta,\ldots, \eta^{b_{n+1}}\zeta)=\lambda\eta^{b_1+\cdots+b_{n+1} + n}$$where $\lambda$ depends on the $b_i$. One gets $\eta$-linearity of the higher multiplications by considering the Stasheff identity $\mathrm{St}_{n+2}$. To see that $\dca$ is not formal, use that the Koszul dual of $\dca$ must be $\R\enn_\Lambda(S)$ again, but this does not agree with $k[\eta,\zeta]^!$. Hence we must have $m_{n+1}(\zeta,\ldots, \zeta)=\lambda_0\eta^{ n}$ for some $\lambda_0\neq 0$, and hence one can choose $\lambda_0=1$.
\end{proof}

\begin{rmk}
	Again, this applies for both $n=2$ and $n>2$, and one can equivalently describe $\dca$ as the strictly unital minimal $A_\infty-k[\eta]$-algebra generated by $\zeta$ subject to the relations $m_r(\zeta,\ldots,\zeta)=\delta_{r,n+1}\eta^{n}$.
\end{rmk}

\section{The mutation-mutation autoequivalence}\label{mutnauto}

In this section, we will study the mutation-mutation equivalence, which is a noncommutative generalisation of the flop-flop autoequivalence. Our main theorem is \ref{mutncontrol}, which is a generalisation of Donovan and Wemyss's result that the contraction algebra of a threefold flop `controls' the flop-flop autoequivalence \cite[5.10]{DWncdf}. In particular, we show that the truncation $A_\mm\coloneqq \tau_{\geq -1}(\dca)$ `controls' the mutation-mutation autoequivalence in more general settings, via noncommutative twists. In the first couple of sections we will set up the theory. We'll do some computations, and show that mutation respects the recollement of \ref{recoll}, which we will use to obtain some results on t-structures analogous to those of \cite{bridgeland}. We prove the main theorem using the machinery of singularity categories and derived localisation; in particular we will need some crucial technical results established in \cite{dqdefm}. We will in fact show that the mutation-mutation autoequivalence, when restricted to the derived category $D(\dca)$, is simply the shift [-2] (\ref{mmshift}). In the hypersurface setting, this will be enough since one can use arguments involving the periodicity element $\eta$ to interchange shifts and truncation.
\subsection{sCY rings and modifying modules}
Given a reasonable commutative ring $R$ and a reasonable module $V$, one can consider the ring $A=\enn_R(V)$ as a sort of noncommutative partial resolution of $R$. One would like to be able to `mutate' $A$ into a new ring $A'=\enn_R(V')$, and obtain a derived equivalence between $A$ and $A'$. In this part we follow Iyama--Reiten \cite{iyamareiten} and Iyama--Wemyss \cite{iwmaxmod} to provide rigorous definitions of `reasonable'.
\begin{defn}\label{scydefn}
	Let $R$ be a commutative $k$-algebra. Say that $R$ is \textbf{singular Calabi--Yau} (or just \textbf{sCY}) if the three following conditions are satisfied:
	\begin{enumerate}
		\item $R$ is Gorenstein.
		\item $R$ has finite Krull dimension $d$.
		\item $R$ is equicodimensional; i.e. all of its maximal ideals have the same height (which is equivalent to specifying that $\dim R_\mathfrak{m} =d$ for all $\mathfrak m \subseteq R$ maximal).
	\end{enumerate}
\end{defn}
\begin{rmk}
	This is a special case of Iyama--Reiten's definition \cite[\S3]{iyamareiten} for noncommutative rings; see \cite[3.10]{iyamareiten} for the proof of equivalence. In \cite{iyamareiten} this condition is called $d$-CY$^-$, and in \cite{iwmaxmod} it is called $d$-sCY. This is because $R$ is $d$-sCY if and only if a Calabi--Yau type condition $\hom_{D(R)}(X,Y[d])\cong D_M\hom_{D(R)}(Y,X)$ holds for certain $X,Y \in D^b(R)$, where $D_M$ denotes the Matlis dual.
\end{rmk}
	A typical example of a sCY ring is a local complete intersection (l.c.i.) domain, or a localisation or completion thereof:
\begin{lem}\label{lciscy}Let $R=k[x_1,\ldots, x_n]/I$ be a l.c.i. domain and $\mathfrak m \subseteq R$ be a maximal ideal. Then all of $R$, $R_{\mathfrak m}$ and $\hat R_{\mathfrak m}$ are sCY.
\end{lem}
\begin{proof}
	The ring $R$ is Gorenstein because it is a l.c.i. domain \cite[21.19]{eisenbud}, equicodimensional because it is an affine domain \cite[13.4]{eisenbud}, and clearly of finite Krull dimension. Hence $R$ is sCY. The localisation $R_{\mathfrak m}$ is sCY by \cite[3.1(3)]{iyamareiten} and $\hat R_{\mathfrak m}$ is sCY by the proof of \cite[3.1(4)]{iyamareiten}.
\end{proof}
\begin{defn}[{\cite[4.1]{iwmaxmod}}]
	Let $R$ be a sCY ring and $V$ a reflexive $R$-module. Say that $V$ is \textbf{modifying} if $\enn_R(V)$ is a Cohen--Macaulay $R$-module.
\end{defn}
\begin{prop}[{\cite[5.12(1)]{iwmaxmod}}]\label{modext}
	Let $R$ be a sCY ring of dimension $d$ with isolated singularities. Let $V$ be a Cohen--Macaulay $R$-module. Then $V$ is modifying if and only if $\ext^i_R(V,V)$ vanishes for all $1\leq i \leq d-2$.
\end{prop}
\begin{cor}\label{twoorthreecor}
	Let $R$ be a sCY ring with isolated singularities and let $V$ be a Cohen--Macaulay $R$-module. \begin{enumerate}
		\item[\emph{i)}] If $R$ is a surface, then $V$ is modifying.
		\item[\emph{ii)}] If $R$ is a threefold, then $V$ is modifying if and only if it is rigid (i.e. $\ext_R^1(V,V)\cong 0$).
	\end{enumerate}
\end{cor}
\begin{lem}\label{mcmsummod}
	Let $R$ be a sCY ring with isolated singularities. Let $M$ be a modifying $R$-module. If $M$ is maximal Cohen--Macaulay then $R\oplus M$ is modifying. 
\end{lem}
\begin{proof}
	Because $M$ is MCM, $\ext^i_R(M,R)$ vanishes for all $i\geq 0$.
\end{proof}

\subsection{Mutation and derived equivalences}
In this section, we fix a complete local isolated hypersurface singularity $R$ of dimension at least 2. Note that $R$ is a sCY ring by \ref{lciscy}. Moreover, $R$ is normal by Serre's criterion.
\begin{rmk}
	Everything we discuss in this section will still work \textit{mutatis mutandis} if $R$ is assumed to be any sCY ring of dimension at least 2 with isolated singularities. We choose $R$ to be a complete local hypersurface in order to simplify notation when dealing with syzygies. In the more general case, one needs to distinguish between $\Omega$ and $\Omega^{-1}$.
\end{rmk}
The mutation of a modifying module will be its syzygy. Syzygies are normally defined up to free summands; we define them here to preserve the number of free summands of an MCM module.
\begin{defn}
	Let $M$ be a MCM $R$-module with no free summands. Take a minimal free resolution $\cdots \xrightarrow{d_1} F_1 \xrightarrow{d_0} F_0$ of $M$. The \textbf{syzygy} of $M$ is the module $\Omega M\coloneqq  \ker(d_0)$.
\end{defn}
\begin{defn}
	Let $M$ be a MCM $R$-module, and write $M=F\oplus M'$ where $F$ is free and $M'$ has no free summands. Put $\Omega M\coloneqq F\oplus \Omega M'$. 
\end{defn}
It is easy to see that for any MCM module $M$ we have a short exact sequence 
\begin{equation}\label{syz}
0 \to \Omega M \to R^m \to M \to 0
\end{equation}
for some $m \in \N$ depending on $M$. Moreover, $\Omega M$ has the same number of free summands as $M$. Because $R$ is a hypersurface, $\Omega$ is 2-periodic:
\begin{prop}[\cite{eisenbudper}]
	If $M$ is a MCM module, then there is an isomorphism $\Omega^2 M \cong M$.
\end{prop}
In particular, one also has a short exact sequence
\begin{equation}\label{cosyz}
0 \to M \to R^m \to \Omega M \to 0.
\end{equation}
\begin{rmk}
	We will be interested in modifying $R$-modules $M$ that are also MCM; firstly to use arguments about the stable category $\stab R$, and secondly to ensure that the sum $R\oplus M$ is modifying (cf. \ref{mcmsummod}). Periodicity in the singularity category, along with the fact that stable Ext agrees with usual Ext in positive degrees, tells us that for an MCM module $M$ there is an isomorphism $\ext^{2}_R(M,M)\cong \underline{\enn}_R(M)$. If $R$ has dimension strictly greater than $3$, then \ref{modext} tells us that an MCM modifying module $M$ must have $\underline{\enn}_R(M)\cong 0$, and hence $M$ must be projective. Putting $A\coloneqq \enn_R(R\oplus M)$ and $e\coloneqq \id_R$, it now follows that $\dq$ is acyclic. In other words, if $R$ has dimension strictly greater than 3, then our methods do not yield much of interest. So in what follows, one may harmlessly assume that $R$ is a surface or a threefold, in which case \ref{twoorthreecor} gives easily checked criteria for when a general MCM module is modifying.
\end{rmk}
Fix a MCM modifying $R$-module $M$ with no free summands. Put $V\coloneqq R\oplus M$ and\linebreak $A\coloneqq  \enn_R(V)$. By construction, $A$ comes with an idempotent $e=\id_R$ with $eAe\cong R$. By \ref{mcmsummod}, the $R$-module $V$ is modifying. Add copies of $R$ to (\ref{cosyz}) to get a short exact sequence
\begin{equation}\label{vcosyz}
0 \to V \to R^l \to \Omega V \to 0
\end{equation}
Apply $\R\hom_R(V,-)$ to (\ref{vcosyz}) and take cohomology to obtain a long exact sequence of $A$-modules 
\begin{equation}0\to \hom_R(V,V) \to \hom_R(V,R^l) \to \hom_R(V,\Omega V) \to \ext^1_R(V,V) \to \ext^1_R(V, R^l) \to \cdots
\end{equation}\label{vlong}Since $M$ is MCM, the $\ext^1_R(V, R^l)$ term vanishes, and we obtain an exact sequence \begin{equation}0\to \hom_R(V,V) \to \hom_R(V,R^l) \to \hom_R(V,\Omega V) \to \ext^1_R(V,V) \to 0 
\end{equation}\label{vlong2}
	Set $$T_A\coloneqq \mathrm{coker}\left[\hom_R(V,V) \to \hom_R(V,R^l)\right]\cong \ker\left[\hom_R(V,\Omega V) \to \ext^1_R(V,V)\right].$$
\begin{rmk}Note that if $M$ was rigid, then so is $V$, and we obtain $T_A\coloneqq \hom_{R}( V, \Omega V)$. In \cite{iwmaxmod}, $\Omega V$ is denoted either $\mu_R^+(V)$ or $\mu_R^-(V)$ (the two agree since $\Omega\cong\Omega^{-1}$).
\end{rmk}Note that the right $A$-module $T_A$ has a projective summand isomorphic to $\hom_R(V,R)$, and hence the ring $\enn_{A}(T_A)$ has an idempotent $\id_{\hom_R(V,R)}$.
\begin{thm}[Iyama--Wemyss]\label{mtilt}Put $B\coloneqq \enn_{R}(\Omega V)$ and $e\coloneqq \id_R \in B$. Put $B'\coloneqq \enn_A(T_A)$ and $e'\coloneqq \id_{\hom_R(V,R)} \in B'$.
	\begin{enumerate}
		\item[\emph{i)}] There is an isomorphism of $R$-algebras $B\cong B'$ that restricts to a ring isomorphism \linebreak $eBe \cong e'B'e'\cong R$.
	\item[\emph{ii)}] The map $\upmu_A\coloneqq \R\hom_A(T_A,-):D(A) \to D(B)$ is an equivalence. We call $\upmu_A$ the \textbf{mutation equivalence}.
	\end{enumerate}
\end{thm}
\begin{proof}This is essentially \cite[6.8]{iwmaxmod}. Note that because $V$ has a free summand, it is a generator. Iyama--Wemyss prove that $T_A$ is a tilting module, and hence induces an equivalence \linebreak $D^b(A) \to D^b(B)$, but this can be promoted to an equivalence $D(A) \to D(B)$ via Happel's theorem \cite{happel}.
\end{proof}
Starting from $B$, one can repeat the above constructions to obtain a tilting module $T_B$. By the above arguments, one obtains a derived equivalence $\upmu_B\coloneqq \R\hom_B(T_B,-):D(B) \to D(\enn_R(\Omega^2 V))$. However, since $\Omega^2 V \cong V$, there is an $R$-linear isomorphism $\enn_R(\Omega^2 V) \cong A$. By composition one gets an autoequivalence $\mm\coloneqq \upmu_B\circ\upmu_A: D(A)\to D(A)$. Writing $_BT_A\coloneqq T_A$ and $_AT_B\coloneqq T_B$, the (derived) hom-tensor adjunction gives an isomorphism $\mm\cong \R\hom_A(_AT_B\lot_B{_BT_A},-)$.
\begin{defn}\label{immdefn}
	We call $\mm: D(A) \to D(A)$ the \textbf{mutation-mutation autoequivalence}. Write $I_\mm\coloneqq _AT_B\lot_B{_BT_A}$, so that $\mm$ is represented by the $A$-$A$-bimodule $I_\mm$.
\end{defn}
\begin{rmk}
	By construction, $_AT_B$ has a $B$-projective resolution of length 2, so it follows that $I_\mm$ has cohomology only in degrees 0 and -1.
\end{rmk}
\begin{rmk}\label{mutnchain}
	In general, not assuming that $R$ is a hypersurface, one obtains an a priori doubly infinite sequence of derived equivalences $$\cdots\xrightarrow{\cong} D(\enn_R(\Omega^i V)) \xrightarrow{\cong} D(\enn_R(\Omega^{i+1} V))\xrightarrow{\cong}\cdots$$
\end{rmk}
\begin{rmk}
	If $M$ is rigid then so is $\Omega M$, and it follows that $I_\mm\cong \hom_R(V,\Omega V)\lot_B\hom_R(\Omega V, V)$. In fact, in this situation one has $I_\mm\cong AeA$ by results of Donovan and Wemyss \cite[5.10]{DWncdf} which we review later in \ref{dwbimodmap}.
\end{rmk}
We finish with some results on the structure of mutation.
\begin{lem}\label{aconmutlemma}
	Let $X$ be an $A/AeA$-module placed in degree zero. Then there is a quasi-isomorphism of $B$-modules $\upmu_A(X)\simeq \ext^1_A({_BT_A},X)[-1]$, with the Ext-module placed in degree one.
\end{lem}
\begin{proof}Because $H^i(\upmu_A(X))\cong \ext^i_A({_BT_A},X)$, we need to show that $H(\upmu_A(X))$ is concentrated in degree one. By construction, ${_BT_A}$ comes with an $A$-projective resolution $$0 \to \hom_R(V,V) \to \hom_R(V,R^l) \to {_BT_A} \to 0.$$Note that $\hom_R(V,R)\cong eA$. Because $X$ is an object of the subcategory $D(\dq)$, there are no maps from $eA$ to $X$.
\end{proof}

\begin{lem}\label{shiftlem}
	Let $X \in D(\dq)$. Then for all $q\in \Z$ there are $B$-module isomorphisms $$H^{1+q}(\upmu_A(X))\cong H^1(\upmu_A(H^q(X))).$$
\end{lem}
\begin{proof}
	Take a $B$-$A$-bimodule quasi-isomorphism $P \to {_BT_A}$ that's a projective resolution of right $A$-modules. Consider the double complex of $B$-modules $E_0^{pq}\coloneqq \hom_A(P^{-p}, X^q)$ whose total product complex is $\mathrm{Tot}^{\Pi}(E_0)\cong \upmu_A(X)$. We may regard $E_0$, equipped with the differential of $X$, as the zeroth page of a (cohomological) spectral sequence $E$. Because $P$ is zero in positive degrees, it follows by the discussion after \cite[5.6.1]{weibel} that the spectral sequence $E$ weakly converges to $H^n(\mathrm{Tot}^{\Pi}(E_0)) \cong H^n(\upmu_A(X))$. An easy computation shows that $E_2^{pq}\cong H^{p}(\upmu_A(H^q(X)))$. By \ref{aconmutlemma}, this module is zero unless $p=1$, where it is $\ext^1_A(_BT_A, H^q(X))$. In other words, the spectral sequence collapses at the $E_2$ page. A weakly convergent spectral sequence which collapses must converge, and we get the desired isomorphisms.
\end{proof}
\begin{cor}\label{shiftcor}
	If $X \in D(A)$ satisfies $H^q(X)\cong 0$, then $H^{1+q}(\upmu_A(X))\cong 0$.
\end{cor}

\subsection{Recollements and t-structures}
	Throughout this part we will use the following setup:
\begin{setup}\label{mutsetup}
	Let $R$ be a complete local isolated hypersurface singularity of dimension at least 2, $M$ a MCM modifying $R$-module with no free summands, $V\coloneqq R\oplus M$, $A\coloneqq \enn_R(V)$, $B\coloneqq \enn_R(\Omega V)$, and $e=\id_R$ (we use the same notation for $\id_R \in A$ and $\id_R \in B$). 
\end{setup}

Let $\upmu_A$ be the mutation equivalence. Recall from \ref{recoll} the existence of the recollement \linebreak $D(\dq)\recol D(A)\recol D(R)$.
	\begin{defn}
	Let $C$, $C'$ be two dgas. Say that $C$ and $C'$ are \textbf{derived Morita equivalent} if there is a $C'\text{-}C$-bimodule $P$ such that $\R\hom_C(P,-):D(C)\to D(C')$ is a derived equivalence. Note that in this case the inverse is necessarily given by the functor $-\lot_{C'}P$.
\end{defn}
\begin{prop}Put $\upmu_{\mathbb{L}}\coloneqq \R\hom_{\dq}(\dqb\lot_B{_BT_A}\lot_A \dq,-) $. Then the diagram $$\begin{tikzcd}[column sep=huge]
	D(\dq)\ar[swap, dd,"\upmu_{\mathbb{L}} "] \ar[r,"i_*=i_!"]& D(A)\ar[l,bend left=25,"i^!"']\ar[l,bend right=25,"i^*"']\ar[r,"j^!=j^*"]\ar[swap, dd,"\upmu_A"] & D(R)\ar[l,bend left=25,"j_*"']\ar[l,bend right=25,"j_!"']\ar[swap, dd,"\id"]
	\\  & &
	\\ D(\dqb) \ar[r,"i_*=i_!"]& D(B)\ar[l,bend left=25,"i^!"']\ar[l,bend right=25,"i^*"']\ar[r,"j^!=j^*"] & D(R)\ar[l,bend left=25,"j_*"']\ar[l,bend right=25,"j_!"']
	\end{tikzcd}$$is a morphism of recollement diagrams, with vertical maps equivalences. In particular, $\dq$ and $\dqb$ are derived Morita equivalent.
\end{prop}
\begin{proof}
	First we check that the three squares on the right-hand side commute. Note that given an $A$-module of the form $\hom_R(V,X)$ then one has $\hom_R(V,X)e \cong X$. Since one has \linebreak $_BT_A\coloneqq \mathrm{coker}(\hom_R(V,V) \to \hom_R(V,R^l))$, and because the functor $Y \mapsto Ye $ is exact, we get $_BT_Ae\cong \mathrm{coker}(V \to R^l)\cong \Omega V \cong Be$ as $B-R$-bimodules. It now follows from unwinding the definitions and using $\upmu_A^{-1}\simeq -\lot_B{}_BT_A$ that $j^*_B \circ \upmu_A \simeq j^*_A$. and $\upmu_A \circ j_*^A \simeq j_*^B$. Similarly, we have $e_BT_A  \cong eA$ as $R$-$A$-bimodules, which implies that $\upmu_A \circ j_!^A \simeq j_!^B$. Now, because morphisms of recollements are determined uniquely by one half (e.g. \cite{kalck} 2.4), there's a unique (up to isomorphism) map $F:D(\dq) \to D(\dqb)$ fitting into a morphism of recollements with the righthand two, and since the righthand two are equivalences, so is $F$. Since the $i_*$ maps are fully faithful, $F$ is determined completely by $i_*F$: if $F'$ is any other functor such that $i_*F' \cong \upmu_A \circ i_*$, then $F'\cong F$. But one can check that the given functor satisfies this condition.
\end{proof}
Carrying out the same proof for $\upmu_B$ and composing vertical maps we arrive at:
\begin{cor}Put $I^\mathbb{L}_\mm\coloneqq \dq\lot_A{I_\mm}\lot_A \dq$ and $\mm_{\mathbb{L}}\coloneqq \R\hom_{\mathbb{L}}(I^\mathbb{L}_\mm,-) $. Then the diagram $$\begin{tikzcd}[column sep=huge]
	D(\dq)\ar[swap, dd," \mm_{\mathbb{L}} "] \ar[r,"i_*=i_!"]& D(A)\ar[l,bend left=25,"i^!"']\ar[l,bend right=25,"i^*"']\ar[r,"j^!=j^*"]\ar[swap, dd,"\mm"] & D(R)\ar[l,bend left=25,"j_*"']\ar[l,bend right=25,"j_!"']\ar[swap, dd,"\id"]
	\\  & &
	\\ D(\dq) \ar[r,"i_*=i_!"]& D(A)\ar[l,bend left=25,"i^!"']\ar[l,bend right=25,"i^*"']\ar[r,"j^!=j^*"] & D(R)\ar[l,bend left=25,"j_*"']\ar[l,bend right=25,"j_!"']
	\end{tikzcd}$$is a morphism of recollement diagrams, with vertical maps equivalences.
\end{cor}
\begin{proof}
	The only thing we need to check is the implicit assertion that there's a quasi-isomorphism of $\dq$-bimodules $$\dq\lot_A{_AT_B}\lot_B \dqb\lot_B \dqb\lot_B{_BT_A}\lot_A \dq\simeq I^\mathbb{L}_\mm$$ But it follows by considering the representing objects of both sides of the equation $$\R\hom_A({_BT_A},i_*-) \simeq i_*i^!\R\hom_A({_BT_A},i_*)$$that ${_BT_A} \lot _A \dq \simeq \dqb\lot_B {_BT_A}\lot_A \dq$ as $B$-$A$-bimodules, and similarly for \linebreak $\dq\lot_A{_AT_B}$. Hence, we have quasi-isomorphisms \begin{align*}
	&\dq\lot_A{_AT_B}\lot_B \dqb\lot_B \dqb\lot_B{_BT_A}\lot_A \dq \\\simeq &\dq\lot_A{_AT_B}\lot_B {_BT_A}\lot_A \dq \\ \eqqcolon &\dq\lot_AI_{\mm}\lot_A \dq
	\\  \eqqcolon &I^\mathbb{L}_\mm
	\end{align*} as required.
\end{proof}

We combine our results on recollements with some standard facts about t-structures; see \cite{bbd} for the definition of a t-structure.
\begin{prop}[\cite{hkmtstrs, amiotcluster, kelleryangmutn, kalckyang}]\label{tstr}
	Let $Z$ be a nonpositive dga. Then the derived category $D(Z)$ admits a $t$-structure $(D^{\leq 0}(Z), D^{\geq 0}(Z))$ where 
	\begin{align*}
	&D^{\leq 0}(Z)\coloneqq \{X\in D(Z): H^i(X)=0 \text{ \normalfont for } i>0 \}
	\\ 	&D^{\geq 0}(Z)\coloneqq \{X\in D(Z): H^i(X)=0 \text{ \normalfont for } i<0 \}
	\end{align*}
	Moreover, the inclusion $\cat{Mod}\text{-}H^0(Z) \into D(Z)$ is an equivalence onto the heart of this t-structure, with inverse given by taking zeroth cohomology.
	
\end{prop}
\begin{rmk}
	When $Z=\dq$ this is the restriction of the standard t-structure on $D(A)$.
\end{rmk}

\begin{prop}[cf. {\cite[4.7]{bridgeland}}]\label{mutntexact}
	The shifted mutation functor $$ X \mapsto \upmu_{\mathbb{L}}(X)[1]:\quad D(\dq) \longrightarrow D(\dqb)$$ is a t-exact equivalence.
\end{prop}
\begin{proof}For brevity, write the functor under consideration as $G$. By construction, $G$ is an equivalence. By \ref{shiftcor}, if $X$ is concentrated in nonnegative degrees, then so is $G(X)$, and similarly for nonpositive degrees. In other words, $G$ is t-exact.
\end{proof}
Taking hearts one arrives at:
\begin{cor}
	$A/AeA$ and $B/BeB$ are (classically) Morita equivalent.
\end{cor}
\begin{defn}
	Let $D$ be the t-structure  on $D(\dq)$ constructed in \ref{tstr}. Let $\tau_A^p$ be the t-structure on $D(A)$ obtained by gluing $D[-p]$ (i.e. $D$ shifted so that the heart is in degree $p$) to the standard t-structure on $D(R)$. In particular, $\tau_A^0$ is the natural t-structure on $D(A)$. Write ${}^p\mathrm{Per}A$ for the heart of $\tau_A^p$, so that e.g. ${}^0\mathrm{Per}A=\cat{Mod}\text{-}A$. Call ${}^p\mathrm{Per}A$ the abelian category of \textbf{$p$-perverse sheaves} on $A$. 
\end{defn}
\begin{thm}[cf. {\cite[4.8]{bridgeland}}]
	Fix a natural number $p$. Then the mutation functor \linebreak $\upmu_A:D(A)\to D(B)$ is $t$-exact for the t-structures $\tau^p_A$ and $\tau^{p+1}_B$. Mutation induces a chain of exact equivalences of abelian categories $$\cdots \to {}^p\mathrm{Per}A \to {}^{p+1}\mathrm{Per}B \to {}^{p+2}\mathrm{Per}A \to \cdots$$
\end{thm}
\begin{proof}
	$\upmu_A$ is t-exact because it is the gluing of two $t$-exact functors. Similarly, $\upmu_B$ is t-exact, and the chain of \ref{mutnchain} becomes a chain of t-exact equivalences. Passing to hearts gets us the second statement. 
\end{proof}

\subsection{Simple modules and deformations}
	Throughout this section we will use the following setup:
\begin{setup}\label{mutsetupp}
	Assume that we are in the situation of Setup \ref{mutsetup}. Assume furthermore that $A/AeA$ is Artinian local and that the dga $\dq$ is cohomologically locally finite.
\end{setup}
In the geometric situations that we care about, the hypotheses of Setup \ref{mutsetupp} are always satisfied (\ref{fdcohom}). Denote by $S_A$ the one-dimensional $A$-module $A/AeA / \text{rad}(A/AeA)$. Since it is naturally an $A/AeA$-module, we may regard it as an object of $D({\dq})$ concentrated in degree zero. Similarly, we denote the analogous one-dimensional $B$-module by $S_B$.
\begin{lem}\label{msissone}
	There is a quasi-isomorphism of $B$-modules $\upmu_A(S_A)\simeq S_B[-1]$.
\end{lem}
\begin{proof}
	A computation using \ref{aconmutlemma} shows that $\upmu_A(S_A)$ is a 1-dimensional object in $D(\dqb)$ concentrated in degree 1. Since it is hence a simple one-dimensional module over $B/BeB$, it must be a copy of $S_B$.
\end{proof}

\begin{cor}[cf.\ {\cite[5.11]{DWncdf}}]
	$\mm(S_A)\simeq S_A[-2]$.
\end{cor}

By \ref{dcaprorep}, $\dq$ prorepresents the derived noncommutative deformations of $S_A$. In particular, we may regard it as the universal deformation of $S_A$, which is invariant under derived equivalences. This motivates the following:
\begin{thm}\label{dictionary}With the setup as above, \hfill
	\begin{enumerate}
		\item[\emph{i)}] The dgas ${\dq}$ and $\dqb$ are quasi-isomorphic over $k$.
		\item[\emph{ii)}] The $k$-algebras $A/AeA$ and $B/BeB$ are isomorphic.
		\item[\emph{iii)}] The $B$-module $\upmu_A({\dq})$ is quasi-isomorphic to $\dqb[-1]$.
		\item[\emph{iv)}] The $B$-module $\upmu_A(A/AeA)$ is quasi-isomorphic to $B/BeB[-1]$. 
	\end{enumerate}
\end{thm}
\begin{proof}
	Since $\upmu_A:D(A) \to D(B)$ is a standard equivalence, it can be enhanced to a quasi-equivalence of dg categories, and it follows that $\R\enn_A(S_A)\simeq\R\enn_B(\upmu_A(S_A))$ as dgas. By \ref{msissone} we have $\R\enn_A(S_A)\simeq \R\enn_B(S_B)$. Taking the Koszul dual of both sides and appealing to \ref{dqiskd} gives the first claim. To see the second claim, now simply observe that $$A/AeA\cong H^0({\dq})\cong H^0(\dqb)\cong B/BeB.$$The third claim takes some more work and needs the deformation-theoretic interpretation. First note that the $A$-module ${\dq}$ is the universal deformation for the $A$-module $S_A$. Since $\upmu_A$ is an equivalence, the $B$-module $\upmu_A({\dq})$ is the universal deformation for the $B$-module $\upmu_A(S_A)$. But the universal deformation for $\upmu_A(S_A)\simeq S_B[-1]$ is just $\dqb[-1]$, which is the third claim. Applying \ref{shiftlem} now gets us isomorphisms $H^{q}(\dqb)\cong H^1(\upmu_A(H^q(\dq)))$ for all $q\in \Z$. Taking $q=0$ gets us an isomorphism $H^1(\upmu_A(A/AeA))\cong B/BeB$. But because $\upmu_A(A/AeA)$ is a module placed in degree one by \ref{aconmutlemma}, it follows that $\upmu_A(A/AeA)\simeq H^1(\upmu_A(A/AeA))[-1]$, which is the fourth claim.
\end{proof}
\begin{rmk}
	Part i) provides a new proof that $A/AeA\cong B/BeB$ as algebras; see \linebreak \cite[6.20]{iwmaxmod} for an alternate proof that uses the interpretation as stable endomorphism rings. Part iv) is analogous to \cite[5.9(1)]{DWncdf}.
\end{rmk}
\begin{rmk}
	By working in the appropriate subcategories, we may promote the quasi-isomorphisms of iii) and iv) to quasi-isomorphisms of $\dqb$-modules, and that of iv) to an isomorphism of $B/BeB$-modules.
\end{rmk}
\subsection{Singularity categories}
Assume that we are in the setup of \ref{mutsetupp}. For brevity, we put $Q\coloneqq \dq$.

\begin{lem}\label{dwbimodmap}
	There is a map $I_\mm \to A$ in the derived category of $A$-bimodules. 
\end{lem}
\begin{proof}
	By definition, $I_\mm\coloneqq {_A}T_B\lot_B{_BT_A}$. By construction, $_AT_B$ admits a map to $\hom_R(\Omega V, V)$, and so $I_\mm$ admits a map to ${\hom_R(\Omega V,V)\lot_B {_BT_A}}$. The proof of \cite[5.10]{DWncdf} adapts to show that ${\hom_R(\Omega V,V)\lot_B {_BT_A}}$ is quasi-isomorphic to $AeA$. Now compose the map $I_\mm \to AeA$ with the inclusion $AeA \into A$.
\end{proof}
\begin{rmk}
	Because $_AT_B$ fits into a short exact sequence of $B$-modules $$0 \to {_AT_B} \to \hom_R(V,\Omega V) \to \ext^1_R(V,V)\to 0$$the cone of $I_\mm \to AeA$ is quasi-isomorphic to $\ext^1_R(V,V)\lot_B{_BT_A}$. If $M$ is rigid, then so is $V$, and one gets a quasi-isomorphism $I_\mm \simeq AeA$, which recovers \cite[5.10]{DWncdf}.
\end{rmk}
Tensoring $I_\mm \to A$ with ${Q}$ on both sides gives a bimodule map $I^\mathbb{L}_\mm \to {Q}$. Similarly, applying $e$ gives an $A$-$R$-bimodule map $I_\mm e \to Ae$, and one can check that it is a quasi-isomorphism. Looking at representables gives us natural transformations $\id_{D(A)}\to \mm$ and $\id_{D({Q})}\to \mm_{\mathbb{L}}$, which are compatible with the recollements.
\begin{defn}
Let $D_\mathrm{fd}({Q})$ be the subcategory of $D(Q)$ on those objects with finite-dimensional total cohomology. Let $\per_\mathrm{fd}({Q})$ be the subcategory of $\per Q$ on those objects with finite-dimensional total cohomology.  Write $\mathcal{M}$ for $\thick_{D_{\mathrm{sg}}(R)}(M)$.
\end{defn}
\begin{prop}[{\cite[4.6.5]{dqdefm}}]
The composition $\Sigma\coloneqq j^*(i^*)^{-1}$ is a well-defined triangle equivalence $\Sigma:\per(Q)/\per_\mathrm{fd}(Q)\to \mathcal{M}$.
	\end{prop}
\begin{lem}\label{fdresplem}
	The functor $\mm_{\mathbb{L}}:D(Q) \to D(Q)$ respects $\per({Q})$, $D_\mathrm{fd}({Q})$, and $\per_\mathrm{fd}({Q})$.
\end{lem}
\begin{proof}
	Since $\mm_{\mathbb{L}}({Q})$ is perfect, it follows that $\mm_{\mathbb{L}}$ preserves all perfect modules. Since finite-dimensional modules are built out of $S_A$ under cones and shifts, and $\mm$ sends $S_A$ to a finite-dimensional module, it follows that $\mm$ preserves $D_\mathrm{fd}({Q})$. The third assertion is  now clear.
\end{proof}
\begin{defn}
	Let $\mm_{\mathrm{sg}}:\mathcal{M}\to \mathcal{M}$ be the autoequivalence defined by the commutative diagram of equivalences $$\begin{tikzcd} \per({Q})/\per_\mathrm{fd}({Q})\ar[d,"\mm_{\mathbb{L}}"] \ar[r,"\Sigma"]& \mathcal M \ar[d, "\mm_{\mathrm{sg}}"] \\ \per({Q})/\per_\mathrm{fd}({Q})\ar[r,"\Sigma"] & \mathcal M 
	\end{tikzcd}$$Observe that one gets a natural transformation $\id_{\mathcal M} \to \mm_{\mathrm{sg}}$.
\end{defn}
\begin{rmk}
	One can enhance $\mm_{\mathrm{sg}}$ to a dg functor, although we will not need this fact.
\end{rmk}
The following is the key technical observation which gives us control over $\mm_{\mathbb{L}}$: 
\begin{prop}
	The natural transformation $\id_{\mathcal M} \to \mm_{\mathrm{sg}}$ is an isomorphism.
\end{prop}
\begin{proof}The idea is that all constructions respect the recollement, which forces $\mm_{\mathrm{sg}}$ to agree with the map induced by $\mm: D(R)\to D(R)$, which is just the identity. Because $\mm$ preserves both $\per{Q}$ and $\per R$, it follows that it preserves $\per A$ too, and moreover descends to an autoequivalence of $\per A / j_! \per R$. By the definition of $\Sigma$, one has a commutative diagram $$\begin{tikzcd}
	\per A \ar[d,"i^*"]\ar[r]& \per A / j_! \per R \ar[dl,"i^*"]\ar[d]\ar[dr,".e"]& 
	\\ \per({Q}) \ar[r]& \per({Q})/\per_\mathrm{fd}({Q})\ar[r,"\Sigma"'] & \mathcal{M}
	\end{tikzcd}$$and by \ref{fdresplem} and the definition of $\mm_{\mathrm{sg}}$ one has a commutative diagram $$\begin{tikzcd}
	\per({Q}) \ar[r]\ar[d,"\mm_{\mathbb{L}}"]& \per({Q})/\per_\mathrm{fd}({Q})\ar[r,"\Sigma"']\ar[d,"\mm_{\mathbb{L}}"] & \mathcal M \ar[d,"\mm_{\mathrm{sg}}"]
	\\ \per({Q}) \ar[r]& \per({Q})/\per_\mathrm{fd}({Q})\ar[r,"\Sigma"'] & \mathcal M
	\end{tikzcd}$$Gluing two copies of the first to the second, one can check that the induced diagram $$\begin{tikzcd}
	\per A \ar[d,"\mm"]\ar[r]& \per A / j_! \per R\ar[d,"\mm"]\ar[r,".e"] & \mathcal M\ar[d,"\mm_{\mathrm{sg}}"]
	\\ \per A \ar[r]& \per A / j_! \per R \ar[r,".e"] & \mathcal M
	\end{tikzcd}$$ commutes. A similar argument with the commutative diagram $$\begin{tikzcd}
	\per A \ar[rr,".e"]\ar[d]&& D^b(R)\ar[d]
	\\ \per A / j_! \per R\ar[r,".e"] & \mathcal{M}\ar[r,hook] & D_{\mathrm{sg}}(R) 
	\end{tikzcd}$$ shows that the diagram $$\begin{tikzcd}
	D^b(R)\ar[d,"\mm"]\ar[rr,bend left=15] & \mathcal{M}\ar[r,hook]\ar[d,"\mm_{\mathrm{sg}}"]& D_\mathrm{sg}(R)\ar[d,"\mm"]
	\\ D^b(R)\ar[rr,bend right=15] & \mathcal{M}\ar[r,hook]& D_\mathrm{sg}(R)
	\end{tikzcd}$$ commutes. But the left-hand vertical map is the identity, so it follows that $\mm_{\mathrm{sg}}$ is the identity map. Moreover, $\id \to \mm$ descends to an isomorphism $\id \to \mm_{\mathrm{sg}}$ because it descends to an isomorphism on $D^b(R)$.
\end{proof}

\subsection{Periodicity and localisation}
Assume that we are in the setup of \ref{mutsetupp}. For brevity, we put $Q\coloneqq \dq$. Recall from \ref{etaex} the existence of the periodicity element $\eta \in H^{-2}(Q)$. Let $E$ be the derived localisation of ${Q}$ at $\eta$ (see \cite{bcl} for the definition of derived localisation).

\begin{prop}There is a commutative diagram of dg categories $$\begin{tikzcd} \per Q\ar[dr, "\Sigma"] \ar[r, "-\lot_Q E"]& \per(E)\ar[d,"\alpha"] \\  & \mathcal{M} \end{tikzcd}$$where $\Sigma$ is the ``singularity functor'' of {\normalfont{\cite[4.6.5]{dqdefm}}} and $\alpha$ is a quasi-equivalence.
\end{prop}
\begin{proof}
Recall from \cite[5.3.7]{dqdefm} that $M$ has dg endomorphism ring $E$ and the functor $\alpha$ is the induced equivalence which sends $E$ to $M$. The singularity functor sends $Q$ to the object $M$, and on endomorphism dgas is the inclusion $Q\to E$. As the dg functors hence agree on generators, they agree on the whole categories.
	\end{proof}
\begin{defn}
	Let $\mm_E:\per E \to \per E$ be the endofunctor defined by $$\mm_E\coloneqq \alpha \mm_{\mathrm{sg}}\alpha^{-1}.$$
	\end{defn}
The following is clear:
\begin{lem}
	The isomorphism $\id_{\mathcal M} \to \mm_\mathrm{sg}$ induces an isomorphism $\id_{\per E} \to \mm_E$.
\end{lem}
\begin{lem}
	The following diagram is commutative: $$\begin{tikzcd} \per Q\ar[d,"\mm_{\mathbb{L}}"] \ar[r, "-\lot_Q E"]& \per E\ar[d,"\mm_{E}"] \\ \per Q \ar[r, "-\lot_Q E"]& \per E \end{tikzcd}$$
	\end{lem}
\begin{proof}
	Follows from the definition of $\mm_E$ along with the fact that $\mm_{\mathrm{sg}}$ commutes with $\Sigma$.
	\end{proof}

\begin{lem}\label{immltriv}
	Applying $-\lot_{{Q}}E$ to the ${Q}$-bimodule map $I_\mm^\mathbb{L} \to {Q}$ gives a ${Q}$-$E$ bimodule quasi-isomorphism $I_\mm^\mathbb{L} \lot_{{Q}} E \to E$.
\end{lem}
\begin{proof}The idea is to look at $\mm_{\mathbb{L}} ^{-1} \to \id$, which becomes a quasi-isomorphism upon inverting $\eta$. Consider the functor $\mm_{\mathbb{L}} ^{-1}: \per Q \to \per Q$ which sends $X$ to $X\lot_Q I^\mathbb{L}_\mm$. It comes with a natural transformation $\mm_{\mathbb{L}} ^{-1}\to \id$, which descends to $\per E$, and is an isomorphism there. Hence, the natural map $X \lot_Q I_\mm^\mathbb{L} \lot_Q E \to X \lot_Q E$ is a quasi-isomorphism of $E$-modules. Hence the natural map $I_\mm^\mathbb{L} \lot_Q E \to E$ must be a quasi-isomorphism of $Q$-$E$-bimodules.
\end{proof}

\begin{prop}\label{mmshiftrep}
	As $Q$-bimodules, $I^\mathbb{L}_\mm$ is quasi-isomorphic to $Q[2]$.
	\end{prop}
\begin{proof}For brevity, write $I\coloneqq I_\mm^\mathbb{L}$. By \ref{dictionary}, {iii)} applied twice, one gets a $Q$-module quasi-isomorphism $\mm(Q)\simeq Q[-2]$ and hence a $Q$-module quasi-isomorphism $\mm^{-1}(Q)\simeq Q[2]$. Now it follows that $I$ is quasi-isomorphic to $Q[2]$ as right $Q$-modules. Pick a right quasi-isomorphism $I\xrightarrow{f} Q[2]$ and tensor it with $E$ to get a commutative diagram
$$\begin{tikzcd}Q[2]\ar[r,"g'"] & Q[2]\lot_Q E \\ I\ar[u,"f"]\ar[r,"g"] & I\lot_Q E\ar[u,"f'"]
\end{tikzcd}$$where $f$ and $f'$ are right $Q$-module quasi-isomorphisms and $g$ and $g'$ are $Q$-bimodule maps. Truncate this diagram to degrees weakly below $-2$ to get a commutative diagram
$$\begin{tikzcd}Q[2]\ar[r,"v'"] & \tau_{\leq -2}\left(Q[2]\lot_Q E\right) \\ I\ar[u,"u"]\ar[r,"v"] & \tau_{\leq -2}\left(I\lot_Q E\right) \ar[u,"u'"]
\end{tikzcd}$$where, as before, $u$ and $u'$ are right $Q$-module quasi-isomorphisms and $v$ and $v'$ are $Q$-bimodule maps. After identifying $Q[2]\lot_Q E$ with $E[2]$, we may identify $g'$ with the shifted localisation map $Q[2]\to E[2]$. By \cite[4.8.4, 5.2.4]{dqdefm}, the localisation map $Q \to E$ induces a quasi-isomorphism $Q \to \tau_{\leq 0}E$. Hence $v'$ is a quasi-isomorphism. Now it follows that $v$ is a quasi-isomorphism too. So as a bimodule, $I$ is quasi-isomorphic to $\tau_{\leq -2}\left(I\lot_Q E\right)$ and it hence remains to show that $\tau_{\leq -2}\left(I\lot_Q E\right)$ is bimodule quasi-isomorphic to $Q[2]$. From \ref{immltriv}, one has a $Q$-$E$-bimodule quasi-isomorphism, and hence a $Q$-bimodule quasi-isomorphism, $I_\mm^\mathbb{L} \lot_{{Q}} E \to E$. So it remains to show that $\tau_{\leq -2}E \simeq Q[2]$ as $Q$-bimodules. But one has $\tau_{\leq -2}E \simeq \tau_{\leq -2}Q$, and furthermore by \ref{etaex}, {iv)}, there is a $Q$-bimodule quasi-isomorphism $ \tau_{\leq -2}Q \simeq Q[2]$.
	\end{proof}

\begin{cor}\label{mmshift}
The autoequivalence $\mm_{\mathbb{L}}:D(Q) \to D(Q)$ is isomorphic to the shift $[-2]$.
\end{cor}
\begin{proof}
Follows immediately by looking at representing objects.
\end{proof}

\begin{thm}\label{mutncontrol}
Let $R$ be a complete local isolated hypersurface singularity of dimension at least 2, $M$ a MCM modifying $R$-module with no free summands, $A\coloneqq \enn_R(R\oplus M)$, and $e=\id_R\in A$. Then the dga $A_\mm\coloneqq \tau_{\geq -1}(\dq)$ controls the mutation-mutation autoequivalence $\mm: D(A)\to D(A)$, in the sense that $\mm$ is represented by the cocone of the natural map $A \to A_\mm$.
\end{thm}
\begin{proof}
	This is fairly straightforward once one looks at the proof of \ref{mmshiftrep}. For brevity denote the cocone of $A \to A_\mm$  by $C$. First observe that one only needs to check the statement on the pieces of the recollement $D(\dq)\recol D(A) \recol D(R)$ individually. On $D(R)$, we have $\mm\cong\id$, and it's easy to see that $eCe\simeq R$. On $D(\dq)$, we have $\mm\cong [-2]$, so we need to check that $\dq \lot_A C \lot_A \dq \simeq \dq[2]$. But $\dq \lot_A C \lot_A \dq$ is precisely $\tau_{\leq -2}(\dq)$, which is (bimodule) quasi-isomorphic to $\dq[2]$ by \ref{etaex}, {iv)}.
\end{proof}

\begin{rmk}\label{twistrmk}
	We remark that $\mm$ can be interpreted as a sort of `noncommutative twist' around $A_\mm$. We follow the proofs in \cite[\S6.3]{DWncdf}; see also \cite{segaltwists} for background on twists. Let\linebreak $F:D(A_\mm) \to D(A)$ be restriction of scalars along $A \to A_\mm$. First observe that $F$ has right and left adjoints given by $R\coloneqq \R\hom_A(A_\mm,-)$ and $L\coloneqq -\lot_A A_\mm$ respectively. By the above, we have an exact triangle of $A$-bimodules $I_\mm \to A \to A_\mm \to$. Applying derived hom and tensor respectively gives exact triangles of endofunctors of $D(A)$ of the form \begin{align*}
	& FR\to \id \to \mm\to \\ & \mm^{-1} \to \id \to FL \to
	\end{align*}
	using that $\R\hom_A(A_\mm,-)\simeq FR$ and $-\lot_A A_\mm\simeq FL$.
\end{rmk}
\begin{rmk}
	By definition, we have a distinguished triangle of $Q$-bimodules $$Q[2]\to Q \to A_\mm\to.$$ Applying $\R\hom_Q(-,A_\mm)$ to this triangle gives us a distinguished triangle of $Q$-bimodules $$\R\enn_Q(A_\mm) \to A_\mm\to A_\mm[-2] \to $$and one can use the long exact sequence to show that, as a $Q$-bimodule, one has an isomorphism $$\ext_A^*(A_\mm, A_\mm)\cong \ext_Q^*(A_\mm, A_\mm)\cong H^*(A_\mm) \oplus H^*(A_\mm)[-3].$$In particular if $A_\mm$ has no cohomology in degree -1 then $A_\mm\simeq A_\con$, and one sees that \linebreak $\ext_A^*(A_\con, A_\con)\cong A_\con \oplus A_\con[-3]$.
\end{rmk}

\begin{rmk}
	Let $\Gamma=\Gamma^{-1}\to \Gamma ^0$ be a $[-1,0]$-truncated noncommutative Artinian dga. As in \ref{prorepremk}, the inclusion-truncation adjunction gives an isomorphism between $\hom(\dq,\Gamma)$ and $\hom(A_\mm,\Gamma)$, and we see that $A_\mm$ controls the $[-1,0]$-truncated derived noncommutative deformations of the simple module $S$.
\end{rmk}

\begin{prop}
	The complex $I_\mm$ is a module, which moreover fits into a short exact sequence of $A$-bimodules
	$$0\to H^{-1}(\dca) \to I_\mm \to AeA\to 0.$$
\end{prop}	
\begin{proof}
	By \ref{mutncontrol}, we have a distinguished triangle of $A$-bimodules $I_\mm \to A \to A_\mm \to$. Because $A$ has cohomology only in degree 0, and $A_\mm$ has cohomology only in degrees $0$ and $-1$, the long exact sequence in cohomology tells us that $I_\mm$ has cohomology only in degree zero. In this degree, the long exact sequence turns into an exact sequence $$0 \to H^{-1}(\dca) \to I_\mm \to A \to A_\con\to 0$$where the rightmost map is the standard projection. Replacing $A \to A_\con$ by its kernel $AeA$ gives the desired result.
\end{proof}
\begin{cor}
	The following are equivalent:
	\begin{itemize}
		\item The map $I_\mm \to AeA$ is an isomorphism.
		\item The map $A_\mm \to A_\con$ is a quasi-isomorphism.
		\item The $R$-module $M$ is rigid.
	\end{itemize}	
\end{cor}
\begin{rmk}\label{mtnrmk}
	In particular, if $X\to \spec R$ is a minimal model of a three-dimensional terminal singularity, then the module $M$ defining the noncommutative model $A$ is rigid and we have \linebreak $I_\mm\cong AeA$, which provides a new proof of \cite[5.10]{DWncdf}. If $R$ is a surface, then $M$ is never rigid by AR duality, and in particular $A_\con$ never controls $\mm$ by noncommutative twists.
\end{rmk}

\appendix
\section{$A_\infty$-algebras}\label{ainfty}
We collect some material about $A_\infty$-algebras that will be useful for our computations. In this appendix we work over a field $k$; in all of our applications $k$ will be algebraically closed and characteristic zero but one does not need either of these hypotheses. We'll broadly follow the treatment of Keller in \cite{kellerainfty}.
\subsection{Basic definitions}
\begin{defn}
	An $A_\infty$-algebra over $k$ is a graded $k$-vector space $A$ together with, for each $n\geq 1$, a $k$-linear map $m_n: A^{\otimes n} \to A$ of degree $2-n$ satisfying for all $n$ the coherence equations (or the \textbf{Stasheff identities}) $$\mathrm{St}_n:\quad \sum{(-1)^{r+st}}m_{r+1+t}(1^{\otimes r} \otimes m_s \otimes 1^{\otimes t})=0$$ where $1$ indicates the identity map, the sum runs over decompositions $n=r+s+t$, and all tensor products are over $k$. We're following the sign conventions of \cite{getzlerjones}; note that other sign conventions exist in the literature (e.g. in \cite{lefevre}).
\end{defn}
\begin{rmk}The original motivation for the definition came from Stasheff's work on $A_\infty$-spaces in \cite{stasheff}. If $X$ is a pointed topological space and $\Omega X$ its loop space, then we have a `composition of loops' map $\Omega X \times \Omega X \to \Omega X$. It's not associative, but it is associative up to homotopy. Similarly, one can bracket the product of four loops $a.b.c.d$ in five different ways, and one obtains five homotopies fitting into the Mac Lane pentagon. These homotopies are further linked via higher homotopies; we get an infinite-dimensional polytope $K$ the \textbf{associahedron} with $(n-2)$-dimensional faces $K_n$ corresponding to the homotopies between compositions of $n$ loops. An $A_\infty$\textbf{-space} is a topological space $Y$ together with maps $f_n: K_n \to Y^n$ satisfying the appropriate coherence conditions. For example a loop space is an $A_\infty$-space. Then, if $Y$ is an $A_\infty$-space, then the singular chain complex of $Y$ is an $A_\infty$-algebra.
\end{rmk}
For readability, we'll often write $a_1 \cdot a_2$ to mean $a_1\otimes a_2$ (multiplication in the tensor algebra). Suppose that $A$ is an $A_\infty$-algebra. Then $\mathrm{St}_1$ simply reads as $m_1^2=0$; in other words $m_1$ is a differential on $A$. Hence we may define the cohomology $HA$. The next identity $\mathrm{St}_2$ tells us that $m_1m_2=m_2(m_1\cdot 1 - 1\cdot m_1)$; in other words $m_2$ is a derivation on $(A,m_1)$. The third identity $\mathrm{St}_3$ yields $$m_2(1\cdot m_2 - m_2 \cdot 1)=m_1m_3+m_3(\sum_{i+j=2}1^{\cdot i}\cdot m_1 \cdot 1^{\cdot j})$$The left hand side is the associator of $m_2$, and the right hand side is the boundary of the map $m_3$ in the complex $\hom(A^{\otimes 3}, A)$. Hence, $m_2$ is a homotopy associative `multiplication' on $A$. In particular, we obtain:
\begin{lem}
	Suppose that $A$ is an $A_\infty$-algebra with $m_3=0$. Then $(A,m_1,m_2)$ is a dga. Similarly, if $A$ is any $A_\infty$-algebra, then $(HA,[m_2])$ is a graded algebra. Conversely, if $(A,d,\mu)$ is a dga, then $(A,d,\mu,0,0,0,\cdots)$ is an $A_\infty$-algebra.
\end{lem}

Additional signs arise in the above formulas via the Koszul sign rule when one wants to put elements into them.  The following lemma is extremely useful:
\begin{lem}
	Fix positive integers $n=r+s+t$ and $n$ homogeneous elements $a_1,\ldots,a_n$ in $A$. Then $$(1^{\cdot r} \cdot m_s \cdot 1^{\cdot t})(a_1\cdots a_n) = (-1)^\epsilon a_1\cdots a_r\cdot m_s(a_{r+1}\cdots a_{r+s})\cdot a_{r+s+1}\cdots a_{n}$$where $\epsilon=s\sum_{j=1}^r {|a_j|}$. In particular, if $s$ is even then the na\"ive choice of sign is the correct one.
\end{lem}
\begin{proof}The Koszul sign rule gives a power of $|m_s|\sum_{j=1}^r {|a_j|}$, which has the same parity as $\epsilon$.
\end{proof}
An $A_\infty$-algebra $A$ is \textbf{strictly unital} if there exists an element $\epsilon\in A^0$ such that $m_1(\epsilon)=0$, $m_1(\epsilon,a)=m_2(a,\epsilon)=a$, and if $n>2$ then $m_n$ vanishes whenever one of its arguments is $\epsilon$.
\begin{defn}
	Let $A$ and $B$ be $A_\infty$-algebras. A \textbf{morphism} is a family of degree $1-n$ linear maps $f_n: A^{\otimes n} \to B$ satisfying the identities $$\sum_{n=r+s+t}(-1)^{r+st}f_{r+1+t}(1^{\otimes r} \otimes m_s \otimes 1^{\otimes t}) = \sum_{i_1+\ldots+ i_r=n}(-1)^{\sigma(i_1,\ldots,i_n)}m_r(f_{i_1}\otimes\cdots\otimes f_{i_r})$$where $\sigma(i_1,\ldots,i_n)$ is the sum $\sum_j (r-j)(i_j-1)$ (note that only terms with $r-j$ odd and $i_j$ even will contribute to the sign).
\end{defn}
In particular, $f_1$ is a chain map. A morphism $f$ is \textbf{strict} if it's a chain map; i.e. $f_n=0$ for $n>1$. A morphism $f$ is a \textbf{quasi-isomorphism} if $f_1$ is. One can compose morphisms by setting $(f\circ g)_n=\sum_{i_1+\ldots+ i_r=n}(-1)^{\sigma(i_1,\ldots,i_n)}f_r \circ(g_{i_1}\otimes\cdots\otimes g_{i_r})$.
\subsection{Coalgebras and homotopy theory}
We give an alternate quick definition of an $A_\infty$-algebra. We recall that a \textbf{dg-coalgebra} (\textbf{dgc} for short) is a comonoid in the category of dg vector spaces over $k$. More concretely, a dgc is a complex $(C,d)$ equipped with a comultiplication $\Delta:C \to C \otimes C$ and a counit $\epsilon:C \to k$, satisfying the appropriate coassociative and counital identities, and such that $d$ is a coderivation for $\Delta$.  A \textbf{coaugmentation} on a dgc is a section of $\epsilon$; if $C$ is coaugmented then $\bar C\coloneqq  \ker \epsilon$ is the \textbf{coaugmentation coideal}. It is a dgc under the reduced coproduct $\bar{\Delta}x = \Delta x -x\otimes 1 -1\otimes x$, and $C$ is isomorphic as a nonunital dgc to $\bar C \oplus k$. A coaugmented dgc $C$ is \textbf{conilpotent} if every $x \in \bar C$ is annihilated by some suitably high power of $\Delta$.
\begin{ex}
	If $V$ is a dg-vector space, then the tensor algebra ${T}^c(V)\coloneqq k\oplus V \oplus V^{\otimes 2} \oplus\cdots$ is a dg-coalgebra when equipped with the \textbf{deconcatenation coproduct} ${T}^c(V) \to {T}^c(V) \otimes {T}^c(V)$ which sends $v_1\cdots v_n$ to $\sum_i v_1\cdots v_i \otimes v_{i+1}\cdots v_n$. The differential is induced from the differential on $V^{\otimes n}$. It is easy to see that ${T}^c(V)$ is conilpotent, since $\bar{\Delta}^{n+1}(v_1\cdots v_n)=0$. The \textbf{reduced tensor coalgebra} is the nonunital dgc $\bar{T}^c(V)\coloneqq V \oplus V^{\otimes 2} \oplus\cdots$.
\end{ex}
In fact, ${T}^c$ is the cofree conilpotent coalgebra functor: if $C$ is conilpotent then $C \to{T}^c(V)$ is determined completely by the composition $l:C \to {T}^c(V) \to V$. In particular, any morphism $f:\bar{T}^c(W) \to \bar{T}^c(V)$ is determined completely by its \textbf{Taylor coefficients} $f_n: W^{\otimes n} \to V$.
\begin{lem}
	Let $f,g$ be composable coalgebra maps between three reduced tensor coalgebras. Then the Taylor coefficients of the composition $f\circ g$ are given by $$(g\circ f)_n = \sum_{i_1+\ldots+ i_r=n}g_r(f_{i_1}\otimes\cdots\otimes f_{i_r})$$
\end{lem}
Note the similarity with composition of $A_\infty$-algebra maps.
\begin{defn}Let $C$ be a dg-coalgebra. A \textbf{coderivation of degree $p$ on $C$} is a linear degree $p$ endomorphism $\delta$ of $C$ satisfying $(\delta\otimes 1 + 1 \otimes \delta)\circ \Delta = \Delta \circ \delta$.
\end{defn} 
The graded space $\mathrm{Coder}(C)$ of all coderivations of $C$ is not closed under composition, but is closed under the commutator bracket. Say that $\delta \in \mathrm{Coder}^1(C)$ is a \textbf{differential} if $\delta^2=0$; in this case $\mathrm{ad}(\delta)$ is a differential on $\mathrm{Coder}(C)$, making $\mathrm{Coder}(C)$ into a dgla. In the special case that $C=\bar{T}^c(V)$, a coderivation is determined by its Taylor coefficients. Coderivations compose similarly to coalgebra morphisms:
\begin{lem}
	Let $\delta, \delta'$ be coderivations on $\bar{T}^c(V)$. Then the Taylor coefficients of the composition $\delta\circ \delta'$ are given by $$(\delta\circ \delta')_n = \sum_{r+s+t=n}\delta_{r+1+t}(1^{\otimes r} \otimes \delta'_s \otimes 1^{\otimes t})$$
\end{lem}
\begin{thm}
	An $A_\infty$-algebra structure on a graded vector space $A$ is the same thing as a differential $\delta$ on $\bar{T}^c(A[1])$.
\end{thm}
\begin{proof}We provide a sketch. Given a coderivation $\delta$ we obtain Taylor coefficients $\delta_n:A[1]^{\otimes n} \to A$ of degree 1; in other words, these are maps $m_n: A^{\otimes n} \to A$ of degree $2-n$. The Stasheff identities are equivalent to $\delta$ being a differential. The sign changes occur in the Stasheff identities because of the need to move elements past the formal suspension symbol $[1]$.
\end{proof}
The following proposition can be checked in a similar manner:
\begin{prop}
	Let $A,A'$ be two $A_\infty$-algebras with associated differentials $\delta, \delta'$. Then an $A_\infty$-morphism $f:A \to A'$ is the same thing as a coalgebra morphism $\bar{T}^c(A[1]) \to \bar{T}^c(A'[1])$ commuting with the coderivations.
\end{prop}
\begin{defn}
	Let $A,A'$ be $A_\infty$-algebras and $f,g$ a pair of maps $A \to A'$. Let $F,G$ be the associated maps $\bar{T}^c(A[1]) \to \bar{T}^c(A'[1])$. Say that $f$ and $g$ are \textbf{homotopic} if there's a map $H: \bar{T}^c(A[1]) \to \bar{T}^c(A'[1])$ of degree $-1$ with $\Delta H = F \otimes H + H\otimes G$ and $F-G=\partial H$, where $\partial$ is the differential in the $\hom$-complex.
\end{defn}
One can unwind this definition into a set of identities on the Taylor coefficients of $H$; this is done in \cite[1.2]{lefevre}. Say that $A,A'$ are \textbf{homotopy equivalent} if there are maps $f:A \to A'$ and $f':A' \to A$ satisfying $f'f\simeq \id_{A}$ and $ff'\simeq \id_{A'}$.
\begin{thm}[\cite{proute}]
	Homotopy equivalence is an equivalence relation on the category $\cat{Alg}_\infty$ of $A_\infty$-algebras. Moreover, two $A_\infty$-algebras are homotopy equivalent if and only if they're quasi-isomorphic.
\end{thm}
The category $\cat{dga}$ of differential graded algebras sits inside the category $\cat{Alg}_\infty$. It's not a full subcategory: there may be more $A_\infty$-algebra maps than dga maps between two dgas. However, two dgas are dga quasi-isomorphic if and only if they're $A_\infty$-quasi-isomorphic: this is shown in, for example, \cite[1.3.1.3]{lefevre}. Abstractly, this follows from the existence of model structures on both $\cat{dga}$ and $\cat{cndgc}$, the category of conilpotent dg-coalgebras, for which the bar and cobar constructions are Quillen equivalences (cf. \ref{ainfkd} for the bar and cobar constructions).

\p Including $\cat{dga}\into\cat{Alg}_\infty$ does not create more quasi-isomorphism classes. Indeed, every $A_\infty$-algebra is quasi-isomorphic to a dga: one can take the adjunction quasi-isomorphism $A \to \Omega B_\infty A$ induced by the bar and cobar constructions. However, we do get new descriptions of quasi-isomorphism class representatives. One nice such representative is the \textbf{minimal model} of an $A_\infty$-algebra.
\subsection{Minimal models}\label{minmods}
An $A_\infty$-algebra is \textbf{minimal} if $m_1=0$. Every $A_\infty$-algebra admits a minimal model. More precisely:
\begin{thm}[Kadeishvili \cite{kadeishvili}]\label{kadeish}
	Let $(A,m_1,m_2,\ldots)$ be an $A_\infty$-algebra, and let $HA$ be its cohomology ring. Then there exists the structure of an $A_\infty$-algebra ${\mathscr{H}\kern -1pt A}= (HA,0,[m_2],p_3,p_4,\ldots)$ on $HA$, unique up to $A_\infty$-isomorphism, and an $A_\infty$-algebra morphism ${\mathscr{H}\kern -1pt A} \to A$ lifting the identity of $HA$.
\end{thm}
\begin{rmk}
	While the multiplication on $HA$ is induced by $m_2$, we need not have $p_n=[m_n]$ for $n>2$; indeed the $m_n$ need not even be cocycles. For example, if $A$ is a non-formal dga, then ${\mathscr{H}\kern -1pt A}$ must have nontrivial higher multiplications. We also note that ${\mathscr{H}\kern -1pt A} \to A$ is clearly an $A_\infty$-quasi-isomorphism, since it lifts the identity on $HA$. We also remark that the theorem follows from the essentially equivalent \textbf{homotopy transfer theorem}: if $A$ is an $A_\infty$-algebra, and $V$ a homotopy retract of $A$, then $V$ admits the structure of an $A_\infty$-algebra making the retract into an $A_\infty$-quasi-isomorphism (see \cite{lodayvallette} 9.4 for details). The result follows since, over a field, the cohomology of any chain complex is always a homotopy retract as one can choose splittings.
\end{rmk}
It's possible to give a constructive proof of Kadeishvili's theorem: Merkulov did this in \cite{merkulov}. One can define the $p_n$ recursively: suppose for convenience that $A$ is a dga. Choose any section $\sigma:HA \to A$ and let $\pi:A \to HA$ be the projection to $HA$. We'll identify $HA$ with its image under $\sigma$. Choose a homotopy $h: \id_A \to \sigma\pi$. Define recursively maps $\lambda_n: (HA)^{\otimes n} \to A$ by $\lambda_2=m_2$, and $$\lambda_n\coloneqq \sum_{s+t=n}(-1)^{s+1}\lambda_2(h\lambda_s\otimes h\lambda_t)$$where we formally interpret $h\lambda_1\coloneqq -\id_A$. Then, $p_n=\pi\circ\lambda_n$. See \cite{markl} for some very explicit formulas (whose sign conventions differ).

\begin{defn}
	Let $G$ be an abelian group. An $A_\infty$-algebra $A$ is \textbf{Adams }$G-$\textbf{graded} or just \textbf{Adams graded} if it admits a secondary grading by $G$ such that each higher multiplication map $m_n$ is of degree $(2-n,0)$.
\end{defn}
If an $A_\infty$-algebra is Adams graded, then by making appropriate choices one can upgrade Merkulov's construction to give an $A_\infty$-quasi-isomorphism of Adams graded algebras $A \to \mathscr{H}\kern -1pt A$. Moreover, if $A$ is strictly unital, one can choose the morphism to be strictly unital. See Section 2 of \cite{lpwzext} for more details.

\p One can sometimes compute $A_\infty$-operations on a dga by means of Massey products. In what follows, $\tilde{a}$ means $(-1)^{1+|a|}{a}$, using the same sign conventions as \cite{kraines}.
\begin{defn}\label{masseys}
	Let $u_1,\ldots, u_r$ be cohomology classes of a dga $A$. Pick representatives $u_i=[a_{i\,i}]$. The \textbf{$r$-fold Massey product} $\langle u_1,\ldots, u_r \rangle$ of the cohomology classes $u_1,\ldots ,u_r$ is defined to be the set of cohomology classes of sums $\tilde{a}_{1\,1}a_{2\,r}+\cdots+\tilde{a}_{1\,r-1}a_{r\,r}$ such that $da_{i\,j}=\tilde{a}_{i\,i}a_{i+1\,j}+\cdots+\tilde{a}_{i\,j-1}a_{j\,j}$ for all $1\leq i \leq j \leq n$ with $(i,j)\neq(1,n)$. This operation is well-defined, in the sense that it depends only on the cohomology classes $u_1,\ldots, u_r$.
\end{defn}
We'll abuse terminology by referring to elements of $\langle x_1,\ldots, x_r \rangle$ as Massey products. We may also decorate the product $\langle x_1,\ldots, x_r \rangle$ with a subscript $\langle x_1,\ldots, x_r \rangle_r$ to emphasise that it is an $r$-fold product.
\begin{rmk}
	We remark that $\langle x_1,\ldots, x_r \rangle$ may be empty: for example, in order for $\langle x,y,z\rangle$ to be nonempty, we must have $xy=yz=0$. More generally, for $\langle x_1,\ldots, x_r \rangle$ to be nonempty, we require that each $\langle x_p,\ldots, x_q \rangle$ is nonempty for $0<q-p<n-1$. Most sources define $\langle x,y,z\rangle$ only when it's nonempty, and leave it undefined otherwise.
\end{rmk}
The point is that, when Massey products exist, Merkulov's higher multiplications $p_n$ are all Massey products, up to sign.
\begin{thm}[{\cite[3.1]{lpwzext}}]
	Let $A$ be a dga and let $x_1,\ldots, x_r$ $(r>2)$ be cohomology classes in $HA$, and suppose that  $\langle x_1,\ldots, x_r \rangle$ is nonempty. Give $HA$ an $A_\infty$-algebra structure via Merkulov's construction. Then, up to sign, the higher multiplication $p_r(x_1,\ldots, x_r)$ is a Massey product.
\end{thm}
So, if $A$ is a formal dga, then all Massey products (that exist) will vanish. The converse is not true: formality of a dga cannot be checked simply by looking at its Massey products. We'll use the existence of Massey products to detect non-formality: the following lemma is computationally useful.
\begin{lem}\label{masseylemma}
	Let $x_{1\,1}=x_{2\,2}=\cdots=x_{r\,r}$. Then, setting $b_i\coloneqq x_{1\,i}$, we may compute the $r$-fold Massey product $\langle [x_{1\,1}],\ldots, [x_{r\,r}] \rangle$ as the set of cohomology classes of sums $\tilde{b}_1b_{r-1}+\cdots+\tilde{b}_{r-1}b_1$ such that $db_i=\tilde{b}_1b_{i-1}+\cdots+\tilde{b}_{i-1}b_1$.
\end{lem}
\begin{proof}
	Inductively, it is easy to see that if $i-j=k-l$ then $x_{i\,j}=x_{k\,l}$. Hence we may write $x_{i\,j}=b_{1+j-i}$ and the result follows. Note that if $x_{1\,1}$ is of odd degree then we may drop the tildes from the $b_i$.
\end{proof}
\subsection{Koszul duality}\label{ainfkd}
\begin{defn}
	A strictly unital $A_\infty$-algebra $A$ is \textbf{augmented} if there's a morphism of strictly unital $A_\infty$-algebras $A \to k$. In this case the \textbf{augmentation ideal} is $\bar{A}\coloneqq \ker{A \to k}$.
\end{defn}
\begin{defn}
	Let $A$ be an augmented $A_\infty$-algebra. The \textbf{bar construction} on $A$ is the coaugmented conilpotent dg-coalgebra $B_\infty A\coloneqq T^{c}(\bar{A}[1])$, the tensor coalgebra on the shifted augmentation ideal. The coproduct is the deconcatenation coproduct and the differential is the natural generalisation of the bar differential; see \cite{lefevre} for a concrete formula. If $A$ is a dga then $B_\infty A$ is the usual bar construction.
\end{defn}
\begin{lem}
	The bar construction preserves quasi-isomorphisms.
\end{lem}
\begin{proof}
The idea is to filter $BA$ by setting $F_p BA$ to be the elements of the form $a_1\otimes\cdots\otimes a_n$ with $n\leq p$, and look at the associated spectral sequence. A proof for dgas is in \cite{lodayvallette}, Chapter 2 and a proof for $A_\infty$-algebras is in \cite{lefevre}, Chapter 1.
\end{proof}
\begin{defn}
	If $A$ is an augmented $A_\infty$-algebra then the \textbf{Koszul dual} is $A^!\coloneqq (BA)^*$, the linear dual of the bar construction. It's a semifree dga, in the sense that the underlying graded algebra is a completed free algebra. In the situations we will be interested in, the underlying graded algebra of $A^!$ will actually be a free algebra in the usual sense.
\end{defn}
Loosely, the differential $d(x^*)$ is the signed sum of the products $x_1^*\cdots x_r^*$ such that \linebreak $d(x_1 | \cdots | x_r)=x$, where $d$ is the $A_\infty$ bar differential. Since taking the linear dual is exact, the Koszul dual functor sends $A_\infty$-quasi-isomorphisms $A \to A'$ to dga quasi-isomorphisms $(A')^! \to A^!$. We often implicitly use the following result:
\begin{thm}[{\cite[2.7.8]{dqdefm}}]\label{dqkd}
	Let $A$ be a nonpositive dga with each $H^i(A)$ finite-dimensional and $H^0(A)$ a local algebra. Then $A$ is quasi-isomorphic to its Koszul double dual $A^{!!}$.
\end{thm}

\vfill\pagebreak

\begin{footnotesize}
	\bibliographystyle{alpha}
	\bibliography{dcabib}
\end{footnotesize}
\textsc{School of Mathematics, University of Edinburgh, James Clerk Maxwell Building, Peter Guthrie Tait Road, Edinburgh EH9 3FD, United Kingdom}\par\nopagebreak\texttt{matt.booth@ed.ac.uk}\hfill\url{https://www.maths.ed.ac.uk/~mbooth/}

\end{document}